\documentclass[oneside,11pt]{amsart}
\usepackage[pagebackref,breaklinks,unicode]{hyperref}
\hypersetup{pdftitle = Metrics from walls}
\title[Constructing metric spaces from systems of walls]{Constructing metric spaces from systems of walls}

\usepackage[margin=2.99cm]{geometry}

\author{Harry Petyt}
\address{Mathematical Institute, University of Oxford, UK}
\email{petyt@maths.ox.ac.uk}

\author{Abdul Zalloum}
\address{Department of Mathematics and Statistics, Queen's University, Canada}
\email{az32@queensu.ca}

\address{Mathematical Institute, University of Oxford, UK}
\email{spriano@maths.ox.ac.uk}

\usepackage{amsmath,amssymb,amsthm, mathrsfs, mathtools}
\usepackage{enumitem} \setlist{nosep} 
\usepackage[utf8]{inputenc}
\usepackage[english]{babel}
\usepackage[pagebackref,breaklinks,unicode]{hyperref}
\usepackage{float}
\usepackage{comment}
\usepackage{csquotes}
\usepackage[font=small,justification=centering]{caption}
\usepackage{subcaption}
\usepackage[dvipsnames]{xcolor}
\usepackage{mathabx} 
\usepackage{xfrac} 
\usepackage[normalem]{ulem} 
\usepackage{dirtytalk}
\usepackage{cancel}
\usepackage{marginnote}

\newtheorem{theorem}{Theorem}[section]
\newtheorem{proposition}[theorem]{Proposition}
\newtheorem{corollary}[theorem]{Corollary}
\newtheorem{lemma}[theorem]{Lemma}

\newtheorem{introthm}{Theorem}
\newtheorem{introcor}[introthm]{Corollary}

\theoremstyle{definition}
\newtheorem{definition}[theorem]{Definition}
\newtheorem{example}[theorem]{Example}

\newtheorem{remark}[theorem]{Remark}

\newcounter{shcount}
\newcommand*{\bsh}[1]{\theoremstyle{definition}\newtheorem{subhead\theshcount}[theorem]{#1}
    \begin{subhead\theshcount}} 
\newcommand*{\esh}{\end{subhead\theshcount}\stepcounter{shcount}} 
\newcommand*{\ubsh}[1]{\theoremstyle{definition}\newtheorem*{subhead\theshcount}{#1}
    \begin{subhead\theshcount}} 
\newcommand*{\uesh}{\end{subhead\theshcount}\stepcounter{shcount}} 

\newcounter{claimcount}

\newenvironment{claim*}[1]{\par\vspace{2mm}\noindent
    \underline{Claim:}\hspace{2mm}#1}{}

\newcounter{enumlabelcount}
\makeatletter\newcommand\enumlabel[1][]{\item[#1]
    \refstepcounter{enumlabelcount}\def\@currentlabel{#1}}\makeatother

\renewcommand*{\backrefalt}[4]{\ifcase #1 (Not cited).\or (Cited p.~#2).\else (Cited pp.~#2).\fi} 

\makeatletter\def\subsection{\@startsection{subsection}{1}\z@{.7\linespacing\@plus\linespacing}
    {.5\linespacing}{\normalfont\scshape\centering}}\makeatother 
\makeatletter\def\part{\@startsection{part}{1}\z@{.7\linespacing\@plus\linespacing}
    {.5\linespacing}{\normalfont\large\scshape\centering}}\makeatother 

\DeclareMathOperator{\cur}{Cur}
\DeclareMathOperator{\dist}{\mathsf{d}}
\DeclareMathOperator{\Dist}{\mathsf{D}}
\DeclareMathOperator{\diam}{diam}
\DeclareMathOperator{\hull}{Hull}
\DeclareMathOperator{\isom}{Isom}
\DeclareMathOperator{\MCG}{MCG}
\DeclareMathOperator{\Out}{Out}
\DeclareMathOperator{\QM}{QM}
\DeclareMathOperator{\sign}{sign}

\newcommand*{\C}{\mathcal C}
\newcommand*{\g}{\mathfrak{g}}
\newcommand*{\hh}{\mathfrak{h}}
\newcommand*{\kk}{\mathfrak{k}}

\newcommand*{\R}{\mathbb R}

\newcommand*{\Z}{\mathbb Z}
\newcommand{\eps}{\varepsilon}

\newcommand*{\cal}{\mathcal}

\newcommand*{\ssm}{\smallsetminus}

\definecolor{harrycomment}{rgb}{0.6,0,0.4}

\begin{document}

\maketitle
\centerline{\textit{With an appendix with Davide Spriano}}

\begin{abstract}
We give a general procedure for constructing metric spaces from systems of partitions. This generalises and provides analogues of Sageev's construction of dual CAT(0) cube complexes for the settings of hyperbolic and injective metric spaces.

As applications, we produce a ``universal'' hyperbolic action for groups with strongly contracting elements, and show that many groups with ``coarsely cubical'' features admit geometric actions on injective metric spaces. In an appendix with Davide Spriano, we show that a large class of groups have an infinite-dimensional space of quasimorphisms.
\end{abstract}

{\hypersetup{hidelinks}\setcounter{tocdepth}{1}\tableofcontents\setcounter{tocdepth}{2}}

\section{Introduction}

A guiding principle of geometric group theory is the idea that one can learn about a group from its actions on nice metric spaces. There are two aspects to implementing this: producing the actions, and studying the properties of groups that possess them. In this article we are interested in the former.

Approaches to producing spaces for groups to act on take many forms, such as local-to-global results \cite{alexanderbishop:hadamard,leary:metric,haettel:link,bowditch:cartan--hadamard} and combination theorems \cite{bestvinafeighn:combination,dahmani:combination,mjreeves:combination,hsuwise:cubulating}. However, amongst such techniques, the stand-out examples are those with inter-categorical features, such as: hyperbolisation procedures \cite{gromov:hyperbolic,davisjanuszkiewicz:hyperbolization,charneydavis:strict,ontaneda:riemannian}, which convert simplicial complexes into aspherical manifolds of nonpositive curvature; Sageev's construction \cite{sageev:ends}, which outputs a CAT(0) cube complex from a system of partitions; and the projection complex machinery \cite{bestvinabrombergfujiwara:constructing,bestvinabrombergfujiwarasisto:acylindrical}, which can (amongst many other things) produce actions on quasitrees from simple geometric conditions.

We shall work in the same vein as these latter two, and provide a general framework that one can use for producing group actions on different types of metric spaces. The framework takes the same form of input as Sageev's construction, but we allow additional flexibility that allows for a much wider variety of metrics to be produced, and under weaker conditions. That is, we start with a set $S$ together with a collection $P$ of bipartitions. This pair has a naturally-defined dual median algebra associated with it, and by considering varying subsets of $2^P$ we obtain a range of metrics on this dual. Sageev's construction is recovered as a degenerate case. The principal idea of this article can therefore be phrased as follows. See Section~\ref{sec:construction} (and also  Definition~\ref{introdef:dualisable_system} later in the introduction) for more precise formulations of some initial statements.

\bsh{Construction} \label{constr:main}
Let $S$ be a set with a collection $P$ of bipartitions. By choosing an appropriate subset $\C\subset2^P$, one obtains a metric space $X$ and a natural map $S\to X$.
\esh

To obtain an action on $X$ from an action on $(S,P)$, one merely has to insist that $\C$ is preserved. The three main types of metric spaces that we systematically produce in this way are Helly graphs, coarsely injective spaces, and hyperbolic spaces. However, the set-up is quite general and can likely be used to provide actions on other interesting metric spaces.

Let us mention that Sageev's construction has also been generalised by Chatterji--Dru\c{t}u--Haglund, with improvements by Fioravanti \cite{chatterjidrutuhaglund:kazhdan,fioravanti:roller}. The two generalisations are of rather different flavours: the one in \cite{chatterjidrutuhaglund:kazhdan,fioravanti:roller} can be thought of as a ``continuous version'' of Sageev's construction, and it firmly belongs in the \emph{median} category; whereas our generalisation is still fundamentally discrete, but is not restricted to $\ell^1$ geometry.

\subsection{Hyperbolic duals}

Group actions on hyperbolic spaces are a major theme of research in the field, especially in the wake of the result of Masur--Minsky that the curve graph of a surface is hyperbolic \cite{masurminsky:geometry:1}. Indeed, this spawned the theory of acylindrical hyperbolicity \cite{sela:acylindrical,bowditch:tight,dahmaniguirardelosin:hyperbolically,osin:acylindrically}. In this context, the aforementioned projection complex machinery is a potent tool, because it allows one to produce actions on quasitrees such that a given \emph{strongly contracting} (i.e. ``negatively-curved'') group element has a so-called \emph{WPD} action \cite{bestvinafujiwara:bounded}, which is sufficient to establish acylindrical hyperbolicity \cite{dahmaniguirardelosin:hyperbolically}. Strongly contracting elements have been studied by many authors, including \cite{arzhantsevacashentao:growth,gekhtmanyang:counting,algomkfir:strongly,papin:strongly,coulon:ergodicity}.

One unsatisfactory aspect of the actions on quasitrees coming from projection complexes is that they are rather unnatural, in that they have little to do with the original space. The procedure also has the limitation of only being able to handle a finite number of strongly contracting elements at once. By building an appropriate collection of partitions, we are able to apply Construction~\ref{constr:main} to produce, for any geodesic metric space $S$, a hyperbolic space that is ``universal'', in the sense that all strongly contracting geodesics are witnessed there. We call it the \emph{hyperbolic core} of $S$. (See Theorem~\ref{thm:universal_contracting_characterisation} for the precise statement.) 

\begin{introthm} \label{introthm:universal}
Let $S$ be a geodesic metric space. Its hyperbolic core $X$ is a hyperbolic space with an $\isom S$--equivariant coarsely Lipschitz map $\pi:S\to X$ such that a geodesic $\gamma\subset S$ is strongly contracting if and only if $\pi\gamma$ is a quasigeodesic. Moreover, the correspondence is quantitative. 
\end{introthm}

This can be compared to (and, via \cite{sistozalloum:morse}, be viewed as greatly extending) the situation for pseudo-Anosov mapping classes acting on the curve graph \cite{masurminsky:geometry:1}; outer space and the free factor complex \cite{bestvinafeighn:hyperbolicity,dowdalltaylor:contracting}; and rank-one isometries acting on the curtain model of a CAT(0) space \cite{petytsprianozalloum:hyperbolic}. 

The partitions that are used to construct the  hyperbolic core are obtained using closest-point projection to strongly contracting geodesics of $S$, similarly to the construction of the curtain model in \cite{petytsprianozalloum:hyperbolic}, but the choice of $\C\subset2^P$ is a little different to the notion of \emph{$L$--separation} used there; see Definition~\ref{def:ball_sep}.

The statement of Theorem~\ref{introthm:universal} is purely about the metric space $S$ and its isometries. Applying it to groups acting properly on $S$, we can obtain a simultaneous WPD action of \emph{all} strongly contracting elements; see Corollary~\ref{cor:equivariant_characterisation} and Proposition~\ref{prop:wpd}.

\begin{introthm} \label{introthm:WPD}
Let $S$ be a geodesic metric space. For any group $G$ acting properly on $S$, elements of $G$ are strongly contracting exactly when they act loxodromically on the hyperbolic core of $S$, and every such element is WPD.
\end{introthm}

Furthermore, under the additional assumption that $G$ acts properly coboundedly on a metric space where all \emph{Morse} geodesics are strongly contracting, we show that the hyperbolic core of that space is a \emph{universal recognising space} for stable subgroups of $G$, in the sense of \cite{balasubramanyachesserkerrmangahastrin:non}. That is, every stable subgroup of $G$ has quasiisometrically embedded orbits in $X$. The following is Theorem~\ref{thm:universal_recognition} in the text. 

\begin{introthm} \label{introthm:stable}
Suppose that a group $G$ acts properly coboundedly on a metric space $S$ where all Morse geodesics are (quantitatively) strongly contracting. A finitely generated subgroup of $G$ is stable if and only if its orbits in the hyperbolic core of $S$ are quasiisometric embeddings.
\end{introthm}

(Here ``quantitatively'' means that the contracting constant depends only on the Morse gauge, not the particular Morse geodesic.) The list of spaces covered by Theorem~\ref{introthm:stable} includes CAT(0) spaces \cite{charneysultan:contracting}, injective spaces \cite{sistozalloum:morse}, Garside groups \cite{calvezwiest:morse}, certain small-cancellation groups \cite{zbinden:small}, and weakly modular graphs with convex balls \cite{soergelzalloum:morse}. It therefore generalises known results for mapping class groups \cite{kentleininger:shadows,durhamtaylor:convex}, hierarchically hyperbolic groups more generally \cite{abbottbehrstockdurham:largest,haettelhodapetyt:coarse,sistozalloum:morse}, and CAT(0) groups \cite{petytsprianozalloum:hyperbolic}.

\medskip

Another application of the hyperbolic core, and of the particulars of its construction, is investigated in the appendix, together with Davide Spriano. A \emph{quasimorphism} of a group $G$ is a map $f:G\to\R$ that fails to be a homomorphism by a bounded amount, in the sense that there is some $r$ such that $|f(gh)-f(g)-f(h)|\le r$ for all $g,h\in G$. There are two ``trivial'' classes of quasimorphisms, namely bounded functions and homomorphisms. The quotient of the space of quasimorphisms by these trivial classes is denoted $\widetilde{\QM}(G)$; it is intimately related with the second (real) cohomology of $G$.

For the most part, calculations of $\widetilde{\QM}(G)$ (when it is nontrivial) rely on $G$ acting with some form of properness on some nice metric space (a notable exception where properness is not needed is \cite{capracefujiwara:rank}). The main result of the appendix is the following, which has no properness assumption. See Theorem~\ref{thm:quasimorphisms}.

\begin{introthm} \label{introthm:quasimorphisms}
Let $S$ be a non-hyperbolic geodesic space where all Morse geodesics are (quantitatively) strongly contracting. If $S$ contains a biinfinite Morse geodesic, then $\widetilde{\QM}(G)$ is infinite-dimensional for every group $G$ acting coboundedly on $S$.
\end{introthm}

Note that non-hyperbolicity is essential - the conclusion of Theorem~\ref{introthm:quasimorphisms} can fail for automorphism groups of trees \cite{iozzipagliantinisisto:characterising}.

\ubsh{On Zbinden's work}
Results similar to Theorems~\ref{introthm:universal} and~\ref{introthm:WPD} are also proved in simultaneous work of Stefanie Zbinden \cite{zbinden:hyperbolic}. More precisely, Zbinden has a construction, rather different to the one employed in this article and more tailored to the setting of strongly contracting geodesics, that can produce, for any geodesic space $S$, a hyperbolic space $Z$ with the following features. 
\begin{itemize}
\item   A geodesic in $S$ is strongly contracting if and only if it projects to a quasigeodesic of~$Z$.
\item   Any group $G$ acting properly on $S$ acts \emph{non-uniformly acylindrically} on $Z$, and strongly contracting elements are loxodromic.
\item   If $S$ has the property that all Morse geodesics are strongly contracting, and if $G$ acts properly coboundedly on $S$, then all generalised loxodromics of $G$ are loxodromic and WPD on $Z$.
\end{itemize}
\uesh

\subsection{Injective duals}

Construction~\ref{constr:main} can also systematically produce Helly graphs and coarsely injective metric spaces. These are recent additions to the arsenal of geometric group theory that have proved to be rather powerful \cite{huangosajda:helly,haettel:lattices,haettelhodapetyt:coarse,haettel:group,zalloum:injectivity}. A geodesic metric space is (coarsely) injective if its metric balls satisfy a Helly property: any balls that pairwise intersect have nonempty total (coarse) intersection.

The theory appears to have much in common with that of CAT(0) cube complexes and CAT(0) spaces, but aside from the fact that every CAT(0) cube complex is injective when equipped with the $\ell^\infty$ metric \cite{maitang:oninjective,miesch:injective}, this is largely heuristic at this stage. In the cubical setting, it is Sageev's construction that is really the foundation on which many (often highly nontrivial) cubulation results are built, such as in \cite{bergeronwise:boundary,hsuwise:cubulating,hagenwise:cubulating:general}. 

This motivates the following, a precise version of which is stated as Theorem~\ref{thm:coarsely_injective}; also see Theorem~\ref{introthm:dual_spaces} below for more information.

\begin{introthm} \label{introthm:injective}
Under simple conditions on $\C$, the metric space $X$ produced by Construction~\ref{constr:main} is coarsely injective.
\end{introthm}

It should be emphasised that the assumptions in Theorem~\ref{introthm:injective} are weaker than those of Sageev's construction.

As a first implementation, we consider a class of groups that naturally generalises hyperbolic spaces and CAT(0) cube complexes. In $\delta$--hyperbolic spaces, Gromov's tree approximation lemma states that the quasiconvex hull of every finite set $A$ is quasiisometric to a tree, where the quasiisometry constant depends only on $\delta$ and the cardinality of $A$ \cite{gromov:hyperbolic}. In view of results for mapping class groups and Teichmüller space \cite{behrstockminsky:centroids,eskinmasurrafi:large,durham:augmented}, the fact that CAT(0) cube complexes can be considered as higher-dimensional versions of trees led Bowditch to introduce the class of \emph{coarse median} spaces \cite{bowditch:coarse}, which are defined by a \emph{cubical} approximation property. 

There are multiple viable ways to formulate such a cubical approximation property, and Bowditch opts for perhaps the most general possibility, which contains precious little metric information. In spite of this, he is able to obtain some remarkably strong results; see especially \cite{bowditch:quasiflats}. To our knowledge, though, all coarse median spaces of prior interest satisfy a seemingly much stronger property that more faithfully generalises Gromov's lemma \cite{behrstockhagensisto:quasiflats,bowditch:convex}. Moreover, this stronger property is sometimes very useful \cite{haettelhodapetyt:coarse,durhamminskysisto:stable,durhamzalloum:geometry}.

We therefore make a small tweak to the definition of coarse median spaces, and introduce the class of \emph{strong coarse median} spaces (these have also been called \emph{locally quasicubical} spaces in \cite{durhamzalloum:geometry}, based on a draft version of the present work). Essentially, a strong coarse median space is a geodesic space where every finite set $A$ has an appropriate ``coarsely convex'' hull that is \emph{median-preservingly} quasiisometric to a CAT(0) cube complex, with the constants depending only on the cardinality of $A$. A \emph{strong-coarse-median group} is then a finitely generated group possessing an equivariant strong coarse median structure. Aside from hyperbolic and cubical groups, the class includes toral relatively hyperbolic groups, mapping class groups, extra-large--type Artin groups, and hierarchically hyperbolic groups more generally \cite{behrstockhagensisto:hierarchically:2,behrstockhagensisto:quasiflats,hagenmartinsisto:extra}. Note that some cubical groups do not \emph{naturally} admit hierarchically hyperbolic structures \cite{shepherd:cubulation}.

Our main result about strong-coarse median groups is the construction of good sets $P$ and $\C$, leading to the following, which generalises and gives a hierarchy-free proof of the main result of \cite{haettelhodapetyt:coarse}; it appears as Corollary~\ref{cor:scm_group_injective}. 

\begin{introthm} \label{introthm:scm_inj}
For any strong coarse median space $S$, Theorem~\ref{introthm:injective} produces a coarsely injective space $X_\C(S)$ such that the natural map $S\to X_\C(S)$ is a quasiisometry. In particular, if $S$ is a strong-coarse-median group, then $S$ acts properly coboundedly on an injective space.
\end{introthm}

Since $\C$ is constructed from metric and median data, every median-preserving isometry of $S$ induces an isometry of $X_\C(S)$. As a special case of this, known results about groups acting properly coboundedly on injective spaces lead to the following, which notably does not need the quasiisometry to be equivariant \cite{lang:injective,keppelermollervarghese:automatic,alonsobridson:semihyperbolic} (see also \cite{haettel:group}).

\begin{introcor} \label{introcor:quasicube}
Suppose that $G$ is quasiisometric to a CAT(0) cube complex $Q$. If the induced quasiaction of $G$ on $Q$ is coarsely median-preserving, then:
\begin{itemize}
\item   $G$ has finitely many conjugacy classes of finite subgroups;
\item   $\mathbb Q$ is not a subgroup of $G$;
\item   polycyclic subgroups of $G$ are virtually abelian and undistorted;
\item   $G$ is semihyperbolic.
\end{itemize}
\end{introcor}

Already this corollary applies to mapping class groups, and many hierarchically hyperbolic groups more generally \cite{petyt:mapping}, overlapping with \cite{haettelhodapetyt:coarse}. In particular, the statement about polycyclic subgroups provides a large generalisation of \cite[Thm~1.2]{farblubotzkyminsky:rank}. In view of the fact that not all cubical groups are naturally hierarchically hyperbolic, it seems likely that there are non-cubical groups that are covered by Corollary~\ref{introcor:quasicube} but not \cite{haettelhodapetyt:coarse}. 

\medskip

Returning to the generality of Theorem~\ref{introthm:injective}, which outputs a coarsely injective space $X$, we build a natural family of paths in $X$, which we call \emph{normal wall paths}. These paths simultaneously behave well with respect to the metric on $X$ and with respect to its structure as a \emph{median algebra} (Propositions~\ref{prop:nwp_rough_geodesic},~\ref{prop:bicombing}). Degenerating to the setting of CAT(0) cube complexes exactly gives the Niblo--Reeves \emph{normal cube paths} \cite{nibloreeves:geometry}. 

One simple use for normal wall paths is that they provide a nice melding of results from \cite{durhamminskysisto:stable} and \cite{haettelhodapetyt:coarse}. The main result of each of those articles is semihyperbolicity of mapping class groups, but this is achieved in different ways. Semihyperbolicity asks for an equivariant set of paths with good fellow-travelling properties. In \cite{haettelhodapetyt:coarse}, these come from Lang's bicombing of injective spaces \cite{lang:injective}, whereas \cite{durhamminskysisto:stable} contains a direct construction that yields paths compatible with the Masur--Minsky hierarchy structure \cite{masurminsky:geometry:2}. It is not clear that these systems of paths are related to each other: the former may not be hierarchical, and the latter may be unnatural to the injective space. The following shows that normal wall paths satisfy both. (See Propositions~\ref{prop:nwp_rough_geodesic} and~\ref{prop:bicombing}, as well as Lemma~\ref{lem:scm_qmqie}.)

\begin{introthm} \label{introthm:mcg_inj} 
Let $S$ be either the mapping class group $\MCG\Sigma$ of a finite-type surface $\Sigma$, or the corresponding Teichmüller space with the Teichmüller metric. Normal wall paths in the dual space $X_\C(S)$ are median paths, rough geodesics, and yield hierarchy paths of $S$: their images under subsurface projections are unparametrised quasigeodesics. Furthermore, they form a $\MCG\Sigma$--invariant bicombing.
\end{introthm}

The statement of Theorem \ref{introthm:mcg_inj} applies to all strong-coarse-median groups, whereas \cite{durhamminskysisto:stable} and \cite{haettelhodapetyt:coarse} cover (colourable) hierarchically hyperbolic groups. In that generality, normal wall paths witness semihyperbolicity. For the Teichmüller metric, a system of quasigeodesics with similar properties to those of Theorem~\ref{introthm:mcg_inj} has been announced by Kapovich--Rafi. 

The interplay between the median and the metric on $X_\C(S)$ has consequences for the geometry of its balls. For instance, following Cannon \cite{cannon:almost}, we say that a metric space $M$ is \emph{almost convex} if its spheres have the following property: for every $k$ there exists $N(k)$ such that for all $r\ge0$ and $m\in M$, if $x,y\in S_r(m)\subset M$ have $\dist(x,y)=k$, then $x$ and $y$ can be joined by a path in the ball $B_r(m)$ of length at most $N(k)$. The following shows that this holds for $X_\C(S)$ in a strong way. (See Lemma~\ref{lem:convex_balls} and Proposition~\ref{prop:nwp_rough_geodesic}.)

\begin{introthm}
If $S$ is a strong coarse median space, then normal wall paths make $X_\C(S)$ almost convex, with $N(k)=k+6$.
\end{introthm}

In \cite[Qns~3.4,~3.5]{farb:some}, Farb asks whether there exist almost convex Cayley graphs of mapping class groups, and whether the Teichmüller metric is almost convex. Although the above theorem does not directly answer Farb's questions, it shows that almost convexity holds in the space $X_\C(S)$, which is equivariantly quasiisometric to the desired spaces, and is witnessed by hierarchy paths.

Let us mention one more application of Construction~\ref{constr:main}, which both makes use of normal wall paths and brings us back to the setting of hyperbolic duals. By using the same set of walls as is used to establish Theorem~\ref{introthm:scm_inj} but a set $\C$ more along the lines of Theorem~\ref{introthm:universal}, one can, for any strong coarse median space, produce a hyperbolic space $Y(S)$ associated with it. In the hierarchically hyperbolic setting, a proof involving normal wall paths shows that $Y(S)$ coarsely recovers the largest hyperbolic space of \cite{abbottbehrstockdurham:largest}. The following summarises Section~\ref{subsec:scm_hyp}.

\begin{introthm}
 For every strong coarse median space $S$, all median-preserving isometries of $S$ are isometries of the hyperbolic space $Y(S)$. If $S$ is the mapping class group of a finite-type surface $\Sigma$, or its Teichmüller space with the Teichmüller metric, then $Y(S)$ is $\MCG\Sigma$--equivariantly quasiisometric to the curve graph of $\Sigma$.
\end{introthm}

This provides a systematic construction of the curve graph from only median and metric information. It also lets the curve graph be viewed as a kind of quotient of the coarsely injective space $X_\C(S)$ of Theorem~\ref{introthm:mcg_inj}, because the underlying walls are the same. It should be noted that the hyperbolic space $Y(S)$ is not actually isometric to the curve graph of $S$, because it is not a geodesic space, only a roughly geodesic space.

\subsection{Further discussion of the construction} \label{subsec:further}

In this subsection, we discuss Construction \ref{constr:main} in more detail and provide several applications. Given a set $S$ and a collection $P$ of bipartitions, an \emph{ultrafilter} is a ``consistent orientation'' of the elements of $P$; see Definition~\ref{def:filter}. Each $s\in S$ determines a \emph{principal} ultrafilter $\phi_s$ by orienting each $h \in P$ towards $s$. Let $\hat X$ be the set of all ultrafilters on $P$. We say that $x,y\in\hat X$ are \emph{separated} by $c\subset P$ if they orient every element of $c$ differently. Given $\C\subset 2^P$, let $\dist_\C$ be the function on $\hat X$ given by $\dist_\C(x,y)=\sup\{|c|:c\in\C\text{ separates }x\text{ from }y\}$, which takes values in $\mathbb N\cup\{\infty\}$.

\begin{definition}[Dualisable system] \label{introdef:dualisable_system} 
Let $(S,P)$ be a set with walls. Suppose that $\C\subset2^P$ is closed under taking subsets and contains all singletons. We say that $\C$ is a \emph{dualisable system} on $P$ if $\dist_\C(\phi_s,\phi_t)<\infty$ for all $s,t\in S$. The \emph{$\C$--dual} of $S$ is $X_\C=(X,\dist_\C)$, where $X=\{x\in\hat X\,:\dist_\C(x,\phi_s)<\infty\}$.
\end{definition}

The following theorem summarises the correspondences between combinatorics of $\C$ and metric properties of $X_\C$ obtained in this paper; see Theorem~\ref{thm:coarsely_injective} and Corollary~\ref{cor:rough_geodesic_hyperbolic}.

\begin{introthm} \label{introthm:dual_spaces}
For any dualisable system $\C$ on $(S,P)$, the space $X_{\C}$ is a metric space, and
\begin{itemize}
\item   if $\C=2^P,$ then $X_{\C}$ is Sageev's dual CAT(0) cube complex;
\item   if $\C$ consists of all chains in $P$, then $X_{\C}$ is the vertex set of a Helly graph.
\end{itemize}
Moreover, assuming that $\C$ is \emph{$m$--gluable}:
\begin{itemize}
\item   if each $c \in \C$ is a chain, then $X_{\C}$ is coarsely injective;
\item   if $\C$ is $L$--separated, then $X_{\C}$ is hyperbolic.
\end{itemize}
\end{introthm}

The set $\C$ is said to be \emph{$m$--gluable} if it is closed under taking certain unions (see Definition~\ref{def:gluable}). It is $L$--separated if each $c \in \C$ is a chain and for any $\{h_1,h_2\}\in\C$, every $c\in\C$ whose elements all \emph{cross} both $h_1$ and $h_2$ has $|c|\le L$.

The final bullet of Theorem~\ref{introthm:dual_spaces} bears similarity to the construction of \cite{bowditch:topological}. Bowditch starts with a group $G$ acting properly discontinuously and cocompactly on the space of triples of a \emph{perfect} metric compactum. He then constructs ``annulus systems'', which are certain collections of bipartitions that play the role of $\C$, and shows that they are $L$--separated, using the geometry of the action.

Whilst Construction~\ref{constr:main} is based on Sageev's construction, one fundamental difference is that it does not require local finiteness of the collection of walls. Of course, there is a trade-off between how general the input can be and how strong the conclusions are. Consider hyperbolic groups, for instance. Adding the assumption of a cocompact cubulation to hyperbolicity has some very strong consequences \cite{agol:virtual,wise:structure}, but there are many hyperbolic groups with property~(T) \cite{zuk:property,kotowskikotowski:random}, and these have no unbounded actions on CAT(0) cube complexes \cite{nibloreeves:groups}. Sageev's construction itself (and even the continuous version of \cite{chatterjidrutuhaglund:kazhdan,fioravanti:roller}) cannot, therefore, convey any of the geometry of such groups. Nevertheless, Theorem~\ref{introthm:universal} outputs a metric space $X$ that is quasiisometric to $G$ and still has several fine properties in common with CAT(0) cube complexes, as we now discuss. 

For one, in the full generality of Construction~\ref{constr:main}, every finite group acting on $S$ fixes a point in a ``first subdivision'' of $X$ (Proposition~\ref{prop:finite_fix}); and so a standard argument shows that if $G$ acts properly coboundedly on $X$ then it has finitely many conjugacy classes of finite subgroups. Even in $\delta$--hyperbolic spaces, one can only guarantee an orbit of diameter roughly $\delta$. In this regard, $X$ can be thought of as playing a similar role to the \emph{injective hull} \cite{lang:injective}. That said, neither the median structure on $X$ nor constructions such as normal wall paths have analogues in the injective hull, as they rely on the walls. Also, $X$ better reflects the coarse geometry of $S$. For instance, if $S$ is the standard cubulation of $\R^3$, then $X$ is quasiisometric to $S$, whereas the injective hull is $\R^4$.

Furthermore, the dual space $X_{\C}(S)$ of a strong coarse median space $S$ has a median algebra structure that is uniformly close to the coarse median structure on $S$ (Lemma~\ref{lem:scm_qmqie}). Each \emph{coarsely convex} set $Y \subset S$ (such as quasiconvex subsets of hyperbolic spaces) is at a uniform Hausdorff-distance from a wall-theoretically convex set $Z$ in $X_\C(S)$: an intersection of half spaces (Definition~\ref{def:convex}). The median algebra structure on $X_\C(S)$ then allows us to produce a \emph{gate} map $\g:X_\C(S)\to Z$. This gate map is a closest-point projection in $X_{\C}(S)$ and is 1-Lipschitz. To summarise, one can replace a strong coarse median space $S$ by a space $X_{\C}(S)$ equivariantly quasiisometric to $S$ and with upgraded fine properties.

This highlights a philosophical difference between this article and \cite{haettelhodapetyt:coarse} compared to \cite{durhamminskysisto:stable,durham:cubulating}. In the former articles, a given space is extended to a larger one with good properties, whereas in the latter ones, subspaces of a given group are accurately approximated by cube complexes.

As a final instance of the improved fine properties of $X_\C(S)$, we mention the following rank-rigidity result for hierarchically hyperbolic groups, which improves upon the coarse product structure of \cite{durhamhagensisto:boundaries,petytspriano:unbounded}.

\begin{introcor}
Let $G$ be a hierarchically hyperbolic group. There is a natural choice of $\C$ such that $X_\C(G)$ is $G$--equivariantly quasiisometric to $G$ and either:
\begin{itemize}
\item   $G$ contains a Morse element; or
\item   $X_\C(G)=X_{\C_1}(S_1)\times X_{\C_2}(S_2)$, where $S_1$ and $S_2$ are unbounded hierarchically hyperbolic spaces and $\C=\C_1\sqcup\C_2$.
\end{itemize}
\end{introcor}

\medskip

The article is divided into two parts. The first part focuses on general properties of Construction~\ref{constr:main}, and natural conditions one can consider on $\C$. The intention is to be self-contained, with a view to being applicable more widely. The outcome is a number of recipes that, given collections of walls satisfying certain combinatorial conditions, will produce associated spaces with good properties.

In the second part, we apply these generalities in two main settings, in order to deduce the results discussed above. In a short final section we point to some potential further directions that could be pursued. The appendix with Davide Spriano concerns quasimorphisms.

\ubsh{Acknowledgements}
We are happy to thank Stefanie Zbinden for friendly discussions about her work and ours. We thank Kasra Rafi for conversations about Teichmüller space, and Matt Durham, Anthony Genevois, and Mark Hagen for comments on an earlier version of this article. We also thank Indira Chatterji for suggesting the name ``hyperbolic core''. We are grateful to the organisers of the thematic program Geometric Group Theory in Montreal in 2023, where this work was started. Finally, we thank the anonymous referee for their very thorough reading of the paper and useful comments.
\uesh

\section{Background on medians and CAT(0) cube complexes} \label{sec:background}

An authoritative general reference for the material discussed here is \cite{bowditch:median:book}. A \emph{median algebra} is a set $X$ together with a ternary operation $\mu$ such that for all $a,b,c,d,e\in X$ we have
\[
\mu(a,b,c)=\mu(b,a,c)=\mu(c,a,b), \quad \mu(a,a,b)=a, \quad \mu(a,b,\mu(c,d,e))=\mu(\mu(a,b,c),\mu(a,b,d),e).
\]
The latter equality is called the \emph{five-point condition}. (See also \cite{roller:poc,bandelthedlikova:median}.) One family of median algebras is provided by \emph{median graphs}. A graph $X$ is median if for each three vertices $v_1,v_2,v_3$ there is a unique point lying on some geodesic from $v_i$ to $v_j$ for all $i,j$. By a result of Chepoi \cite[Thm~6.1]{chepoi:graphs}, a graph is median if and only if it is the 1--skeleton of some CAT(0) cube complex (see also \cite{genevois:algebraic}).

CAT(0) cube complexes can equivalently be characterised in terms of their \emph{hyperplanes}. (See \cite{wise:structure} for thorough information on cubical hyperplanes.) Indeed, \emph{Sageev's construction} \cite{sageev:ends}, as clarified in \cite{nica:cubulating,chatterjiniblo:from}, shows how to reconstruct the one-skeleton from the combinatorics of the hyperplanes. Each hyperplane $h$ has two corresponding \emph{halfspaces}, $h^-$ and $h^+$, which are the components of its complement. 

The median $\mu$ on a CAT(0) cube complex $Q$ can be described in terms of hyperplanes as follows: given vertices $x_1,x_2,x_3\in Q$, the point $\mu(x_1,x_2,x_3)$ is the unique vertex that, for every hyperplane $h$, lies on the same side of $h$ as the majority of the $x_i$.

A subcomplex $C$ of a CAT(0) cube complex $Q$ is \emph{convex} if any of the following equivalent conditions hold:
\begin{itemize}
\item   $\mu(c_1,c_2,x)\in C$ for all $c_1,c_2\in C$, $x\in Q$ (this is called \emph{median-convexity});
\item   $C$ is the largest subcomplex contained in the intersection of some collection of halfspaces;
\item   $C$ is (geodesically) convex with respect to either the CAT(0) or the $\ell^1$ metric on $Q$.
\end{itemize}
The \emph{convex hull}, $\hull_QA$, of $A\subset Q$ is the smallest convex subcomplex of $Q$ containing $A$.

Given a convex subcomplex $C$ of a CAT(0) cube complex $Q$, there is a natural projection map $\g_C:Q\to C$, called the \emph{gate} map. Here are three equivalent descriptions of the gate of a vertex $x\in Q$:
\begin{itemize}
\item   $\g_C(x)$ is the unique closest point in $C$ to $x$ when $Q$ is equipped with the combinatorial metric;
\item   $\g_C(x)$ is the unique vertex of $Q$ such that a hyperplane $h$ separates $x$ from $\g_C(x)$ if and only if $h$ separates $x$ from $C$;
\item   $\g_C(x)$ is the unique point of $C$ with the property that $\mu(c,\g_C(x),x)=\g_C(x)$ for all $c\in C$.
\end{itemize}
Combining the first and third characterisations explains the terminology: every point $c\in C$ has a geodesic to $x$ that passes through $\g_C(x)$.

The following technical lemma will not be needed until Section~\ref{sec:scm}.

\begin{lemma} \label{lem:gate_via_medians}
Let $Q$ be a CAT(0) cube complex, and let $a_1,\dots,a_n\in Q$, for $n\geq2$. The gate map $Q\to\hull_Q(a_1,\dots,a_n)$ can be expressed as $x\mapsto\mu(a_n,x,\mu(a_{n-1},x,\mu(\dots,\mu(a_2,x,a_1)\dots)$.
\end{lemma}

\begin{proof}
Write $\g'(x)$ for the map in the statement of the lemma. Starting at the ``inner level'' of the expression for $\g'(x)$, observe that $\mu(a_2,x,a_1)\in\hull_Q(a_1,a_2)\subset\hull_Q(a_1,\dots,a_n)$. Proceeding outwards one level at a time, we see inductively that $\g'(x)\in\hull_Q(a_1,\dots,a_n)$.

Suppose that a hyperplane $h$ of $Q$ has $x\in h^-$, $\g'(x)\in h^+$. 
Consider the ``outer level'' of the expression for $\g'(x)$. The majority of the arguments of that median must lie on the same side of $h$ as $\g'(x)$, so we must have $a_n\in h^+$, and also the nested expression must determine a point of $h^+$. The same argument then applies at the next level of the expression to give $a_{n-1}\in h^+$, and successively we find that $a_i\in h^+$ for all $i$. 

We have shown that the only hyperplanes separating $x$ from the point $\g'(x)\in\hull_Q(a_1,\dots,a_n)$ are those that separate $x$ from $\hull_Q(a_1,\dots,a_n)$ itself. By the second characterisation above, $\g'(x)$ is the gate of $x$ to $\hull_Q(a_1,\dots,a_n)$.
\end{proof}

\part{Generalising Sageev's construction}

\section{Ultrafilters and dualisable systems} \label{sec:construction}

In this section we introduce the core framework within which we operate.  

A \emph{set with walls} is a pair $(S,P)$, where $S$ is a set and $P$ is a set of bipartitions $h=\{h^+,h^-\}$ of $S$ (notice that by definition, a partition always consists of non-empty subsets and so $h^+, h^-$ must both be non-empty). That is, $h^-,h^+\subset S$ have $S=h^-\cup h^+$ and $h^-\cap h^+=\varnothing$. We refer to $h$ as a \emph{wall}, and to $h^\pm$ as the \emph{halfspaces} of $h$.

\begin{definition}[Ultrafilter] \label{def:filter}
A \emph{filter} $\phi$ on $P$ consists of a subset $Q\subset P$ and a choice of halfspace $\phi(h)\in\{h^+,h^-\}$ for each $h\in Q$ such that: 
\[
\text{if }h_1,h_2\in Q\text{ have }h_1^+\subset h_2^+\subset S,\text{ then }\phi(h_1)=h_1^+\text{ implies that }\phi(h_2)=h_2^+.
\]
We say that $\phi$ is \emph{supported} on $Q$. An \emph{ultrafilter} is a filter whose support is $P$. 
\end{definition}

The terminology comes from the fact that one can see these filters as being filters on the Boolean subalgebra of $2^S$ generated by the halfspaces of $S$. In a Boolean algebra every filter extends to an ultrafilter, giving the following.

\begin{lemma}[Ultrafilter lemma] \label{lem:ultrafilter_lemma}
Every filter on $P$ extends to an ultrafilter.
\end{lemma}

\ubsh{Space of ultrafilters}
Write $\hat X$ for the set of ultrafilters on $S$ defined by $P$. Every point $s\in S$ defines an ultrafilter by setting $\phi_s(h)$ to be the halfspace of $h$ that contains $s$, for every $h\in P$. If each pair of points in $S$ is separated by some element of $P$, then $s\mapsto\phi_s$ is injective. Even when this is not the case we tend to abuse notation and fail to distinguish between $S$ and its image in $\hat X$.
\uesh

Another way of casting Lemma~\ref{lem:ultrafilter_lemma} is that it says that any collection of halfspaces that intersect pairwise in $S$ has nonempty total intersection in $\hat X$: it is a kind of Helly property.

We say that $h_1,h_2\in P$ \emph{cross} if all four orientations of $h_1$ and $h_2$ are filters. Equivalently, all four quarterspaces $h_1^\pm\cap h_2^\pm$ are nonempty, in either $S$ or $\hat X$. Equivalently, there is no pair of orientations, say $h_1^+$ and $h_2^+$, such that $h_1^+\subset h_2^+$ as subsets of~$S$.

Let $h\in P$. If $s$ and $t$ are points of $S$ that lie in different halfspaces of $h$, then we say that $h$ \emph{separates} $s$ from $t$. More generally, given subsets $A,B\subset\hat X$, we say that $h$ \emph{separates} $A$ from $B$ if $x_1(h)=x_2(h)\ne y_1(h)=y_2(h)$ for all $x_1,x_2\in A$ and all $y_1,y_2\in B$. Similarly, we say that $h_1\in P$ separates $x\in\hat X$ from $h_2\in P$ if $x(h_1)\subset x(h_2)$.

\subsection{Dualisable systems}

So far, the set $\hat X$ is both extremely large (in general) and lacking any metric structure. The following definition lets us deal with both of these issues simultaneously. 

\begin{definition} \label{def:dualisable_system}
A \emph{dualisable system} for $(S,P)$ is a subset $\C\subset2^P$ such that the following hold.
\begin{itemize}
\item   $\C$ contains all singletons and is closed under taking subsets.
\item   For each pair $s,t\in S$, there is a number $M_{st}$ such that $|c|\le M_{st}$ for all $c\in\C$ with the property that every $h\in c$ separates $s$ from $t$. 
\end{itemize}
\end{definition}


\ubsh{The dual space} 
Given a dualisable system $\C$ for $(S,P)$, consider the function on $\hat X\times\hat X$ given by 
\[
\dist_\C(x,y) \,=\, \sup\{|c|\,:\,c\in\C\text{ and every }h\in c\text{ separates }x\text{ from }y\}.
\]
The \emph{$\C$--dual} of $S$ is the space $X_\C=\{x\in\hat X\,:\,\dist_\C(x,s)<\infty\text{ for all }s\in S\}$, equipped with the function $\dist_\C$.
\uesh

\begin{lemma}
The function $\dist_\C$ is an extended metric. Its restrictions to $X_\C$ and $S$ (strictly speaking its image in $\hat X$) are metrics. 
\end{lemma}

\begin{proof}
Because $\C$ contains all singletons, the function $\dist_\C$ separates points of $\hat X$. It is evidently symmetric. The fact that the restrictions take only finite values follows from the second bullet point of Definition~\ref{def:dualisable_system}. It remains to show that $\dist_\C$ satisfies the triangle inequality.

Let $x,y,z\in\hat X$. Every $c\in\C$ separating $x$ from $y$ can be partitioned as $c=c_x\sqcup c_y$, where every element of $c_x$ separates $x$ from $z$ and every element of $c_y$ separates $y$ from $z$. By definition, $\dist_\C(x,z)\ge|c_x|$ and $\dist_\C(z,y)\ge|c_y|$. Since this holds for every $c$ separating $x$ from $y$, we have $\dist_\C(x,y)\le\dist_\C(x,z)+\dist_\C(z,y)$. 
\end{proof}

By a \emph{chain} in $P$, we mean a sequence $c=(h_i)_{i\in I}$, where $I$ is some (finite or infinite) interval in $\Z$ and $h_i\in P$, such that $h_{i-1}^-\subset h_i^-\subset h_{i+1}^-$ for all $i$.

\begin{example} \label{eg:sageevable}
In order to be able to apply Sageev's construction to a set with walls $(S,P)$, it must be assumed that each pair of points in $S$ is separated by finitely many walls. This assumption is exactly equivalent to the statement that $2^P$ is a dualisable system. If we then take $\C=2^P$, then $X_\C$ is exactly the dual CAT(0) cube complex of Sageev. If $X_\C$ contains no infinite cubes, then taking $\C'$ to be the set of all chains in $P$ we get that $X_{\C'}$ is exactly the Helly graph obtained from $X_\C$ by \emph{thickening} each cube (i.e. replacing it by a complete graph) \cite{bandeltvandevel:superextensions}. Otherwise $X_{\C'}$ will be bigger. Indeed, if $S=(\bigoplus_{\mathbb N}\{0,1\},\ell^1)$ is an infinite cube and $P$ is the set of cubical walls, then $X_\C=S$, whereas $X_{\C'}$ is metrically an uncountable clique.

The set of all chains can be a dualisable system even if $2^P$ fails to be. For instance, in the above example, $\C'$ is dualisable for $(X_{\C'},P)$, whereas $\C$ is not. The following simple example illustrates how configurations of this type can arise naturally in continuous settings. Let $S=\mathbb C$ and let $P$ be the set of walls induced by taking all lines through 0. The points 1 and $-1$ are separated by uncountably many walls, but no two walls form a chain.
\end{example}

Recall that a subset $Y$ of a median algebra $\hat X$ is said to be \emph{median-convex} if for any $x,y \in Y$ and $z \in \hat X$, we have $\mu(x,y,z) \in Y$.

\begin{lemma}[Median algebra] \label{lem:dual_median}
The set $P$ gives $\hat X$ the structure of a median algebra. If $\C$ is a dualisable system, then $X_\C$ is median-convex. In particular, $X_\C$ is itself a median algebra.
\end{lemma}

\begin{proof}
Given three ultrafilters $x_1,x_2,x_3\in\hat X$, define an orientation $\phi$ of each $h\in P$ by setting $\phi(h)$ to be the halfspace that contains the majority of the $x_i$. This is clearly an ultrafilter on $P$, and it is straightforward to check that this assignment $\mu:(x_1,x_2,x_3)\mapsto\phi$ satisfies the five-point condition. 

If $x_1,x_2\in X_\C$ and $z\in\hat X$, then $x_2$ is at finite $\dist_\C$--distance from $x_1$. If $c\in\C$ separates $x_1$ from $\mu(x_1,x_2,x_3)$ then it separates $x_1$ from $\{x_2,x_3\}$. In particular $\dist_\C(x_1,\mu(x_1,x_2,x_3))\le\dist_\C(x_1,x_2)$ is finite.
\end{proof}

The median defined in the proof of Lemma~\ref{lem:dual_median}, which is also known as the ``majority vote'' median, will be denoted $\mu:\hat X^3\to\hat X$ throughout. 

Although it is both a median algebra and a metric space, in general $X_\C$ will not be a \emph{median metric space}, as can be seen from Example~\ref{eg:sageevable}. Note the following simple observation.

\begin{lemma}
If a group $G$ acts on $S$ and preserves the dualisable system $\C$, then $G$ acts by isometries on $X_\C$.
\end{lemma}

\subsection{Convexity and gates} \label{subsec:gate}

Given a dualisable system $\C$ on $(S,P)$, aside from the notion of median-convexity on $X_\C$ that is available thanks to Lemma~\ref{lem:dual_median}, there is another, stronger version of convexity that is perhaps more natural. 

\begin{definition}[$P$-convex] \label{def:convex}
Let $(S,P)$ be a set with walls. A subset $\hat C\subset\hat X$ is \emph{$P$-convex} if there is a subset $Q\subset P$ and a choice $\hat C(h)\in\{h^+,h^-\}$ for all $h\in Q$ such that $\hat C=\{v\in\hat X\,:\,v(h)=\hat C(h)\text{ for all }h\in Q\}$. In other words, $\hat C$ is an intersection of $P$-halfspaces in $\hat X$. 
\end{definition}

Note that $\hat C$ is nonempty if and only if the choices $\hat C(h)$ define a filter with support $Q$.

\begin{remark}[$P$--convexity vs median-convexity] \label{rem:median_convex_vs_wall_convex}
Whilst it is true that every $P$-convex set in $\hat X$ is also median-convex with respect to $\mu$, the converse can fail in general. For instance, the set $\bigoplus_{\mathbb N}\{0,1\}\subset \prod_{\mathbb N}\{0,1\}$ is median-convex but is not an intersection of halfspaces. In this example, $\bigoplus_{\mathbb N}\{0,1\}$ is exactly the set of points in $\prod_{\mathbb N}\{0,1\}$ that are at finite $\ell^1$ distance from the zero sequence. Another, perhaps simpler, example appears in Remark~\ref{rem:gatedness}.
\end{remark}

We now define a projection map to $P$-convex sets inside $\hat X$.

\begin{proposition}[Gates] \label{prop:gate}
Given a nonempty $P$-convex set $\hat C\subset\hat X$ and an ultrafilter $z\in\hat X$, let $w$ be the orientation of all elements of $P$ obtained from $z$ by switching the orientation of every element of $P$ that separates $z$ from $\hat C$. The orientation $w$ is an ultrafilter, and $w\in\hat C$.
\end{proposition}


\begin{proof}
Let $Q\subset P$ be the subset witnessing the $P$-convexity of $\hat C$. Note that $z\in \hat C$ if and only if no $h\in Q$ separates $z$ from $\hat C$, in which case $w=z$. Otherwise, it is clear from the construction that if $w$ is an ultrafilter then it lies in $\hat C$. It therefore suffices to show that $w$ is a filter, because its support is $P$. That is, supposing that $h_1,h_2\in P$ satisfy $h_1^+\subset h_2^+$, we must show that if $w(h_1)=h_1^+$, then $w(h_2)=h_2^+$.

First suppose that $h_1$ does not separate $z$ from $\hat C$. By Lemma~\ref{lem:ultrafilter_lemma} there must be some $c\in\hat C$ such that $c(h_1)=z(h_1)$. From the construction of $w$, we have $c(h_1)=z(h_1)=w(h_1)=h_1^+$. Since $c$ and $z$ are ultrafilters, $c(h_2)=h_2^+=z(h_2)$, so $h_2$ also does not separate $z$ from $\hat C$. We therefore have $w(h_2)=z(h_2)=h_2^+$, as desired.

Alternatively, $h_1$ separates $z$ from $\hat C$. Because $\hat C$ is nonempty, there exists $c\in\hat C$. We have $c(h_1)\ne z(h_1)$ and $z(h_1)\ne w(h_1)=h_1^+$. Because $c$ is an ultrafilter, we therefore have $c(h_2)=h_2^+$. If $h_2$ separates $z$ from $\hat C$, then $c(h_2)$ disagrees with $z(h_2)$, which disagrees with $c(h_2)$, so $w(h_2)=h_2^+$. Otherwise, $h_2$ does not separate $z$ from $\hat C$, and so $c(h_2)$ agrees with $z(h_2)$, which agrees with $w(h_2)$. Thus we have $w(h_2)=h_2^+$ in either case.
\end{proof}

A subset $A$ of a median algebra $(M,m)$ is called \emph{gated} if for every $x\in M$ there exists a unique $g\in A$ such that $m(a,g,x)=g$ for all $a\in A$. In view of the following lemma, we shall call the ultrafilter $w$ constructed in Proposition~\ref{prop:gate} the \emph{gate} of $z$ to $\hat C$, and write $w=\g_{\hat C}(z)$.

\begin{lemma} \label{lem:P-convex_gated}
A subset $A \subset \hat{X}$ is gated if and only if it is $P$-convex.    
\end{lemma}

\begin{proof} 
It is easy to check that if $A$ is $P$-convex, then the map $\g_A$ makes $A$ gated. Conversely, suppose that $A\subset\hat X$ is gated. For each $z\in\hat X\ssm A$, let $g_z$ be its gate to $A$, and set $P_z=\{h\in P\,:\,h\text{ separates }z\text{ from }g_z\}$. If $h\in P_z$ for some $z$, then orient it so that $z\in h^+$, $g_z\in h^-$. The set $A$ must be contained in $h^-$, for otherwise we could find $a\in A\cap h^+$, and then $\mu(a,z,g_z)\in h^+$ cannot be $g_z$. Hence $A$ is contained in an intersection of $P$-halfspaces in $\hat X$. But we showed that any wall separating a point from its gate must separate that point from $A$, so in fact $A$ is equal to that intersection of $P$-halfspaces.
\end{proof}

\begin{remark}[Median algebras, convexity, and gatedness] \label{rem:gatedness}
Let $(M,m)$ be a median algebra. Recall that a subset $B\subset M$ is \emph{median-convex} if $m(b,b',x)\in B$ for all $b,b'\in B$, $x\in M$. If we take a pair of points $b,b'\in B$, then they have a natural hull, namely $\{m(a,b,b')\,:\,a\in M\}$, which is necessarily contained in $B$. The map $m(b,b',\cdot):M\to B$ is a gate map to the hull of this pair of points. Via Lemma~\ref{lem:gate_via_medians}, the four-point condition then implies that $B$ is ``finitely gated'': it contains the hull of each of its finite subsets, which are all gated. In a sense, this means that we can gate to the ``finitely supported'' parts of $B$. 

In the discrete, finite-rank case, median-convexity and gatedness agree, so if $B$ is median-convex then there is a gate map to it. In general, though, this is not the case - we give an example below. In a sense, as seen in the above discussion, this is because median-convexity is a ``limit of finite subsets'' type of property. This perspective shows that gatedness is perhaps more natural than median-convexity, because it is an absolute property without cardinality restrictions.

The two notions give two families of walls that one can consider on a given median algebra $M$. The more common choice is the set of all \emph{median-convex} walls: bipartitions $\{h^+,h^-\}$ where both $h^+$ and $h^-$ are nonempty median-convex subsets \cite{roller:poc,fioravanti:roller}. Alternatively, one can consider \emph{gated} walls, where $h^+$ and $h^-$ are required to be nonempty and gated. In the case where $M$ is a \emph{Stone} median algebra (i.e. a compact and totally disconnected topological median algebra, see \cite[\S12.5]{bowditch:median:book}), one can see that a wall is gated exactly when it is a \emph{clopen} wall: a partition into two nonempty open halfspaces. The clopen walls are the ones that are used in duality statements. 

A good example to keep in mind is the following. Let $S=\Z$, with $P$ its usual cubical walls. The space $\hat X$ is the Roller compactification \cite{roller:poc}. It is a topological median algebra whose points can be thought of as living in $\Z\cup\{-\infty,\infty\}$. There are median-convex walls in $\hat X$ not coming from $P$. For instance, consider $h^+=\{\infty\}$, $h^-=\hat X\ssm h^+$. This defines a median-convex wall of $\hat X$, but note that it is not clopen: the halfspace $h^+$ is not open. Correspondingly, the halfspace $h^-$ is not gated. The clopen walls of $\hat X$ are exactly the gated walls, which are exactly those coming from $P$. This is part of what makes the $P$-convexity in Definition~\ref{def:convex} the appropriate notion here, even though $S$ itself can sometimes fail to be $P$-convex in $\hat X$.
\end{remark}

To clarify, if $(S,P)$ is a set with walls, then the term ``wall'' will always mean an element of $P$, even though the median algebra $\hat X$ may have additional median-convex walls.

\begin{lemma} \label{lem:gate_X}
If $C=X_\C\cap\hat C$ is nonempty, where $\hat C$ is some $P$-convex subset of $\hat X$, then $\g_{\hat C}(z)\in C$ for all $z\in X_\C$.
\end{lemma}

\begin{proof}
Let $x\in C$. Any $z\in X_\C$ lies at finite distance from $x$. As $\g_{\hat C}(z)$ is obtained from $z$ by switching the orientations of a subset of the walls separating $z$ from $x$, every wall separating $\g_{\hat C}(z)$ from $x$ must separate $z$ from $x$. Hence $\g_{\hat C}(z)$ is at finite distance from $x$, and so lies in both $X_\C$ and $\hat C$. 
\end{proof}

In view of Lemmas~\ref{lem:P-convex_gated} and~\ref{lem:gate_X}, a subset $C\subset X_\C$ is gated in the median algebra $X_\C$ exactly when there exists some $P$-convex set $\hat C\subset\hat X$ such that $C=X_\C\cap\hat C$. Indeed, if $C\subset X_\C$ is gated then the argument of Lemma~\ref{lem:P-convex_gated} shows that $C=X_\C\cap\hat C$ for some $P$-convex $\hat C\subset\hat X$, and if the latter holds then Lemma~\ref{lem:gate_X} shows that $C$ is gated.

If $C=X_\C\cap\hat C$ for some $P$-convex $\hat C\subset\hat X$ and $C$ is nonempty, then we write $\g_C(z)=\g_{\hat C}(z)$ for $z\in X_\C$. 


Both of the next two lemmas could equally well be stated for $P$-convex subsets of $\hat X$, but that is not to our purposes. The following is immediate from the construction of the gate and the definition of $\dist_\C$.

\begin{lemma}[Closest-point] \label{lem:gate_closest_point}
Let $C\subset X_\C$ be gated, and let $x\in X_\C$. If $h\in P$ separates $x$ from $\g_C(x)$, then $h$ separates $x$ from $C$. In particular, $\dist_\C(x,c)\ge\dist_\C(x,\g_C(x))$ for all $c\in C$.
\end{lemma}

\begin{lemma}[Lipschitz] \label{lem:gate_lipschitz}
Let $C\subset X_\C$ be gated and let $A,B\subset X_\C$. If $h\in P$ separates $\g_C(A)$ from $\g_C(B)$, then $h$ separates $A$ from $B$. In particular, $\g_C$ is 1--Lipschitz.
\end{lemma}

\begin{proof}
If $h$ separates $\g_C(A)$ from $\g_C(B)$, then there are points of $C$ on both sides of $h$, so $h$ cannot separate any point of $X_\C$ from $C$. In particular,  for all $z\in X_\C$ the orientation of $h$ determined by $z$ is the same as that determined by $\g_C(z)$, so $A$ and $B$ are on different sides of $h$.
\end{proof}

We finish this section with a simple observation about finite group actions. Given a dualisable system $\C$ for $(S,P)$, one could define a \emph{first subdivision} of $X_\C$ by ``doubling'' each $h\in P$ into two identical partitions $h_1,h_2$ and leaving all other crossing relations the same. We refrain from making this formal, because we generally wish to work with $P$ being a \emph{set} of bipartitions, rather than a multiset, but it is completely analogous to the first cubical subdivision of a CAT(0) cube complex. After simply noting the following, subdivisions will not be mentioned again.

\begin{proposition} \label{prop:finite_fix}
Let $\C$ be a dualisable system on $(S,P)$. If $G$ is a finite group acting on $S$ that preserves $P$, then $G$ fixes a point in the first subdivision of $X_\C$.
\end{proposition}

\begin{proof}
Let $A$ be a $G$--orbit in $S$. For each $h\in P$, let $h^+$ be the halfspace containing more than half of the elements of $A$, if it exists. Let $Q$ be the set of such $h$. It is easy to see that $\phi(h)=h^+$ is a filter supported on $Q$, and any ultrafilter extending it lies in $X_\C$, because it is separated from each element of $A$ only by elements of $P$ separating points of $A$. Also note that $\phi$ is fixed by $G$. If $P=Q$ then we are done.

Otherwise, $Q\ne P$. Let $Q'$ be the set of all $h\in P\ssm Q$ such that there is some $g\in G$ for which $gh\ne h$ and $gh$ does not cross $h$. Note that since $h$ divides $A$ into equal halves, both $h$ and $gh$ define the same partition of $A$. By definition, there are halfspaces $h^-$ and $(gh)^-$ that are disjoint. We must have $(gh)^-=g(h^-)$, for otherwise we would have $g(h^+)\subsetneq h^+$, which would imply that $g$ had infinite order. One similarly argues that this labelling did not depend on the choice of $g$. In other words, the allocation $\phi(h)=h^+$ on $Q'$ is $G$--invariant. 

Let us show that this allocation $\phi$ is a filter on $Q'$. If not, then there are $h_1,h_2\in Q'$ such that $h_1^+$ and $h_2^+$ are disjoint. Let $g\in G$ be such that $gh_1\ne h_1$ and $gh_1$ does not cross $h_1$. We have $h_2^+\subset h_1^-\subset gh_1^+$, and so we must have $g^{-1}h_2^+\subset h_1^+$. But this shows that $h_2^+$ is disjoint from its translate by $g^{-1}$, contradicting $G$--invariance of $\phi$. Hence if $P=Q\cup Q'$, then $\phi$ is an ultrafilter that is fixed by $G$.

The final case is that there is some $h\in P\ssm Q$ such that for each $g\in G$, either $h$ crosses $gh$ or is equal to it. If $k\in P\ssm(Q\cup Q')$ does not cross $h$, then set $\phi(k)$ to be the halfspace containing a halfspace of $h$, and extend to $G\cdot\{k\}$ equivariantly. Let $Q_h$ be the set of elements of $P\ssm(Q\cup Q')$ whose $G$--translates all cross $h$. We obtain an ultrafilter extending $\phi$ in the first subdivision of $X_\C$ by pointing the ``doubled'' copies of the elements of $Q_h$ towards each other, and $\phi$ is fixed by $G$.
\end{proof}

The corresponding statement does not always hold for \emph{bounded} group actions, as can be seen from a transitive action of $\Z$ on an infinite clique.

\section{Systems of chains} \label{sec:properties}

In the applications we have in mind for the constructions of Section~\ref{sec:construction}, the dualisable system $\C$ will be a set of chains in $P$. Recall that by a chain in $P$, we mean a sequence $c=(h_i)_{i\in I}$, where $I$ is some (finite or infinite) interval in $\Z$ and $h_i\in P$, such that $h_{i-1}^-\subset h_i^-\subset h_{i+1}^-$ for all $i$.

\begin{definition}[System of chains]
A dualisable system $\C$ is a \emph{system of chains} if every element of $\C$ is a chain in $P$.
\end{definition}

An immediate advantage of assuming that elements of $\C$ are chains is that every element of $\C$ then comes with a total order, so that one can speak of the \emph{minimal} and \emph{maximal} element of a finite chain $c\in\C$. In this section, we investigate some of the metric properties that one can deduce about the dual space $X_\C$ when $\C$ is a system of chains.

\begin{lemma} \label{lem:convex_balls}
If $\C$ is a system of chains on $(S,P)$, then every ball in $X_\C$ is gated.
\end{lemma}

\begin{proof}
Let $x\in X_\C$ and $r\ge0$. Since $\dist_\C$ is integer-valued, we may assume that $r\in\Z$. Let $Q$ be the set of all $h\in P$ such that there exists some $c\in\C$ containing $h$, with $|c|=r+1$ and all other elements of $c$ separating $x$ from $h$. For each $h\in Q$, let $\hat C(h)=x(h)$ be the halfspace containing $x$. This defines a $P$-convex set $\hat C\subset\hat X$. If $z\in\hat C$, then because $\C$ is a system of chains, we have $\dist_\C(x,z)\le r$. In particular, $\hat C\subset X$. Conversely, if $\dist_\C(x,z)\le r$, then no element of $Q$ can separate $z$ from $x$, so $z\in\hat C$. Thus the $P$-convex set $\hat C$ is the $r$--ball about $x$ in $X$. But $\hat C\subset X_\C$, so it is also the $r$--ball about $x$ in $X_\C$.
\end{proof}

For systems of chains, we now describe a family of paths that are analogues of the \emph{normal cube paths} introduced by Niblo--Reeves in \cite{nibloreeves:geometry}. In fact, in the case where $2^P$ is dualisable (see Example~\ref{eg:sageevable}) and $\C$ is the set of all chains, $X_\C$ is a CAT(0) cube complex with the $\ell^\infty$ metric and the paths we shall define will exactly be normal cube paths. 

\begin{definition}[Normal wall path]
Let $\C$ be a system of chains. Given $x,y\in X_\C$, the \emph{normal wall path} $\sigma(x,y)=\sigma_{xy}$ from $x$ to $y$ is the sequence
\[
(x=\g_{B(x,0)}(y),\, \g_{B(x,1)}(y),\, \dots,\, \g_{B(x,n)}(y)=y),
\]
where $B(x,r)$ denotes the $r$--ball in $X_\C$ centred on $x$, which is gated by Lemma~\ref{lem:convex_balls}. See Figure~\ref{fig:normal_wall_paths}.
\end{definition}

\begin{figure}[ht]
\includegraphics[width=14cm, trim = 2.5cm 9cm 7.5cm 6cm, clip]{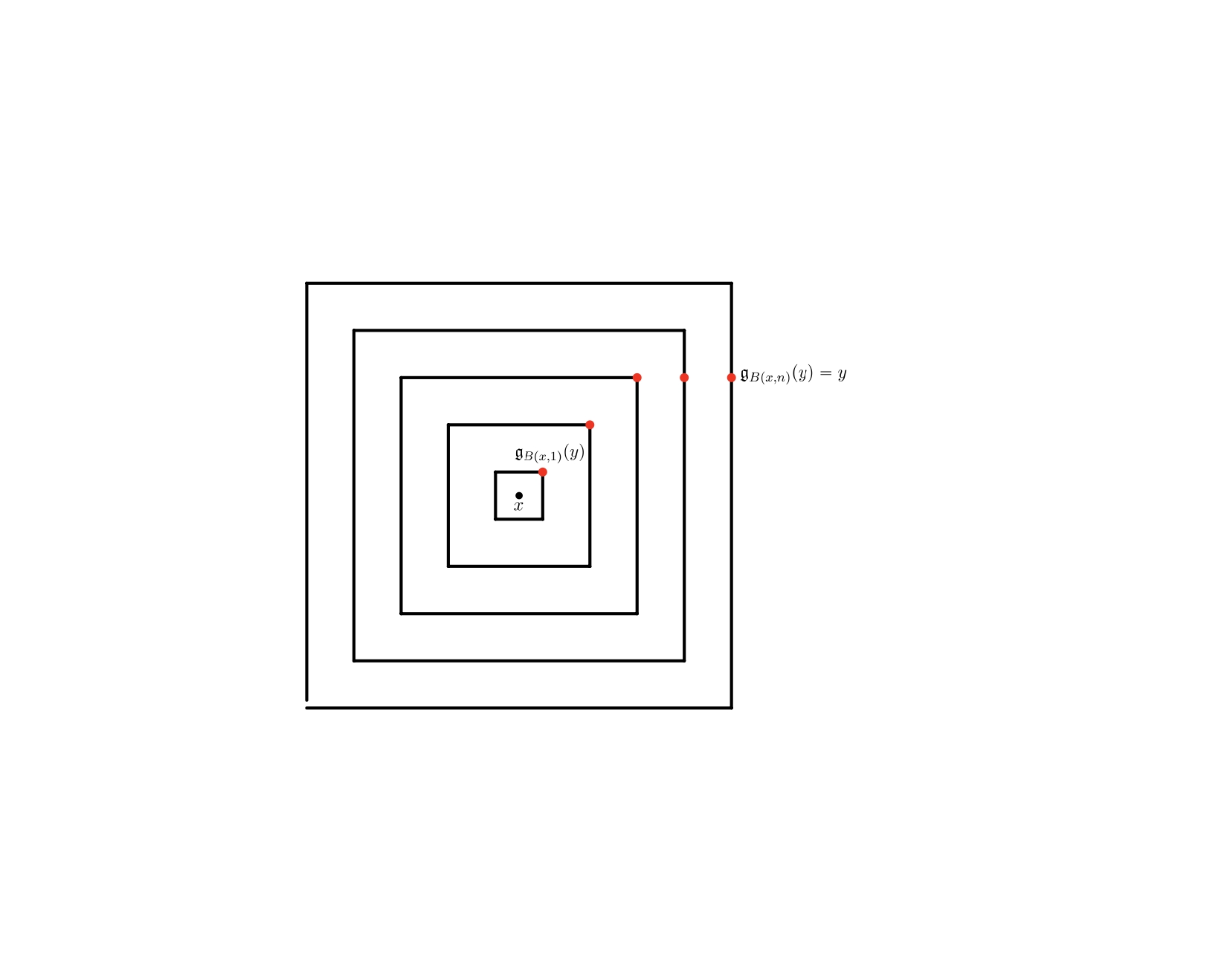}\centering
\caption{$\sigma(x,y)$ is constructed by gating $y$ to integer balls centred on $x$.} \label{fig:normal_wall_paths}
\end{figure}

As with normal cube paths in CAT(0) cube complexes, $\sigma(x,y)$ need not agree with $\sigma(y,x)$.

By the definition of the gate map, $\sigma_{xy}(r)$ is obtained from $y$ by switching the orientations of exactly the walls separating $y$ from $B(x,r)$. Hence if $r_1<r_2<r_3$, then $\sigma_{xy}(r_2)$ is equal to its median with $\sigma_{xy}(r_1)$ and $\sigma_{xy}(r_3)$. In other words, $\sigma_{xy}$ is a \emph{median path} and cannot cross any wall twice. The following lemma can be viewed as a strengthening of this observation.

\begin{lemma} \label{lem:one_sided_consistent}
Let $x,y\in X_\C$. If $z\in[\sigma_{xy}(r),y]$, then $\sigma_{xz}(t)=\sigma_{xy}(t)$ for all $t\le r$. In particular, $\sigma(x,\sigma_{xy}(r))\subset\sigma(x,y)$.
\end{lemma}

\begin{proof}
The point $\sigma_{xy}(r)$ is obtained from $y$ by switching the orientations of exactly those walls that separate $y$ from the ball $B(x,r)$. Thus the fact that $z\in[\sigma_{xy}(r),y]$ means both that the walls separating $z$ from $B(x,r)$ are a subset of those separating $y$ from $B(x,r)$, and that every wall separating $x$ from $\sigma_{xy}(r)$ separates $x$ from $z$. Hence no wall separates $\sigma_{xz}(r)$ from $\sigma_{xy}(r)$, so the two are equal. The rest follows.
\end{proof}

An extremely useful property that a system of chains can have is the ability to combine certain elements of $\C$ to obtain a larger element of $\C$. This allows one to make certain constructions piecewise, which is often necessary when several points are involved. If $c$ is a chain in $P$ with maximal element $h$, then we write $c^+$ to mean the subset $h^+$ of $S$. If $k$ is the minimal element of $c$, then we similarly write $c^-$ for the subset $k^-$ of $S$. If $c_1,c_2\in\C$, then we write $c_2\subset c_1^+$ to mean that every element of $c_2$ has a halfspace contained in $c_1^+$.

\begin{definition}[Gluable] \label{def:gluable}
Let $\C$ be a system of chains. We say that $\C$ is \emph{$m$--gluable} if the following holds. Suppose that $c_1=\{,\dots,h_{-2},h_{-1}\}$ and $c_2=\{h_1,h_2,\dots\}$ are elements of $\C$ such that $c_1\cup c_2$ is a chain, $c_2\subset c_1^+$, and $c_1\subset c_2^-$. There exists a subset $b\subset \{h_{-m},\dots,h_{-1},h_1,\dots,h_m\}$ of consecutive halfspaces, of size at most $m$, such that $(c_1\cup c_2)\ssm b$ is an element of $\C$.
\end{definition}

We shall primarily be interested in $m$--gluable systems with $m\le3$. Note that the only 0--gluable system of chains in $P$ containing all singletons is the set of all chains in $P$.

A \emph{discrete geodesic} in an integer-valued metric space $(Y,\dist)$ is the image of an isometric embedding of an interval in $\Z$. A \emph{$k$--rough geodesic} is the image of a $(1,k)$--quasiisometric embedding of an interval in $\Z$. A \emph{$k$--weak rough geodesic} from $x$ to $y$ is a sequence $(z_r)_{r=0}^n$ with $z_0=x$ and $z_n=y$ such that $|\dist(x,z_r)-r|\le k$ and $|\dist(z_r,y)-(n-r)|\le k$, where $n=\dist(x,y)$. Note that the metric space dual to any dualisable system has integer-valued metric.

\begin{proposition} \label{prop:nwp_rough_geodesic}
If $\C$ is an $m$--gluable system of chains, then normal wall paths are $m$--weak rough geodesics and $3m$--rough geodesics. 

Furthermore, if $x,y\in X_\C$ have $\dist_\C(x,y)=n$, then for any $r\in[0,n]$ we have $r-m\le\dist_\C(x,\sigma_{xy}(r))\le r$ and $n-r\le\dist_\C(\sigma_{xy}(r),y)\le n-r+m$.
\end{proposition}

\begin{proof}
Let us write $z_r=\sigma_{xy}(r)$. As in the proof of Lemma~\ref{lem:convex_balls}, let $Q$ be the set of all $h\in P$ such that there exists some $c\in\C$ with $|c|\ge r$ that separates $x$ from $h$.

First we control $\dist_\C(x,z_r)$. By definition, $z_r$ lies in the $r$--ball about $x$, so we just need to lower-bound $\dist_\C(x,z_r)$. Recall that $z_r$ is obtained from $y$ by switching the orientations of exactly the elements of $Q$ that separate $x$ from $y$. Let $\{h_1,\dots,h_n\}\in\C$ realise $\dist_\C(x,y)$. If $h_{r-m}$ does not separate $z_r$ from $x$, then it separates $z_r$ from $y$, so we must have $h_{r-m}\in Q$. Let $\{k_1,\dots,k_r\}\in\C$ separate $x$ from $h_{r-m}$. Because $\C$ is $m$--gluable, there is then a subset of $\{k_1,\dots,k_r,h_{r-m},\dots,h_n\}$ of size at least $n+1$ that is an element of $\C$. But this contradicts the fact that $\dist_\C(x,y)=n$. Hence $h_{r-m}$ separates $z_r$ from $x$, and in particular $\dist_\C(x,z_r)\ge r-m$.

Now we control $\dist_\C(z_r,y)$. Because $\dist_\C(x,z_r)\le r$, the triangle inequality gives $\dist_\C(z_r,y)\ge n-r$. Now let $\{h'_1,\dots,h'_p\}\in\C$ realise $\dist_\C(z_r,y)$. Every element of $P$ separating $z_r$ from $y$ is in $Q$, so in particular there is some $\{k'_1,\dots,k'_r\}\in\C$ separating $x$ from $h'_1$. By $m$--gluability of $\C$, we get $n=\dist_\C(x,y)\ge r+p-m$, and hence $\dist_\C(z_r,y)=p\le(n-r)+m$. This in particular shows that $(z_r)$ is an $m$--weak rough geodesic. 

%
%
%
To complete the proof, let $r_1<r_2$. By Lemma~\ref{lem:one_sided_consistent}, $\sigma(x,z_{r_2})\subset\sigma(x,y)$, so using the above inequalities, we compute
\begin{align*}
\dist_\C(x,y) &=\, r_2+(n-r_2+m)-m \\
\,&\ge\, \dist_\C(x,z_{r_2})+\dist(z_{r_2},y)-m \\
&\ge\, \dist_\C(x,z_{r_1})+\dist_\C(z_{r_1},z_{r_2})+\dist_\C(z_{r_2},y)-2m \\
&\ge\, r_1+\dist_\C(z_{r_1},z_{r_2})+n-r_2-3m,
\end{align*}
which shows that $\dist_\C(z_1,z_2)\le r_2-r_1+3m$. A similar computation shows that $\dist_\C(z_{r_1},z_{r_2})\ge r_2-r_1-3m$, and hence $\sigma(x,y)$ is a $3m$--rough geodesic.
\end{proof}

A particular case of Proposition~\ref{prop:nwp_rough_geodesic} is that if $\C$ is the set of all chains, then $X_\C$ is discretely geodesic. In other words, it is the vertex set of a graph.

A \emph{bicombing} $\sigma$ of a metric space $Y$ is the choice of a path $\sigma(x,y)$ from $x$ to $y$ for each $x,y\in Y$. This notion is very general, so one often speaks of a bicombing by, say, geodesics, where the $\sigma(x,y)$ are required to be geodesics; or one imposes certain fellow-travelling conditions on the various paths in $\sigma$.

For convenience, if $r>\dist_\C(x,y)$, then we define $\sigma_{x,y}(r)$ to be equal to $y$.

\begin{proposition} \label{prop:bicombing}
If $\C$ is an $m$--gluable system of chains, then normal wall paths form a bicombing of $X_\C$ by rough geodesics. Moreover, for every $r$ we have 
\[
\dist_\C(\sigma_{x_1y_1}(r),\sigma_{x_2y_2}(r)) \,\le\, \max\{\dist_\C(x_1,x_2), \dist_\C(y_1,y_2)\}+3m.
\]
\end{proposition}

\begin{proof}
Let us write $\dist_\C(x_1,x_2)=R_x$ and $\dist_\C(y_1,y_2)=R_y$. After relabelling, we may assume that $\dist_\C(x_1,y_1)\le\dist_\C(x_2,y_2)$, which in turn is at most $\dist_\C(x_1,y_1)+R_x+R_y$ by the triangle inequality. Let us write $z^i_r=\sigma_{x_iy_i}(r)$.

Given $r\le\dist_\C(x_2,y_2)$, let $c\in\C$ realise $\dist_\C(z^1_r,z^2_r)$. Let $c_x\subset c$ be the subchain consisting of all elements that separate $x_1$ from $x_2$, of which there are at most $R_x$. Define $c_y$ similarly, and let $c'=c\ssm(c_x\cup c_y)$, which has $|c\ssm c'|\le R_x+R_y$. Because $z^i_r\in[x_i,y_i]$, every $h\in P$ that separates $z^i_r$ from either $x_i$ or $y_i$ also separates $x_i$ from $y_i$. Thus every element of $c'$ either separates $\{z^1_r,x_1,x_2\}$ from $\{z^2_r,y_2,y_1\}$ or separates $\{z^1_r,y_1,y_2\}$ from $\{z^2_r,x_2,x_1\}$. Because $c$ is a chain, only one of these options can occur. The two cases are similar, so let us assume the former. See Figure~\ref{fig:fellow_travel}. 

\begin{figure}[ht]
\includegraphics[width=7cm]{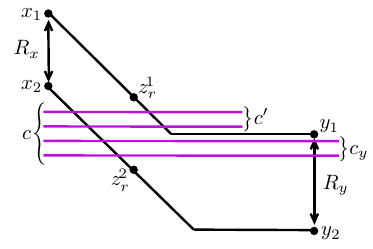}\centering
\caption{The configuration considered in the proof of Proposition \ref{prop:bicombing}.}  \label{fig:fellow_travel}
\end{figure}

Suppose that $c'\ne\varnothing$. Because $c$ is a chain, the sets $c_x$ and $c_y$ are disjoint. Let $h$ be the minimal element of $c'$: it separates $c'\ssm\{h\}$ from $x_2$. We know that $h$ separates $z^1_r$ from $y_1$. The set $H=\{v\in X_\C\,:\,v(h)\ne z^1_r(h)\}$ is gated, so we can consider the point $w=\g_H(z^1_r)\in H$. Every wall separating $z^1_r$ from $w$ separates $z^1_r$ from $H$, by Lemma~\ref{lem:gate_closest_point}. Also, since every element of $c_x$ separates $z^1_r$ from $c'$, it also separates $z^1_r$ from $w$, and from this we see that $\dist_\C(z^1_r,w)\ge|c_x|+1$. 

Note that $w\in[z^1_r,y_1]$, because every wall separating $z^1_r$ from $w$ also separates $z^1_r$ from $y_1$. Because $z^1_r=\sigma_{x_1y_1}(r)$, Lemma~\ref{lem:one_sided_consistent} tells us that $\sigma_{x_1w}(t)=z^1_t$ for $t\le r$. By Proposition~\ref{prop:nwp_rough_geodesic}, we have
\[
\dist_\C(x_1,w) \,\ge\, \dist_\C(x_1,z^1_r)+\dist_\C(z^1_r,w)-m \,\ge\, r+|c_x|+1-2m.
\]
From the fact that $w\in[z^1_r,y_1]\ssm\{z^1_r\}$, we also have that $\dist_\C(x_1,w)\ge r+1$.

Let $b\in\C$ realise $\dist_\C(x_1,w)$. Because $|b|\ge r+1$, at least one element of $b$ separates $z^1_r$ from $w$. It follows that $b$ separates $x_1$ from $H$, so $b\cup(c'\ssm\{h\})\cup c_y$ is a chain separating $x_1$ from $z^2_r$. By $m$--gluability of $\C$, there is some element of $\C$ that can be obtained from $b\cup(c'\ssm\{h\})\cup c_y$ by removing at most $m$ elements. Letting $b_x$ be the subchain of $b$ consisting of all elements separating $x_1$ from $x_2$, which has cardinality at most $R_x$, we therefore have
\begin{align*}
r \,\ge\, \dist_\C(x_2,z^2_r) \,&\ge\, |b\cup(c'\ssm\{h\})\cup c_y| - m - |b_x| \\
	&=\, \dist_\C(x_1,w)+(|c|-|c_x|-1)-m-|b_x| \\
    &\ge\, (r+|c_x|+1-2m)+|c|-|c_x|-1-m-R_x \,=\, r+|c|-3m-R_x.
\end{align*}
This implies that $|c|\le R_x+3m$. We have shown that if $c'\ne\varnothing$, then $\dist_\C(z^1_r,z^2_r)\le R_x+3m$.

\medskip

Now suppose that $c'=\varnothing$ but $|c|>m+\max\{R_x,R_y\}$. As $c=c_x\cup c_y$, each element of $c$ either separates $x_1$ from $x_2$ or separates $y_1$ from $y_2$ (possibly both). Since $|c|>R_y$, at least one element of $c$ lies in $c_x\ssm c_y$. Any such element $h$ must either separate $x_1$ from $\{x_2,y_1,y_2\}$ or $x_2$ from $\{x_1,y_1,y_2\}$. Since the argument is similar in either case, let us assume that the former holds. Given this, let $h\in c$ be the minimal element of $c$: it separates $z^1_r$ from $c\ssm\{h\}$, and necessarily lies in $c_x\ssm c_y$. Similarly to the $c'\ne\varnothing$ case, let $H=\{v\in X\,:\,v(h)\ne z^1_r(h)\}$ and let $w=\g_H(z^1_r)\in H$. 

Because $h$ separates $x_1$ from $y_1$, we have $w\in[z^1_r,y_1]$, and so $\dist_\C(x_1,w)>r$. Let $b\in\C$ realise $\dist_\C(x_1,w)$. Every element of $b$ separates $x_1$ from $H$, and thus at least one separates $z_r$ from $H$. In particular, $b\cup c\ssm\{h\}$ is a chain separating $x_1$ from $z^2_r$. By $m$--gluability of $\C$, there is some element of $\C$ that can be obtained from $b\cup c\ssm\{h\}$ by removing at most $m$ elements. We therefore compute
\begin{align*}
\dist_\C(x_2,z^2_r) \,&\ge\, \dist_\C(x_1,z^2_r)-\dist_\C(x_1,x_2) \\
	&\ge\, (|b|+|c|-1-m)-R_x \\
    &\ge\, (r+1)+(m+1+\max\{R_x,R_y\})-1-m-R_x \,\ge\, r+1,
\end{align*}
which is a contradiction. Hence $|c|\le m+\max\{R_x,R_y\}$ if $c'=\varnothing$.
\end{proof}

One cannot expect stronger properties of the bicombing: for normal cube paths in CAT(0) cube complexes the inequality of Proposition~\ref{prop:bicombing} is optimal, as $m=0$ in that case \cite{nibloreeves:geometry}.

\begin{definition}
A metric space $Y$ is \emph{$k$--coarsely injective} if for each collection of balls $B(x_i,r_i)$ with $r_i+r_j\ge\dist(x_i,x_j)$ for all $i,j$, there is some point $z\in\bigcap B(x_i,r_i+k)$. A graph is \emph{Helly} if this same property holds with $k=0$ when one considers only balls of integer radius centred on vertices.
\end{definition}

\begin{theorem}[Coarsely injective] \label{thm:coarsely_injective}
Let $\C$ be an $m$--gluable system of chains. If $m>0$, then $X_\C$ is $2m$--coarsely injective. If $m=0$, then $X_\C$ is a (not necessarily locally finite) Helly graph.
\end{theorem}

\begin{proof}
Let $\{B_X(x_i,r_i)\}$ be a collection of balls with the property that $r_i+r_j\ge\dist_\C(x_i,x_j)$ for all $i,j$. If $m=0$, assume that each $r_i$ is an integer. (As noted above, if $m=0$ then $\C$ is the set of all chains and $X_\C$ is the vertex set of a graph.) For each $i$, let 
\[
Q_i \,=\, \{w\in P\,:\, \text{ there is some }c\in\C\text{ separating }x_i\text{ from }w\text{ with }|c|\ge r_i+2m\}.
\]
Define an orientation $\phi$ of $\bigcup Q_i$ by pointing each $w\in Q_i$ towards $x_i$. 

We must first show that $\phi$ is well defined. For this, suppose that $w\in Q_i\cap Q_j$. Let $c_i,c_j\in \C$ be given by the definitions of $Q_i$ and $Q_j$, so that $|c_i|\ge r_i+2m$ and $|c_j|\ge r_j+2m$. If $w$ separates $c_i$ from $c_j$, then $c_i\cup\{w\}\cup c_j$ is a chain. But then by twice applying $m$--gluability of $\C$, we find some $b\subset c_i\cup\{w\}\cup c_j$ such that $b\in\C$ and $|b|\ge|c_i|+1+|c_j|-2m\ge r_i+r_j+1$. As $b$ separates $x_i$ from $x_j$, this contradicts our assumption that $r_i+r_j\ge\dist_\C(x_i,x_j)$. Thus $w$ cannot separate $c_i$ from $c_j$, and so cannot separate $x_i$ from $x_j$. This shows that $\phi$ is well defined.

Next we show that $\phi$ is a filter. Suppose that we have walls $w_i\in Q_i$ and $w_j\in Q_j$ such that $w_i^+\subset w_j^+$ and $\phi(w_i)=w_i^+$. Let $c_i$ and $c_j$ be given by the definitions of $Q_i$ and $Q_j$. If $\phi(w_j)=w_j^-$, then $w_j$ separates $c_j$ from $w_i$ and $c_i$. With three applications of $m$--gluability of $\C$, we find some $b\subset c_i\cup\{w_i,w_j\}\cup c_j$ such that $b\in\C$ and $|b|\ge|c_i|+2+|c_j|-3m\ge r_i+r_j+2$. This contradicts the assumption that $r_i+r_j\ge\dist_\C(x_i,x_j)$. Thus $\phi(w_j)=w_j^+$, which shows that $\phi$ is a filter.

Let $z$ be an ultrafilter extending $\phi$. By construction, for each $i$ there is no element of $\C$ of length greater than $r_i+2m$ that separates $z$ from $x_i$. In other words, $z\in\bigcap B_X(x_i,r_i+2m)$, as desired.
\end{proof}

Theorem~\ref{thm:coarsely_injective} also implies the existence of a good bicombing on $X_\C$, thanks to work of Lang \cite{lang:injective}. While less related to the wall structure, it has the advantage of being symmetric and \emph{roughly conical}, and so will not in general be the same as the bicombing by normal wall paths.

\section{Producing hyperbolic spaces} 

One use for the construction in Section~\ref{sec:construction} is to produce hyperbolic spaces that can help study a given space. For that we need a source of negative curvature, which is provided by \emph{$L$--separation}. There may be many different natural hyperbolic spaces that one could produce in this way, and Section~\ref{subsec:unification} gives a way to combine them into a single hyperbolic space subsuming them.

\subsection{Separated systems} \label{subsec:relative_separation}

Recall that walls $h_1$ and $h_2$ are said to cross if all four quarterspaces $h_1^\pm\cap h_2^\pm$ are nonempty in $S$.

\begin{definition}
We say that a system of chains $\C$ is \emph{$L$--separated} if the following holds for every $\{h_1,h_2\}\in\C$. If $c\in\C$ is such that every $h\in c$ crosses both $h_1$ and $h_2$, then $|c|\le L$. 
\end{definition}

\begin{remark}
In many situations, something stronger than $L$--separation holds. Indeed, there may be some larger set $\cal D\subset 2^P$ with the property that for every $\{h_1,h_2\}\in\C$, if every element of $d\in\cal D$ crosses both $h_1$ and $h_2$, then $|d|\le L$. It may even be that one can take $\cal D$ to be the set of all chains in $P$, as is the case in \cite{genevois:hyperbolicities} and \cite{petytsprianozalloum:hyperbolic}. Whilst these stronger properties may be useful, they are not needed for the arguments in this section.
\end{remark}

We make a simple observation regarding elements of a gluable, separated system.

\begin{lemma}[Gluing] \label{lem:gluing}
Suppose that $\C$ is an $L$--separated, $m$--gluable system of chains. If $c=\{\dots,h_{-1}\}$ and $c'=\{k_1,\dots\}$ are elements of $\C$ such that 
\[
h_{-1}^+\cap k_j^\pm\ne\varnothing \quad\text{and}\quad h_i^\pm\cap k_1^-\ne\varnothing,
\]
for all $i,j$, then there is a subset $b\subset\{h_{-m-1},\dots,h_{-1},k_1,\dots, k_{L+m}\}$ of cardinality at most $L+m+1$ such that $c\cup c'\ssm b\in\C$.
\end{lemma}

\begin{proof}
The wall $h_{-2}$ cannot cross $k_{L+1}$, for then so would $h_{-1}$, contradicting $L$--separation of $c$. Hence $c\ssm\{h_{-1}\}$ and $c'\smallsetminus\{k_1,\dots,k_L\}$ satisfy the hypotheses of the $m$--gluability assumption on $\C$.
\end{proof}

With this in hand, we turn to showing that $X_\C$ is hyperbolic.

\begin{definition}[Cross chain] \label{def:cross_chain}
Let $\C$ be a system of chains. A $\times$--chain (cross chain) for $x_1,x_2,x_3,x_4\in X_\C$ is a subset $\chi\subset P$ with a decomposition $\chi=\bigsqcup_{i=1}^4\chi_i$ such that every element of $\chi_i$ separates $x_i$ from $\{x_{i+1},x_{i+2},x_{i+3}\}$; and such that $\chi_i\cup\chi_j\in\C$ for all $i,j$.

We say that a $\times$--chain is \emph{maximal} if it is not a proper subset of another $\times$--chain.
\end{definition}

Note that the notion of $\times$--chain makes sense even when $\C$ is not separated. However, it is at its most useful in this setting, because maximal $\times$--chains can be effectively compared with chains realising the distances between the defining points.

\begin{lemma} \label{lem:cross_vs_L}
Let $\C$ be an $L$--separated, $m$--gluable system of chains. If $c\in\C$ realises $\dist_\C(x_1,x_2)$ and $\chi$ is a maximal $\times$--chain for $x_1,x_2,x_3,x_4\in X_\C$, then
\[
|\chi_1|+|\chi_2|-2(L+m+1) \,\le\, |\{h\in c\,:\, h\text{ does not separate }x_3\text{ from }x_4\}|
    \,\le\, |\chi_1|+|\chi_2|+4(L+m+1).
\]
\end{lemma}

\begin{proof}
Let $r=|\{h\in c\,:\, h\text{ does not separate }x_3\text{ from }x_4\}|$. Let $c_1$ be the subset of $c$ consisting of those elements that do not separate $x_2$ from either $x_3$ or $x_4$, and define $c_2$ similarly, so that $r=|c_1|+|c_2|$. See Figure~\ref{fig:cross_chain}. By Lemma~\ref{lem:gluing}, after removing at most $L+m+1$ elements of each of $c_1$, $c_2$, $\chi_3$, and $\chi_4$, we obtain a $\times$--chain. Maximality of $\chi$ then implies that $r\le|\chi_1|+|\chi_2|+4(L+m+1)$. 

\begin{figure}[ht]
\includegraphics[width=12cm, trim = 0cm 6.5cm 0cm 4.5cm, clip]{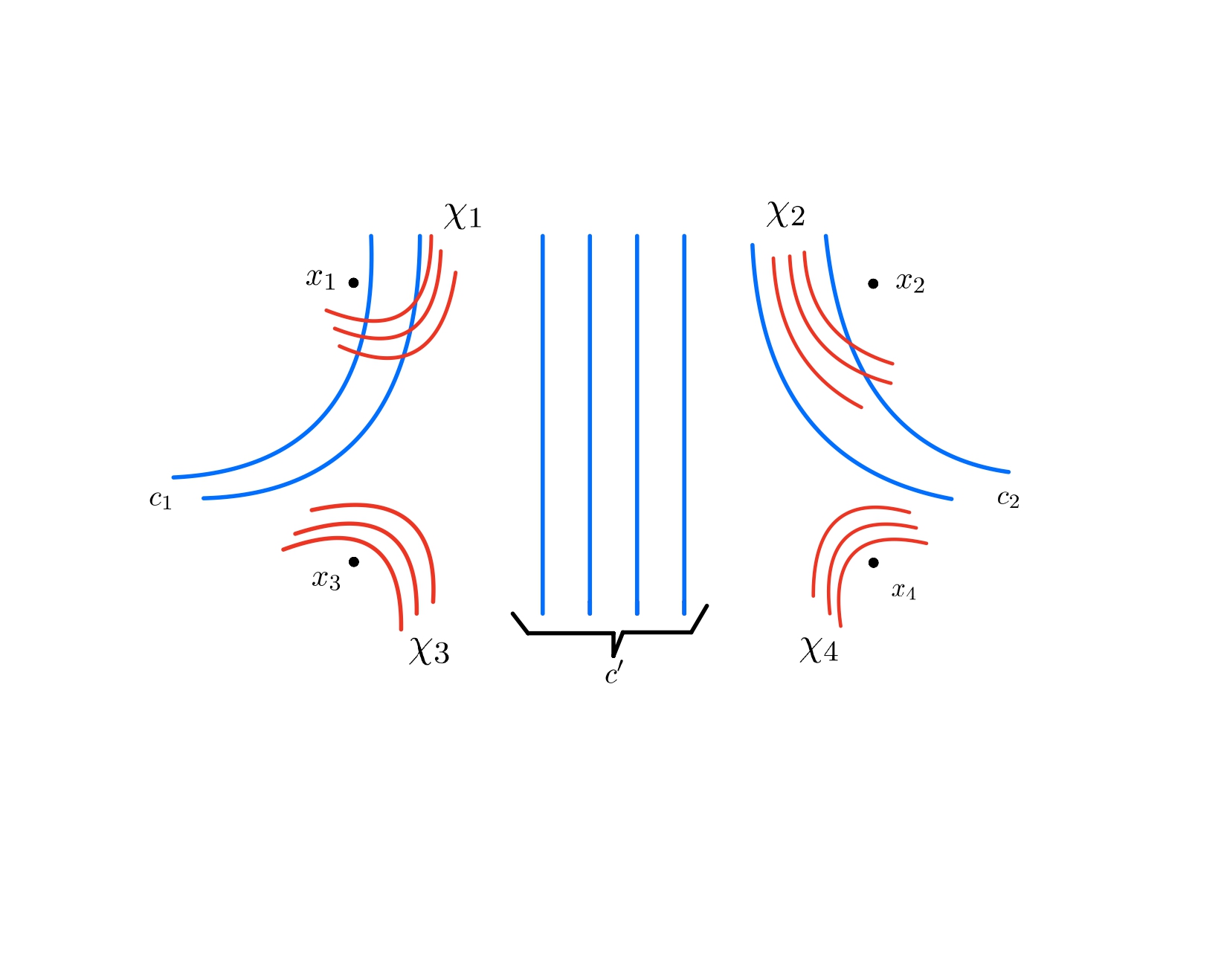}\centering
\caption{Lemma~\ref{lem:cross_vs_L}. $\,\chi$ is a union of ``corners''.} \label{fig:cross_chain}
\end{figure}

On the other hand, let $c'=c\smallsetminus(c_1\cup c_2)$ be the subset of $c$ consisting of those elements that separate $x_3$ from $x_4$. After relabelling, every element of $c'$ separates $\{x_1,x_3\}$ from $\{x_2,x_4\}$. From Lemma~\ref{lem:gluing} and the fact that $\chi_1\cup\chi_2\in\C$, there is some $b\subset\chi_1\cup c'\cup\chi_2$ such that $b\in\C$ separates $x_1$ from $x_2$ and $|b|\ge|\chi_1|+|c'|+|\chi_2|-2(L+m+1)$. Since $c$ realises $\dist_\C(x_1,x_2)$, we have $|b|\le|c|$, so $r\ge|\chi_1|+|\chi_2|-2(L+m+1)$.
\end{proof}

A metric space $(Y,\dist)$ is said to be \emph{four-point hyperbolic} with constant $\delta$ if for every $x_1,x_2,x_3,x_4\in Y$ we have $\dist(x_1,x_2)+\dist(x_3,x_4)\le\delta+\max\{\dist(x_1,x_3)+\dist(x_2,x_4),\,\dist(x_1,x_4)+\dist(x_2,x_3)\}$.

\begin{proposition} \label{prop:four-point}
If $\C$ is an $L$--separated, $m$--gluable set of chains, then the metric $\dist_\C$ is four-point hyperbolic with constant $30(L+m+1)$.
\end{proposition}

\begin{proof}
Given $x_1,x_2,x_3,x_4\in X_\C$, let $\chi$ be a maximal $\times$--chain for them. If $\dist_\C(x_i,x_j)$ and $|\chi_i|+|\chi_j|$ differ by at most $4L+4m+5$ for all $i,j$ then we are done. Otherwise we can relabel so that $\dist_\C(x_1,x_2)\ge|\chi_1|+|\chi_2|+4L+4m+6$. Let $c_{12}\in\C$ realise $\dist_\C(x_1,x_2)$, and let $c'_{12}$ be the subchain consisting of those elements that separate $x_3$ and $x_4$. By Lemma~\ref{lem:cross_vs_L} we have $|c'_{12}|\ge 2$. Perhaps after relabelling $x_3$ and $x_4$, the chain $c'_{12}$ separates $\{x_1,x_3\}$ from $\{x_2,x_4\}$. 

Now let $c_{13}\in\C$ realise $\dist_\C(x_1,x_3)$. If $|c_{13}|-|\chi_1|-|\chi_3|>5L+4m+5$, then Lemma~\ref{lem:cross_vs_L} implies the existence of a subchain $c'_{13}$ of length $L+1$ whose elements all separate $\{x_1,x_4\}$ from $\{x_2,x_3\}$, and hence cross every element of $c'_{12}$, which contradicts the assumption that $\C$ is $L$--separated. A similar argument shows that if $c_{24}\in\C$ realises $\dist_\C(x_2,x_4)$, then $|c_{24}|-|\chi_2|-|\chi_4|\le5L+4m+5$. In the other direction, we can apply Lemma~\ref{lem:gluing} to $\chi_1$ and $\chi_3$, and similarly to $\chi_2$ and $\chi_4$. This shows that $\dist_\C(x_1,x_3)+\dist_\C(x_2,x_4)$ and $|\chi|$ differ by at most $2(5L+4m+5)+2(L+m+1)$.

Now let $c_{34}\in\C$ realise $\dist_\C(x_3,x_4)$, and let $c'_{34}$ be the subchain consisting of those elements that separate $x_1$ from $x_2$. By applying Lemma~\ref{lem:gluing} twice to each of the pairs $(c'_{12},c_{34}\ssm c'_{34})$ and $(c'_{34},c_{12}\ssm c'_{12})$, one finds that $|c'_{12}|$ and $|c'_{34}|$ can differ by at most $2(L+m+1)$. Combining with Lemma~\ref{lem:cross_vs_L}, one can then compute that $\dist_\C(x_1,x_2)+\dist_\C(x_3,x_4)$ differs from $|\chi|+2|c'_{12}|$ by at most $14(L+m+1)$.

Finally consider the analogously defined $c_{14}$, $c'_{14}$, $c_{23}$, and $c'_{23}$. Both of these are treated the same, so let us just consider the former. Because $c'_{12}$ separates $x_1$ from $x_4$, we can use Lemma~\ref{lem:gluing} to see that $|c'_{14}|\ge|c'_{12}|-2(L+m+1)$. If $|c'_{14}|\le L$, then certainly $|c'_{14}|\le|c'_{12}|+L$, so suppose otherwise. If $c'_{14}$ separates $x_1$ from $x_3$ then its elements all cross $c'_{12}$, which contradicts the fact that $c'_{12}\in\C$. Hence $c'_{14}$ separates $x_1$ from $x_2$. By Lemma~\ref{lem:gluing} we thus have $|c'_{12}|\ge|c'_{14}|-2(L+m+1)$. We have shown that $|c'_{14}|$ differs from $|c'_{12}|$ by at most $2(L+m+1)$. 

By a similar argument to above, Lemma~\ref{lem:cross_vs_L} now lets us show that $\dist_\C(x_1,x_4)+\dist_\C(x_2,x_3)$ differs from $|\chi|+2|c'_{12}|$ by at most $16(L+m+1)$. 
The result follows by combining the various estimates.
\end{proof}

Together with Proposition~\ref{prop:nwp_rough_geodesic}, this shows that $X_\C$ is hyperbolic in the appropriate sense for non-geodesic spaces (\emph{cf.} \cite[Prop.~A.2]{petytsprianozalloum:hyperbolic}).

\begin{corollary} \label{cor:rough_geodesic_hyperbolic}
If $\C$ is a separated, gluable system of chains, then $X_\C$ is a roughly geodesic hyperbolic space.
\end{corollary}

Since four-point hyperbolicity passes to subsets, $S\subset X$ is four-point hyperbolic. However, four-point hyperbolicity is rather weak on its own, so it is interesting to know when $(S,\dist_\C)$ is roughly geodesic. One can also ask the \emph{a priori} stronger question of when $S$ is coarsely dense in $X$. The following proposition shows that these two properties are in fact equivalent.

\begin{proposition}[Coarsely dense] \label{prop:sufficient_dense_L}
Let $\C$ be an $L$--separated, $m$--gluable system of chains. If $(S,\dist_\C)$ is $k$--weakly roughly geodesic, then $S$ is $(3k+4(L+m+1))$--coarsely dense in $X_\C$.
\end{proposition}

\begin{proof}
Let $x\in X_\C$. Pick an arbitrary point $s\in S$, and let $c\in\C$ realise $\dist_\C(s,x)$. If $|c|\le3k+4(L+m+1)$, then we are done, so assume otherwise. Because $c$ is a finite chain, there is some $t\in S$ that is separated from $s$ by every element of $c$. Let $b\in\C$ realise $\dist_\C(t,x)$. Again, we can assume that $|b|>3k+4(L+m+1)$. Because $S$ is $k$--weakly roughly geodesic, there exists $s'\in S$ such that $\dist_\C(s,s')\le|c|+k$ and $\dist_\C(s',t)\le\dist_\C(s,t)-|c|+k\le|b|+k$. See Figure~\ref{fig:density_hyp}. 


No element of $c$ separates $x$ from $t$, whereas every element of $b$ separates $x$ from $t$. Therefore, by Lemma~\ref{lem:gluing} there is some $a\subset b\cup c$ with $a\in\C$ and $|a|\ge|b|+|c|-L-m-1$. In particular, the fact that $\dist_\C(s',t)\le|b|+k$ implies that all but at most $k+L+m+1$ elements of $c$ must separate $s'$ from $s$. Let $c_s\subset c$ consist of those elements that do separate $s$ from $s'$.

Let $c'\in\C$ realise $\dist_\C(s',x)$. Consider the subset $c'_s\subset c'$ consisting of those elements that do not separate $s$ from $s'$. We can apply Lemma~\ref{lem:gluing} to the pair $(c_s,c'_s)$ to obtain an element of $\C$ that separates $s$ from $x$ and has cardinality at least $|c_s|+|c'_s|-L-m-1$. Because $c$ realises $\dist_\C(s,x)$, this implies that $|c'_s|\le k+2(L+m+1)$. 

\begin{figure}[ht]
\includegraphics[width=10cm, trim = 4cm 3cm 4cm 3.5cm, clip]{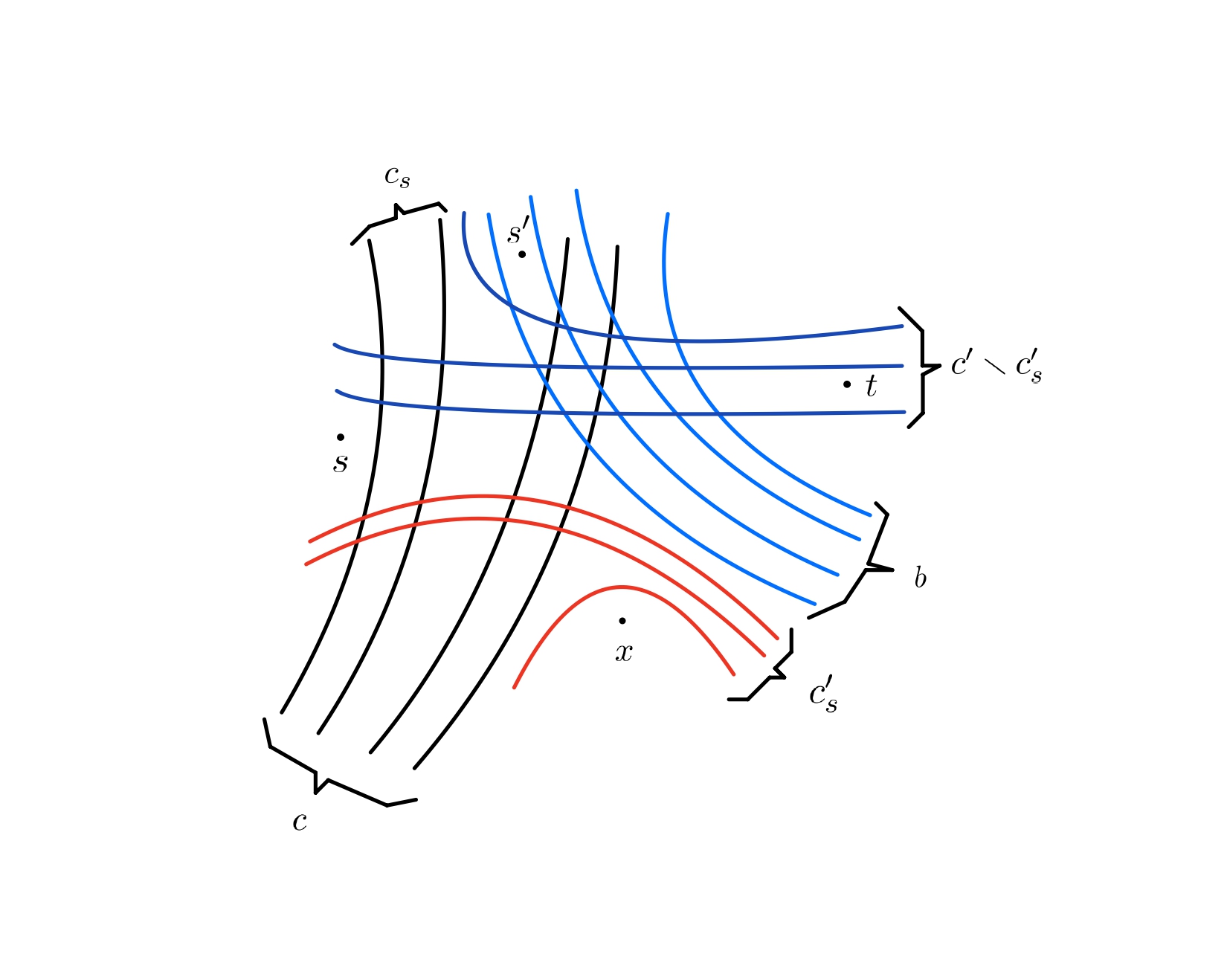}
\caption{Proposition~\ref{prop:sufficient_dense_L}. The chain $c$ realises $\dist_\C(s,x)$, yielding $t$, and then $b$ realises $\dist_\C(t,x)$. Weak rough geodesicity gives $s'$, and $c'$ realises $\dist_\C(s',x)$.} \label{fig:density_hyp}
\end{figure}

Finally, consider $c'\ssm c'_s\in\C$. Every element of this chain separates $\{s,x\}$ from $s'$, so we can apply Lemma~\ref{lem:gluing} to the pair $(c_s,c'\smallsetminus c'_s)$ to find an element of $\C$ that separates $s$ from $s'$ and has cardinality at least $|c_s|+|c'\ssm c'_s|-L-m-1$. From the fact that $\dist_\C(s,s')\le|c|+k$, we deduce that $|c'\smallsetminus c'_s|\le2(k+L+m+1)$. We therefore have $|c'|\le3k+4(L+m+1)$, completing the proof.
\end{proof}

In practice it can often be verified from additional information about the pair $(S,P)$ that $(S,\dist_\C)$ is weakly roughly geodesic. Proposition~\ref{prop:sufficient_dense_L} then shows that working with $(S,\dist_\C)$ is largely the same as working with $(X,\dist_\C)$, though the latter has the advantage that it has better fine properties, as illustrated by Section~\ref{sec:properties}.

The following example shows how $S$ can fail to be weakly roughly geodesic or coarsely dense in $(X,\dist_\C)$.

\begin{example}
Let $S$ be a copy of the real line, and identify $S$ with a horocycle in the hyperbolic plane $\mathbb H^2$. Let $P$ be the set of partitions obtained from the set of all $\mathbb H^2$ geodesics intersecting $S$ by viewing each such geodesic as dividing $\mathbb H^2$ into two halves. Let $\C$ be the set of chains such that each pair of defining geodesics is at $\mathbb H^2$--distance at least 1. With this system, the metric $\dist_\C$ on $S$ is quasiisometric to the subspace metric from $\mathbb H^2$.
\end{example}

\subsection{Graded systems} \label{subsec:unification}

In \cite{petytsprianozalloum:hyperbolic}, a key step in producing the curtain model of a CAT(0) space is the resolution of an infinite sequence of increasingly informative hyperbolic models into a single hyperbolic space. The goal of this section is to give that procedure a more general framing, which will be useful in Section~\ref{sec:contracting}. 

\begin{definition}[Graded system] \label{def:graded_system}
A sequence $(\C_R)$ of dualisable systems on $(S,P)$ is said to be a \emph{graded system} on $(S,P)$ if there are numbers $L_R\ge1$, $m_R\ge0$ and a sequence $(\kappa_R)\subset(0,\infty)$ such that the following hold.
\begin{itemize}
\item   $\C_R\subset\C_{R+1}$ for all $R$.
\item   $\C_R$ is an $L_R$--separated, $m_R$--gluable system of chains for each $R$.
\item   For each $s,t\in S$ there exists $M_{st}$ such that for every sequence $(c^R)$ with $c^R\in\C_R$ separating $s$ from $t$, we have $\sum_{R=1}^\infty\kappa_R|c^R|\le M_{st}$.
\end{itemize}
\end{definition}

Note that if we assume $L_R$ to be minimal such that $\C_R$ is $L_R$--separated, then the condition that $\C_R\subset\C_{R+1}$ implies that $L_{R+1}\ge L_R$, but we could potentially have $m_{R+1}<m_R$. 

Recall that $\hat X$ denotes the set of all ultrafilters defined by $P$. If $(\C_R)$ is a graded system, then the constructions of Section~\ref{sec:construction} produce a sequence of metric spaces $X_R=(X_{\C_R},\dist_{\C_R})$, which will be hyperbolic by the results of Section~\ref{subsec:relative_separation}. However, the sets $X_R$ in general will be getting smaller (setwise, whilst the metric grows) as $R$ increases, because as $\C_R$ gets larger the condition that $\dist_\C(s,x)<\infty$ becomes more stringent. We unify these spaces as follows.

\ubsh{The graded dual}
Given a graded system $(\C_R)$ on $(S,P)$, let $(\lambda_R)$ be a sequence of positive numbers such that $\lambda_R\le\kappa_R$ and $\sum_{R=1}^\infty\lambda_R(L_R+m_R+1)=\Lambda<\infty$. Consider the extended metric on $\hat X\times\hat X$ given by
\[
\Dist(x,y) \,=\, \sum_{R=1}^\infty\lambda_R\dist_R(x,y).
\]
The \emph{graded dual} of $S$ with respect to $(\C_R)$ is the space $X=\{x\in\hat X\,:\,\Dist(s,x)<\infty\text{ for all }s\in S\}$, equipped with the metric $\Dist$.
\uesh

The final assumption in the definition of a graded system is what ensures that $\Dist(s,t)<\infty$ for all $s,t\in S$, so every principal ultrafilter lies in $X$. 

\begin{remark} \label{rem:X_smaller}
If $x\in X$ then certainly $x\in X_R$ for every $R$, but it can happen that $X$ is a strict subset of $\bigcap_{R=1}^\infty X_R$. For instance, if $S$ is the region of the euclidean plane bounded between the $x$-axis and a sufficiently slowly growing sublinear function, $P$ is the set of partitions induced by CAT(0) curtains, and $\C_R$ is as in \cite{petytsprianozalloum:hyperbolic}, then any ultrafilter extending the filter defined by the unique CAT(0) boundary point of $S$ lies in every $X_R$ but not in $X$.
\end{remark}

\begin{definition}[Perichain]
Given a graded system $(\C_R)$ on $(S,P)$, a \emph{perichain} $\bar c$ is a choice of element $c^R\in\C_R$ for each $R$. A $\times$--perichain is a choice of $\times$--chain $\chi^R$ from $\C_R$ for each $R$.
\end{definition}

For a perichain $\bar c$, write $\|\bar c\|=\sum\lambda_R|c^R|$. Note that if $c^R$ realises $\dist_R(x,y)$ for all $R$, then $\|\bar c\|=\Dist(x,y)$. 

\begin{lemma} \label{lem:cross_corners}
Let $(\C_R)$ be a graded system. If $\bar c$ is a perichain realising $\Dist(x_1,x_2)$ and $\bar\chi$ is a maximal $\times$--perichain for $x_1,x_2,x_3,x_4\in X$, then 
\[
\|\bar\chi_1\|+\|\bar\chi_2\|-2\Lambda \,\le\, \|\bar c'\| \,\le\, \|\bar\chi_1\|+\|\bar\chi_2\|+4\Lambda,
\]
where ${c'}^R\subset c^R$ consists of all elements of $c^R$ not separating $x_3$ from $x_4$.
\end{lemma}

\begin{proof}
Apply Lemma~\ref{lem:cross_vs_L} for each $R$.
\end{proof}

The proof of the following proposition is similar to that of Proposition~\ref{prop:four-point}, but one has to ensure some compatibility along the terms of the graded system.

\begin{proposition} \label{prop:graded_hyperbolic}
If $(\C_R)$ is a graded system, then $(X,\Dist)$ is four-point hyperbolic, with constant $16\Lambda$.
\end{proposition}

\begin{proof}
Fix $x_1,x_2,x_3,x_4\in X$, and let $\bar\chi$ be a maximal $\times$-perichain for them. For each distinct $i,j,k$, let $S^R(1i|jk)$ be the cardinality of a maximal element of $\C_R$ separating $\{x_1,x_i\}$ from $\{x_j,x_k\}$. Observe that since $\C_R$ is $L_R$--separated, at most one of the $S^R(1i|jk)$ can be greater than $2L_R$ for any given $R$.

Suppose that $i$ is such that there is no value of $R$ for which $S^R(1i|jk)>2L_R$. Let $\bar c_{jk}$ be a perichain realising $\Dist(x_j,x_k)$, and let $\bar b_{jk}$ be the subperichain consisting of all elements of $\bar c_{jk}$ that either separate $x_j$ from all three other $x_l$ or separate $x_k$ from all three other $x_l$. We have $\|\bar b_{jk}\|\ge\|\bar c_{jk}\|-2\Lambda$. According to Lemma~\ref{lem:cross_corners}, this gives
\begin{align} \label{eqn:no_relevant_L}
\|\bar\chi_j\|+\|\bar\chi_k\|-2\Lambda \,\le\, \Dist(x_j,x_k) \,\le\, \|\bar\chi_j\|+\|\bar\chi_k\|+6\Lambda. 
\end{align}
We similarly have the corresponding inequality for $\Dist(x_1,x_i)$.

If there is no $i$ for which there is some value of $R$ for which $S^R(1i|jk)>2L_R$, then the inequality~\eqref{eqn:no_relevant_L} holds for all pairs $(x_j,x_k)$. This implies the four-point inequality with constant $8\Lambda$.

Otherwise, let $R_0$ be minimal such that some $S^{R_0}(1i|jk)>2L_{R_0}$. After relabelling, we may assume that $i=2$. In other words, there is an element of $\C_{R_0}$ of length greater than $2L_{R_0}$ that separates $x_1$ and $x_2$ from $x_3$ and $x_4$. Because $\C_{R_0}\subset\C_R$ for all $R\ge R_0$, we can apply $L_R$--separation to a pair of walls in a chain realising $S^{R_0}(12|34)$ to see that if $R\ge R_0$ then $\max\{S^R(13|24),S^R(14|23)\}\le L_R$. In particular, minimality of $R_0$ means that $\max\{S^R(13|24),S^R(14|23)\}\le2L_R$ for all $R$. Thus inequality~\eqref{eqn:no_relevant_L} holds for the pairs $(x_1,x_2)$ and $(x_3,x_4)$.

We now bound the distance between each of the other four pairs of points. The argument is the same for each, so for ease of notation we shall work with the pair $(x_1,x_3)$. By applying Lemma~\ref{lem:gluing} twice, we see that $\dist_R(x_1,x_3)\ge|\chi_1^R|+|\chi_3^R|+S^R(12|34)-2(L_R+m_R+1)$, and hence
\[
\|\bar\chi_1\|+\|\bar\chi_3\|+\sum_{R=1}^\infty\lambda_R S^R(12|34) -2\Lambda \,\le\, \Dist(x_1,x_3).
\]

On the other hand, let $c^R_{13}\in\C_R$ realise $\dist_R(x_1,x_3)$, and let $b^R_{13}$ be the subset consisting of those walls that do not separate $x_2$ from $x_4$. By Lemma~\ref{lem:gluing}, after removing at most $L_R+m_R+1$ elements from each of $\chi^R_2$ and $\chi^R_4$ and at most $2(L_R+m_R+1)$ elements of $b^R_{13}$, we are left with a $\times$-chain. By maximality of $\bar\chi$, it follows that $|\chi^R|\ge|\chi^R_2|+|\chi^R_4|+|b^R_{13}|-4(L_R+m_r+1)$, or in other words that 
\[
|b^R_{13}|\,\le\,|\chi^R_1|+|\chi^R_3|+4(L_R+m_R+1).
\]
Every element of $c^R_{13}\ssm b^R_{13}$ either separates $\{x_1,x_2\}$ from $\{x_3,x_4\}$, or separates $\{x_1,x_4\}$ from $\{x_2,x_3\}$. But $S^R(14|23)\le2L_R$, and so $|c^R_{13}\ssm b^R_{13}|\le S^R(12|34)+2L_R$. Hence for every $R$ we have 
\[
\dist_R(x_1,x_3) \,\le\, |b^R_{13}|+S^R(12|34)+2L_R \,\le\, |\chi_1^R|+|\chi_3^R|+S^R(12|34)+6(L_R+m_R+1).
\]
In tandem with the above lower bound on $\Dist(x_1,x_3)$, summing this yields
\[
\|\bar\chi_1\|+\|\bar\chi_3\|+\sum_{R=1}^\infty \lambda_RS^R(12|34)-2\Lambda \,\le\, \Dist(x_1,x_3) \,\le\, \|\bar\chi_1\|+\|\bar\chi_3\|+\sum_{R=1}^\infty \lambda_RS^R(12|34)+6\Lambda.
\]
As noted, the same argument holds for each of the pairs $(x_1,x_4)$, $(x_2,x_3)$, and $(x_2,x_4)$. The four-point inequality for $x_1,x_2,x_3,x_4$ follows, with constant $16\Lambda$.
\end{proof}

\begin{proposition} \label{prop:graded_wrg}
If $(\C_R)$ is a graded system such that the sequence $(m_R)$ is bounded above by some number $M$, then $(X,\Dist)$ is $4M\Lambda$--weakly roughly geodesic.
\end{proposition}

\begin{proof}
Given $x,y\in X$, let $R_0$ be sufficiently large so that $\sum_{R>R_0}\lambda_R\dist_R(x,y)\le\Lambda$. 
Let $\sigma_{xy}$ be the normal wall path in $(X_{R_0},\dist_{R_0})$ from $x$ to $y$, as constructed in Section~\ref{sec:properties}, and write $z_i=\sigma_{xy}(i)$. We do not know exactly how $\Dist$ compares with $\dist_{R_0}$ on $X$, so this path could \emph{a priori} fail to be a rough geodesic of $X$, but we shall nonetheless use it to show that $X$ is weakly roughly geodesic. This will be a consequence of the following observations, which are built into the definition of the $z_r$.
\begin{itemize}
\item   Every element of $P$ separating $x$ from $z_r$ separates $x$ from $\{z_r,y\}$.
\item   Every element of $P$ separating $z_r$ from $z_{r+1}$ separates $\{x,z_r\}$ from $\{z_{r+1},y\}$.
\item   Every element of $P$ separating $z_r$ from $y$ separates $\{x,z_r\}$ from $y$.
\end{itemize}

First observe that these imply that $z_r$ is indeed an element of $X$, because we must have $\Dist(x,z_r)\le\Dist(x,y)<\infty$, and $x$ lies at finite distance from $S$. In order to establish the proposition, it suffices to bound $\Dist(z_r,z_{r+1})$ and $\Dist(x,z_r)+\Dist(z_r,y)-\Dist(x,y)$. 

For the latter, note that for each $R$ we can apply Lemma~\ref{lem:gluing} to find that $\dist_R(x,y)\ge\dist_R(x,z_r)+\dist_R(z_r,y)-L_R-m_R-1$. Summing over $R$, we find that $\Dist(x,y)\ge\Dist(x,z_r)+\Dist(z_r,y)-\Lambda$, as desired.

It remains to bound $\Dist(z_r,z_{r+1})$. We know from Proposition~\ref{prop:nwp_rough_geodesic} that $\dist_{R_0}(z_r,z_{r+1})\le1+2m_{R_0}\le1+2M$. Because $\C_R\subset\C_{R+1}$ for all $R$, this implies that $\dist_R(z_r,z_{r+1})\le1+2M$ for all $R\le R_0$. On the other hand, if $R>R_0$, then by the above observations we have that $\dist_R(x,y)\ge\dist_R(z_r,z_{r+1})$. We can therefore compute
\begin{align*}
\Dist(z_r,z_{r+1}) \,&=\, \sum_{R=1}^{R_0}\lambda_R\dist_R(z_r,z_{r+1}) \,+\, \sum_{R>R_0}\lambda_R\dist_R(z_r,z_{r+1}) \\
&\le\, \sum_{R=1}^{R_0}\lambda_R(1+2M) \,+\, \sum_{R>R_0}\lambda_R\dist_R(x,y) \,\le\, 4M\Lambda. \qedhere
\end{align*}
\end{proof}

\begin{remark} \label{rem:graded_roughly_geodesic}
It follows that in the situation of Proposition~\ref{prop:graded_wrg}, $(X,\Dist)$ is a coarsely injective, hence roughly geodesic, hyperbolic space \cite[Prop.~A.2]{petytsprianozalloum:hyperbolic}. Moreover, in the proof of Proposition~\ref{prop:graded_wrg}, we showed that given $x,y\in X$, there is some $R$ such that the normal wall path in $X_R$ from $x$ to $y$ is a uniform weak rough geodesic in $X$. We therefore have \emph{a posteriori} that these paths are uniform rough geodesics in $X$.
\end{remark}

As in Section~\ref{subsec:relative_separation}, it is also desirable to know that the original set $S$ is a roughly geodesic hyperbolic space with respect to the metric $\Dist$. Interestingly, in contrast to Proposition~\ref{prop:sufficient_dense_L}, it is uncertain whether $(S,\Dist)$ being roughly geodesic is equivalent to its being dense in $(X,\Dist)$.





\part{Applications}

\section{The hyperbolic core} \label{sec:contracting}

The goal of this section is to construct, for a given geodesic space $S$, a hyperbolic space $X$ with the property that every \emph{strongly contracting} quasigeodesic in $S$ is witnessed as a (parametrised) quasigeodesic in $X$. Any quasigeodesic in $S$ can be perturbed to one whose image is a closed subset of $S$, so we shall always assume that quasigeodesics are closed. In particular, this means that every point in $S$ has a nonempty set of closest points in each quasigeodesic.

\begin{definition}[Strongly contracting]
Let $\alpha$ be a quasigeodesic in a geodesic space $S$. Given $x\in S$, let $\pi_\alpha(x)$ be the (nonempty) set of closest points in $\alpha$ to $x$. We say that $\alpha$ is \emph{$D$--strongly contracting} if for any ball $B\subset S$ disjoint from $\alpha$ we have $\diam(\pi_\alpha(B))<D$.
\end{definition}

If $\alpha$ is strongly contracting, then although $\pi_\alpha$ may send points to sets of cardinality greater than one, the image of a point has bounded diameter. For a set $I\subset\alpha$, we write $\pi^{-1}(I)$ to mean the set of all points $x\in S$ such that $\pi_\alpha(x)\subset I$.

For the remainder of this section, fix a geodesic space $S$. In order to apply the methods of the previous sections, we need two things: a good set of bipartitions of $S$, and a choice of which collections of bipartitions should be counted. As an intermediate step towards these goals, we consider a natural collection of subspaces of $S$, which we call \emph{curtains}, in analogy with \cite{petytsprianozalloum:hyperbolic}.

\begin{definition}[Curtains]
Suppose that $\alpha$ is a $D$--strongly contracting geodesic of length $20D$, with $D\ge1$. A \emph{curtain} dual to $\alpha$ is a set $\pi_\alpha^{-1}(I)$, where $I\subset\alpha$ is a subgeodesic of length $10D$ not containing any endpoint of $\alpha$. 
\end{definition}

Let us write $\cur S$ for the set of all curtains in $S$. That is, $\cur S$ contains every curtain dual to every $D$--strongly contracting geodesic in $S$, for every $D\ge1$. Note that $\cur S$ may very well be empty: this happens if $S$ has no strongly contracting geodesics that are long compared to their strong-contracting constant, for instance in the euclidean plane. 

The fact that we only consider strongly contracting \emph{geodesics} in the construction of curtains is justified by the following well-known lemma.

\begin{lemma} \label{lem:contracting_iff_pass_close} 
A $q$-quasigeodesic $\alpha$ is strongly contracting if and only if there is a constant $D'$ such that the following hold.
\begin{itemize}
\item   For every $x\in S$ we have $\diam\pi_\alpha(x)\le D'$.
\item   For every $x,y\in S$ with $\dist(\pi_\alpha(x),\pi_\alpha(y))>20D'$, and for every geodesic $\beta$ from $x$ to $y$, every subpath of $\alpha$ from $\pi_\alpha(x)$ to $\pi_\alpha(y)$ lies in the $5D'$--neighbourhood of $\beta$.
\end{itemize}
\end{lemma}

\begin{proof}
The forward direction is given by \cite[Lem.~4.5]{charneysultan:contracting}. For the reverse direction suppose that the two conditions hold for the $q$--quasigeodesic $\alpha$, and let $B=B(x,r)$ be a ball in $S$ that is disjoint from $\alpha$. If $\diam\pi_\alpha(B)>50D'$, then by the first assumption there is some $y\in B$ such that $\dist(\pi_\alpha(x),\pi_\alpha(y))>20D'$. Let $\beta$ be a geodesic from $x$ to $y$, and let $z\in\pi_\alpha(x)$. By the second assumption, $\beta$ comes $5D'$--close to both $z$ and $\pi_\alpha(y)$. But then the length of $\beta$ must be at least $(\dist(x,z)-5D')+20D'-5D'>\dist(x,z)$, which implies that $z\in B$, in conflict with the choice of $B$. Thus $\alpha$ is $50D'$--strongly contracting.
\end{proof}

Every curtain $\hh$ has two nonempty \emph{halfspaces}, $\hh^+$ and $\hh^-$: if $\hh$ is dual to the strongly contracting geodesic $\alpha$ at an interval $I$, then they are the sets of points $x$ such that $\pi_\alpha(x)$ intersects one of the two components of the complement of $I$ in $\alpha$. Note that the choice of length of $I$ means that $x\in X$ cannot simultaneously be in $\hh^+$ and $\hh^-$, and in particular, $\{\hh^-,\hh,\hh^+\}$ is a tripartition of $S$. More strongly, we have $\dist(\hh^-,\hh^+)\ge3D>1$. Note that this ``thickness'' increases with the contracting constant.

We say that $\hh$ \emph{separates} two points or subsets of $S$ if they lie in opposite halfspaces of $\hh$. A \emph{chain} of curtains is then a sequence $(\hh_i)$ such that $\hh_i$ separates $\hh_{i-1}$ from $\hh_{i+1}$ for all $i$.

\begin{definition}[Ball-separation] \label{def:ball_sep}
For a natural number $R$, we say that disjoint curtains $\hh_1$ and $\hh_2$ are \emph{$R$--ball-separated} if there exists a ball $B\subset\hh_1^+\cap\hh_2^-$ with radius at most $R$ and such that every geodesic from $\hh_1^-$ to $\hh_2^+$ meets $B$. 

An \emph{$R$--chain} is a chain $(\hh_i)$ of curtains such that $\hh_i$ and $\hh_j$ are $R$--ball-separated for all $i,j$.
\end{definition}

We say that an $R$--chain $(\hh_i)$ \emph{crosses} a curtain $\kk$ if all four quarterspaces $\hh_i^\pm\cap\kk^\pm$ are nonempty for all $i$.

\begin{lemma} \label{lem:ball_relative_separation}
If curtains $\kk_1$ and $\kk_2$ are $R$--ball-separated, then every $R$--chain crossing both $\kk_1$ and $\kk_2$ has length at most $3R+5$.
\end{lemma}

\begin{proof}
Let $B$ be a ball of radius at most $R$ such that every geodesic from $\kk_1^-$ to $\kk_2^+$ passes through $B$. Let $\mathfrak c=(\hh_1,\dots,\hh_{2n+R+1})$ be an $R$--chain that crosses both $\kk_1$ and $\kk_2$. Because $\dist(\hh_i^-,\hh_i^+)\ge1$ for each $i$, at most $R+1$ elements of $\mathfrak c$ can intersect $B$. After switching the order of the $\hh_i$, we therefore have that $\hh_1,\dots,\hh_n$ are all disjoint from $B$, and $B\subset\hh_n^+$. 

\begin{figure}[ht]
\includegraphics[width=4.5cm, trim = 4cm 6.5cm 4cm 5.5cm]{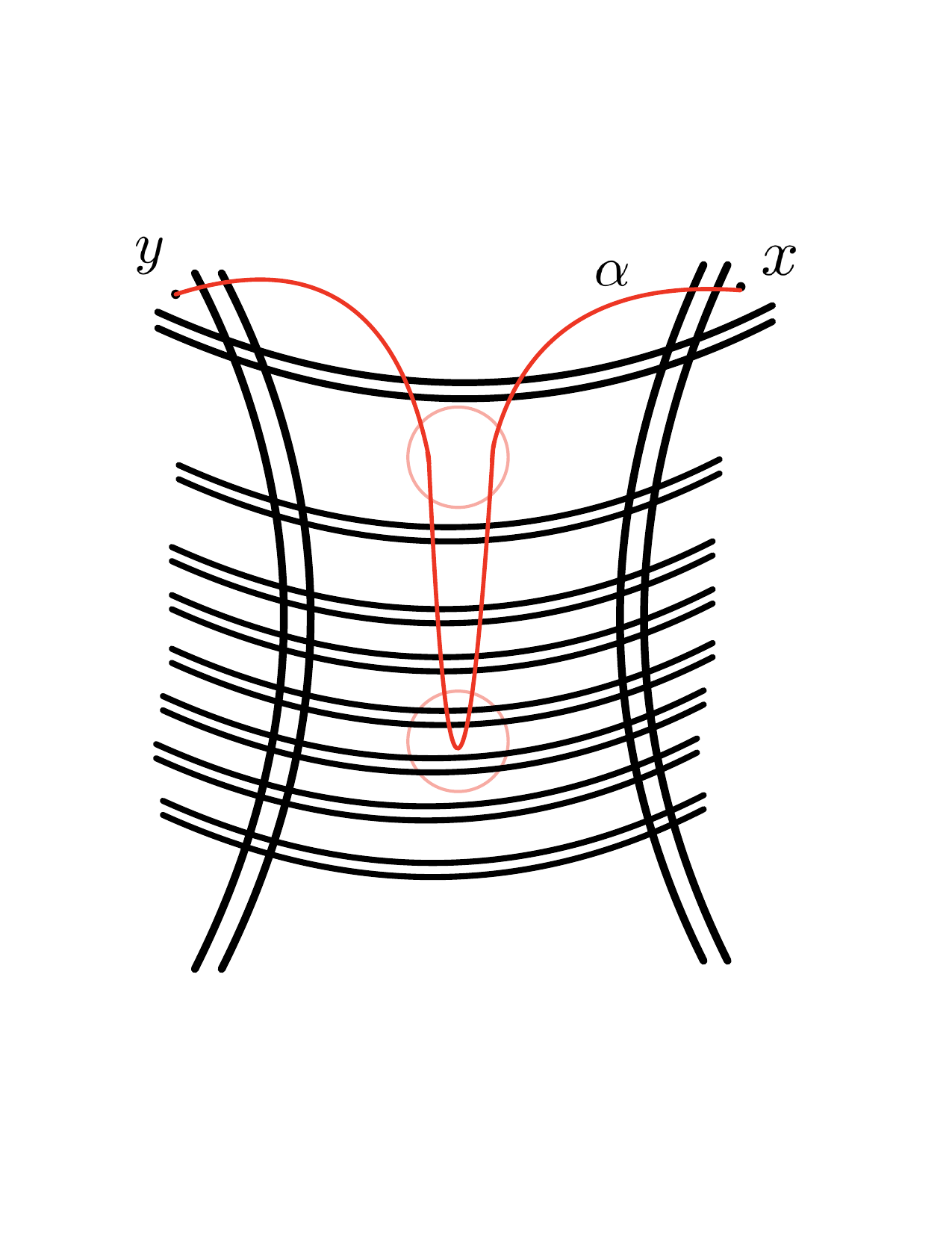}\centering
\caption{Lemma~\ref{lem:ball_relative_separation}. Geodesics from $x$ to $y$ have to pass through two balls of radius $R$.} \label{fig:returning_geodesic}
\end{figure}

Since $\hh_1$ and $\hh_2$ are $R$--ball-separated, there is a ball $B'$ of radius at most $R$ such that every geodesic from $\hh_1^-$ to $\hh_2^+$ must pass through $B'$. Because $\hh_1$ crosses both $\kk_1$ and $\kk_2$, there are points $x\in\hh_1^-\cap\kk_1^-$ and $y\in\hh_1^-\cap\kk_2^+$. If $\alpha$ is a geodesic from $x$ to $y$, then $\alpha$ must meet $B$. Let $z\in B\cap\alpha$, and let $\alpha_1=\alpha[x,z]$, $\alpha_2=\alpha[z,y]$. Both $\alpha_1$ and $\alpha_2$ must pass through $B'$. But $\diam B'\le2R$, so the subsegment of $\alpha$ lying in $\hh_2^+$ must have length at most $2R$. Because $z\in\hh_n^+$ and $\dist(\hh_i^-,\hh_i^+)\ge1$, it follows that $n\le R+2$.
\end{proof}

\ubsh{The hyperbolic core}
Each curtain $\hh\in\cur S$ induces two natural bipartitions of $S$, namely $(\hh^-\cup\hh,\hh^+)$ and $(\hh^-,\hh\cup\hh^+)$. We let $P$ be the \emph{set} of all bipartitions of $S$ induced by curtains in this way (different curtains can induce a common bipartition). Let $\C_R$ be the set of all chains $\{h_i\}\subset P$ such that there exists an $R$--chain $(\hh_i)\subset\cur S$ with $\hh_i$ inducing $h_i$. 

According to Lemma~\ref{lem:ball_graded} below, the sequence $(\C_R)$ is a graded system. For each $R$, let $X_R$ be the $\C_R$--dual of $S$. The \emph{hyperbolic core} $X$ is the graded dual of $S$ with respect to $(\C_R)$.
\uesh

\begin{lemma} \label{lem:ball_graded}
The sequence $(\C_R)$ defined above is a graded system on $(S,P)$, and each $\C_R$ is 3--gluable.
\end{lemma}

\begin{proof}
Every element of $\C_R$ is induced by a chain of curtains, and is therefore itself a chain from $P$. If a pair of curtains is $R$--ball-separated then it is $(R+1)$--ball-separated, so $\C_R\subset\C_{R+1}$. Lemma~\ref{lem:ball_relative_separation} provides separation of the $\C_R$. It remains to check that $\C_R$ is 3--gluable and to find an appropriate sequence $(\kappa_R)$ as in Definition~\ref{def:graded_system}.

For the latter, note that since $\dist(\hh^-,\hh^+)\ge1$ for every $h\in\cur S$, if $c\in\C_R$ separates $s\in S$ from $t\in S$, then $|c|\le\dist(s,t)+2$. This means we can take $\kappa_R=\frac1{R^2}$, for instance.

To show that $\C_R$ is 3--gluable, suppose that $c_1,c_2\in\C_R$ are such that $c_1\cup c_2$ is a chain with $c_2\subset c_1^+$ and $c_1\subset c_2^-$. Let $h_{-1},h_{-2}$ be the two maximal elements of $c_1$, and let $k_1,k_2,k_3$ be the three minimal elements of $c_2$. Let $(\hh_i)$ be an $R$--chain inducing $c_1$, and let $(\kk_i)$ be an $R$--chain inducing $c_2$. Let $B$ be a ball of radius at most $R$ such that every geodesic from $\kk_2^-$ to $\kk_3^+$ meets~$B$. 

Since $(\kk_i)$ is a chain of curtains, $\kk_2\subset\kk_1^+$. Also, the fact that $\{h_{-1},k_1\}$ is a chain in $P$ implies that $\hh_{-1}^-\subset\kk_1\cup\kk_1^-$. Hence $\kk_2$ is disjoint from $\hh_{-1}^-$. In particular, $\kk_2$ is disjoint from $\hh_{-2}$. Moreover, any geodesic from $\hh_{-2}^-$ to $\kk_3^+$ is a geodesic from $\kk_2^-$ to $\kk_3^+$, so any such geodesic meets $B$. That is, $\hh_{-2}$ and $\kk_3$ are $R$--ball-separated curtains. We conclude that $(\hh_i)_{i\le-2}\cup(\kk_j)_{j\ge3}$ is an $R$--chain, and hence $c_1\cup c_2\ssm\{h_{-1},k_1,k_2\}\in\C_R$.
\end{proof}

In view of Propositions~\ref{prop:graded_hyperbolic} and~\ref{prop:graded_wrg}, we therefore have that the hyperbolic core is a roughly geodesic hyperbolic space (see also Remark~\ref{rem:graded_roughly_geodesic}). According to the following proposition, one could also work in the roughly geodesic hyperbolic space $(S,\Dist)$ if preferred.

\begin{proposition} \label{prop:subspace_wrg}
With the subspace metric $\Dist$, the set $S\subset X$ is weakly roughly geodesic. Moreover, geodesics of $(S,\dist)$ are uniform unparametrised rough geodesics of $(X,\Dist)$. 
\end{proposition}

\begin{proof}
If $s_1,s_2\in S$ have $\dist(s_1,s_2)\le 2$, then $\Dist(s_1,s_2)\le\Lambda$, where $\Lambda$ is as in Section~\ref{subsec:unification}. Hence it suffices, given $r,t\in S$ and $s$ lying on an $(S,\dist)$--geodesic $\alpha$ from $r$ to $t$, to upper bound $\Dist(r,s)+\Dist(s,t)-\Dist(r,t)$. This will also imply the ``moreover'' statement. Let $\bar c=(c^R)$ be a perichain realising $\Dist(r,s)$. Let $c^R_s$ be the subset of $c^R$ that does not separate $r$ from $t$. We first show that $|c^R_s|\le R+3$. 

For this, fix $R$, let $(\hh_i)$ be an $R$--chain of curtains inducing $c^R$, and let $\hh_1,\dots,\hh_n$ be the curtains inducing $c^R_s$, with increasing index signifying increasing distance from $s$. Whilst we may have $s\in\hh_1$, we certainly have $s\in\hh_2^-$. Similarly, $r,t\in\hh_{n-1}^+$. Let $B$ be a ball of radius at most $R$ such that every geodesic from $\hh_{n-1}^+$ to $\hh_{n-2}^-$ must meet $B$. In particular, $\alpha[r,s]$ and $\alpha[s,t]$ both meet $B$. Because $\hh_2,\dots,\hh_{n-2}$ all separate $B$ from $s$, and $\dist(\hh_i^-,\hh_i^+)\ge1$, we therefore have $n\le R+3$, as desired.

Because of this, we can find perichains $\bar b_1$ and $\bar b_2$ such that $\bar b_1$ separates $r$ from $\{s,t\}$ and $\bar b_2$ separates $t$ from $\{s,r\}$, such that $\Dist(r,s)\le\|\bar b_1\|+\Lambda$ and $\Dist(s,t)\le\|\bar b_2\|+\Lambda$. Applying Lemma~\ref{lem:gluing} in each $\C_R$, we have $\Dist(r,t)\ge\|\bar b_1\|+\|\bar b_2\|-\Lambda$, and hence $\Dist(r,t)\ge\Dist(r,s)+\Dist(r,t)-3\Lambda$, which completes the proof.
\end{proof}



The following theorem characterises strongly contracting quasigeodesics of $S$ as those that quasiisometrically embed in the hyperbolic core. In particular, it shows that the existence of a strongly contracting ray in $(S,\dist)$ is sufficient for $(S,\Dist)\subset X$ to be unbounded.

\begin{theorem} \label{thm:universal_contracting_characterisation}
For each $q,D\ge1$ there exists $\nu\ge1$ such that if $\alpha$ is a $D$--strongly contracting $q$--quasigeodesic in a geodesic space $S$, then $\alpha\to X$ is a $\nu$--quasiisometric embedding in its hyperbolic core.

Conversely, for each $q,\nu$ there exists $D$ such that if $\alpha$ is a $q$--quasigeodesic in a geodesic space $S$ with the property that $\alpha\to X$ is a $\nu$--quasiisometric embedding in its hyperbolic core, then $\alpha$ is $D$--strongly contracting.
\end{theorem}

\begin{proof}
Let $\alpha:[0,T)\to S$ be a $D$--strongly contracting $q$--quasigeodesic, with $T\in[0,\infty]$. By considering a geodesic from $\alpha(0)$ to $\alpha(n)$ for each possible $n$, Lemma~\ref{lem:contracting_iff_pass_close} tells us that there is $R=R(D,q)$ such that there are $R$--chains of curtains separating points of $\alpha$, the cardinality of which is uniformly linearly lower bounded in terms of $R$ and $n$. Thus $\alpha\to X_R$ is a quasiisometric embedding. It follows from the construction of $\Dist$ that $\alpha\to X$ is also a quasiisometric embedding. This establishes the forward direction of the theorem.

Now let us consider the converse. Let $\Lambda$ be as in Section~\ref{subsec:unification}. We may assume that $q$, $\nu$, and $\Lambda$ are positive integers. Write $K=2q\nu\Lambda$. Let $\alpha$ be a $q$--quasigeodesic in $S$ such that $\alpha\to X$ is a $\nu$--quasiisometric embedding. After perturbing $\alpha$, we can coarsely cover it by a monotone sequence $(x_n)\subset\alpha$ such that $\dist(x_n,x_{n+1})=30K^2$.

As $\alpha\to X$ is a $\nu$--quasiisometric embedding, we have $\Dist(x_n,x_{n+1})\ge50K\Lambda$. By definition of $\Lambda$ and the fact that the separation constant $L_R$ of $\C_R$, coming from Lemma~\ref{lem:ball_relative_separation}, is at least $R$, there must be some $R$ such that $\dist_{\C_R}(x_n,x_{n+1})\ge50KR$. Hence there is an $R$--chain of curtains $(\hh^n_1,\dots,\hh^n_{50KR-2})$ separating $x_n$ from $x_{n+1}$. The fact that $\dist(\hh^-,\hh^+)\ge1$ for every $\hh\in\cur S$ now implies that $50KR-2\le\dist(x_n,x_{n+1})=30K^2$, and hence $R\le K$.

By repeatedly applying Lemma~\ref{lem:gluing}, and recalling that $\C_K$ is 3--gluable, we obtain an element $c\in\C_K$ by taking the union of all $(\hh^n_i)$ and removing at most $2(L_K+4)$ from each. By Lemma~\ref{lem:ball_relative_separation}, this leaves at least two elements of $(\hh^n_i)$ in $c$ for each $n$. Fix a choice of two, and label them $\hh_n,\kk_n$, with $\hh_n$ separating $x_n$ from $\kk_n$.

Because $\hh_n$ and $\kk_n$ are $K$--ball-separated, there is some ball $B_n$ of radius at most $K$ such that every geodesic from $\hh_n^-$ to $\kk_n^+$ must meet $B_n$. The subsegment of $\alpha$ from $x_n$ to $x_{n+1}$ has diameter at most $50K^2q^3$, because $\alpha$ is a $q$--quasigeodesic. Let $\gamma_n$ be a geodesic from $x_n$ to $x_{n+1}$. It meets $B_n$. Let $\gamma=\bigcup_{n\in\Z}\gamma_n$, which lies at Hausdorff-distance at most $100K^2q^3$ from $\alpha$. Because the $x_n$ formed a monotone sequence, the path $\gamma$ is therefore an unparametrised quasigeodesic. We aim to apply Lemma~\ref{lem:contracting_iff_pass_close} to find that $\gamma$, and hence $\alpha$, is strongly contracting.

\begin{figure}[ht]
\includegraphics[width=14cm, trim = 0cm 5cm 0cm 6cm, clip]{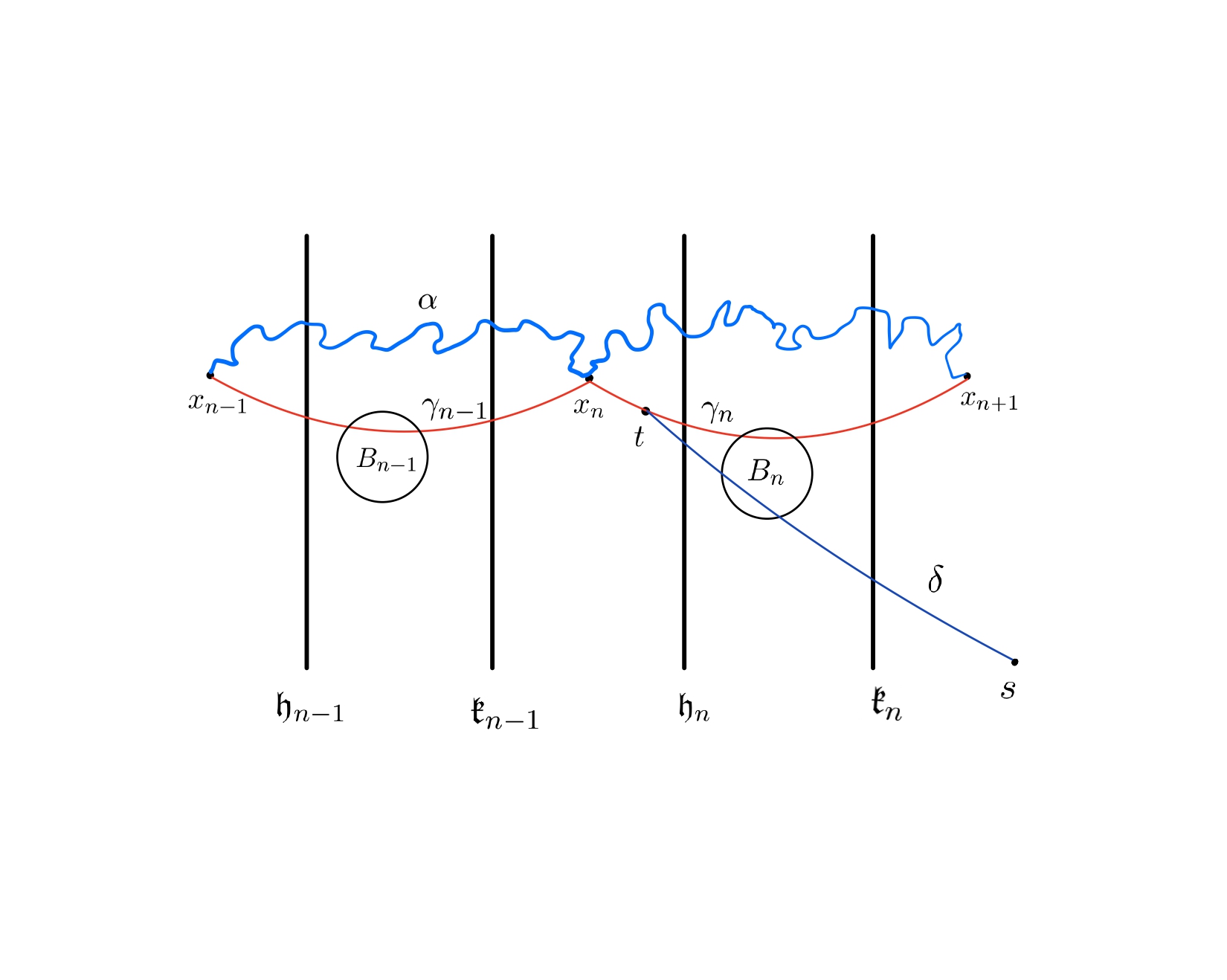}\centering
\caption{Theorem \ref{thm:universal_contracting_characterisation}. $\,x_n$ and $x_{n+1}$ are distant points along $\alpha$, separated by $\{\hh_n,\kk_n\}\in\C_K$. The path $\gamma$ is a piecewise geodesic, and any geodesic $\delta$ from $s$ to a projection $t\in\pi_\gamma(s)\cap\hh_n^-$ passes through $B_n$.} \label{fig:ball_separation}
\end{figure}

First we show that points of $S$ have uniformly bounded projection to $\gamma$. Given $s\in S$, let $n$ be maximal such that $s\in\kk_n^+$. If there is a point $t\in\pi_\gamma(s)$ lying in $\hh_n^-$, then any geodesic $\delta$ from $s$ to $t$ must meet $B_n$, so $\delta$ comes $2K$--close to $\gamma$ at some point in $B_n$, and hence $t$ is $2K$--close to $B_n$. Similarly, if $m$ is minimal such that $s\in\hh_m^-$, then any point of $\pi_\gamma(s)$ lying in $\kk_m^+$ is $2K$--close to $B_m$. The choices of $m$ and $n$ ensure that $m\in\{n+1,n+2\}$. In particular, the piecewise geodesic $\gamma_n\cup\gamma_{n+1}\cup\gamma_{n+2}$ meets both $B_n$ and $B_m$, so 
\[
\diam\pi_\gamma(s) \,\le\, 2K+\diam B_n+|\gamma_n|+|\gamma_{n+1}|+|\gamma_{n+2}|+\diam B_m+2K \,\le\, 100K^2.
\]

Moreover, as $\dist(\hh^-,\hh^+)\ge1$ for every $\hh\in\cur S$, we observe that no point of $\pi_\gamma(s)$ can be separated from $s$ by more than $2K+2$ of the $\hh_i$, because points of $\pi_\gamma(s)$ that lie in $\hh_n^-$ are $2K$--close to $B_n$, which meets $\hh_n^+$ and has diameter at most $2K$, and similarly points of $\pi_\gamma(s)$ that lie in $\hh_{m+1}^+\subset\kk_m^+$ are $2K$--close to $B_m$.

It remains to establish the second condition of Lemma~\ref{lem:contracting_iff_pass_close}. Suppose that $s_1,s_2\in S$ have $\dist(\pi_\gamma(s_1),\pi_\gamma(s_2))>(10K)^3\ge10(K+1)|\gamma_i|$, and let $\beta$ be a geodesic in $S$ from $s_1$ to $s_2$. Let $a$ be minimal such that $s_1\in\hh_a^-$, and let $b$ be maximal such that $s_2\in\kk_b^+$. Perhaps after swapping $s_1$ and $s_2$, the lower bound on $\dist(\pi_\gamma(s_1),\pi_\gamma(s_2))$ implies that $b-a>5(K+1)$, because $\gamma_i$ is a geodesic from $x_i\in\hh_i^-$ to $x_{i+1}\in\kk_i^+$. Thus the above observation yields that $\hh_{a+2K+2},\kk_{a+2K+2},\dots,\hh_{b-2K-2},\kk_{b-2K-2}$ all separate $\{s_1\}\cup\pi_\gamma(s_1)$ from $\{s_2\}\cup\pi_\gamma(s_2)$.

We deduce from this that for any subpath $\delta$ of $\gamma$ from $\pi_\gamma(s_1)$ to $\pi_\gamma(s_2)$, both $\beta$ and $\delta$ meet all of $B_{a+2K+2},\dots,B_{b-2K-2}$. In particular, the intersection of $\delta$ with the ``middle'' region $\hh_{a+2K+2}^+\cap\kk_{b-2K-2}^-$ of $S$ lies in a neighbourhood of $\beta$ of radius at most $|\gamma_i|+\diam B_i\le50K^2$.

The two ``end'' regions $S\ssm\hh_{a+2K+2}^+$ and $S\ssm\kk_{b-2K-2}^-$ are treated similarly, so let us just consider the former. By minimality of $a$, the above observation implies that every point of $\pi_\gamma(s_1)$ lies in $\hh_{a-2K-4}^+$. Since we know that $\beta$ meets $B_{a+2K+2}$, this implies that the intersection of $\delta\subset\gamma$ with this region lies in a neighbourhood of $\beta$ of radius at most $\diam B_{a+2K+2}+\sum_{i=a-2K-4}^{a+2K+2}|\gamma_i|\le(10K)^3$.

We have shown that the conditions of Lemma~\ref{lem:contracting_iff_pass_close} are met by $\gamma$, so it is strongly contracting. Since $\gamma$ and $\alpha$ lie at a bounded Hausdorff-distance, $\alpha$ is also strongly contracting.
\end{proof}

\begin{remark}
The assumption that $\alpha$ is a quasigeodesic in $S$ is essential for the converse direction of Theorem~\ref{thm:universal_contracting_characterisation}. For instance, if, as in Remark~\ref{rem:X_smaller}, $S$ is the region of the euclidean plane bounded between the $x$--axis and a sufficiently slowly growing sublinear function, then $X$ will be a quasiray even though no ray in $S$ is strongly contracting. But in this example no geodesic ray in $S$ is quasiisometrically embedded in $X$. 
\end{remark}

Next we show that, for isometries, being strongly contracting can be detected by just looking at pairs of curtains.

\begin{definition}
An isometry $g\in\isom S$ is \emph{strongly contracting} if there exist $D,q$ and $s\in S$ such that $\langle g\rangle\cdot s$ is a $D$--strongly contracting $q$--quasigeodesic in $S$.

We say that an isometry $g$ \emph{skewers} a pair of disjoint curtains $\hh_1,\hh_2$ if there exists $n$ such that $g^n\hh_1^+\subsetneq\hh_2^+\subsetneq\hh_1^+$.
\end{definition}

\begin{corollary} \label{cor:equivariant_characterisation}
If $g\in\isom S$ has quasiisometrically embedded orbits in $S$, then the following are equivalent.
\begin{enumerate}
\item   $g$ is strongly contracting. \label{item:sc}
\item   $g$ skewers a pair of ball-separated curtains. \label{item:skewer}
\item   $g$ acts loxodromically on the hyperbolic core $X$. \label{item:lox}
\end{enumerate}
\end{corollary}

\begin{proof}
Theorem~\ref{thm:universal_contracting_characterisation} shows that~\eqref{item:sc} and~\eqref{item:lox} are equivalent. 

Assume that $g$ acts loxodromically on $X$. Let $s\in S$ and let $n$ be sufficiently large that $\Dist(s,g^ns)>10\Lambda$. In particular, by definition $\Lambda\ge\sum\lambda_R(L_R+1)$, there must be some $R$ for which $\dist_R(s,g^ns)\ge L_R+9$. That is, there is some $c=\{h_1,\dots,h_k\}\in\C_R$ separating $s$ from $g^ns$, with $k\ge L_R+9$.

Consider the chain $g^nc\in\C_R$, which separates $g^ns$ from $g^{2n}s$, and recall from Lemma~\ref{lem:ball_graded} that $\C_R$ is 3--gluable. According to Lemma~\ref{lem:gluing}, there is some $b\in\C_R$ that can be obtained from $c\cup g^nc$ by removing a subset of $\{h_{k-3},h_{k-2},h_{k-1},h_k,g^nh_1,\dots,g^nh_{L_R+3}\}$ of cardinality at most $L_R+4$, and $b$ necessarily separates $s$ from $g^{2n}s$. 

Having found $b$, we can then find an element of $\C_R$ that separates $s$ from $g^{3n}s$ by considering $b\cup g^{2n}c$ in a similar way. Repeating this inductively, we find that $\bigcup_{m\in\Z}\{g^{mn}h_{L_R+4},\dots,g^{mn}h_{k-4}\}\in\C_R$, and moreover $\{g^{mn}h_{L_R+4},\dots,g^{mn}h_{k-4}\}$ separates $g^{rn}s$ from $g^{tn}s$ for all $r\le m$ and all $t>m$. In particular, if $(\hh_i)$ is a chain of curtains inducing $\{h_i\}$, then $g^{2n}$ skewers the ball-separated curtains $\hh_{k-4},g^n\hh_{k-4}$. Thus~\eqref{item:lox} implies~\eqref{item:skewer}.

Finally, suppose that $g$ skewers a pair of $R$--ball-separated curtains $\hh$ and $\kk$. Because $\hh$ and $\kk$ are $R$--ball-separated, $(g^{nm}\hh)_{m\in\Z}$ is an $R$--chain of curtains. In particular, for any $s\in S$ we have $\dist_{\C_R}(s,g^{nm}s)\ge m-2$. This shows that $g$ act loxodromically on $X$, so \eqref{item:skewer} implies~\eqref{item:lox}.
\end{proof}

More strongly, for any group $G$ acting properly on $S$, we show that the action of $G$ on $X$ is weakly proper along the axis of each strongly contracting element.

\begin{definition}[WPD]
Let $G$ be a group acting on a hyperbolic space $Y$. An element $g\in G$ is \emph{WPD} if there is a point $x\in Y$ such that for each $\eps>0$ there exists $m>0$ for which only finitely many $h\in G$ satisfy both $\dist(x,hx)<\eps$ and $\dist(g^mx,hg^mx)<\eps$.
\end{definition}

\begin{proposition} \label{prop:wpd}
Suppose that $G$ acts properly on a geodesic space $S$. If $g\in G$ is strongly contracting, then $g$ is a WPD loxodromic for the action of $G$ on the hyperbolic core $X$.
\end{proposition}

\begin{proof}
Fix $s\in S$ and let $\eps>0$. By definition, $g$ being strongly contracting implies that it has quasiisometrically embedded orbits in $S$, so by Corollary~\ref{cor:equivariant_characterisation}, $g$ acts loxodromically on $X$. In particular, there exists $m$ such that $\Dist(s,g^ms)\ge 2\eps+2\Lambda$. Let $\bar c=(c^R)$ be a perichain realising $\Dist(s,g^ms)$.

If $h\in G$ has the property that at most $2R+3$ elements of $c^R$ separate $hs$ from $hg^ms$ for every $R$, then for every $R$ all but at most $2R+3$ elements of $c^R$ either separate $s$ from $hs$ or separate $g^ms$ from $hg^ms$ (or both). For such an element $h$, we therefore have
\[
\Dist(s,hs)+\Dist(g^ms,hg^ms) \,\ge\, \sum_{R=1}^\infty\lambda_R(|c^R|-2R-3) \,\ge\, \|\bar c\|-2\Lambda \,\ge\, 2\eps.
\]

This shows that if $h\in G$ satisfies $\Dist(s,hs)<\eps$ and $\Dist(g^ms,hg^ms)<\eps$, then there is some $R$ for which at least $2R+4$ elements of $c^R$ separate $hs$ from $hg^ms$. Letting $(\hh_i)$ be an $R$--chain of curtains inducing the subset of $c^R$ separating $hs$ from $hg^ms$, we see that there is a ball $B\subset S$ of radius at most $R$ such that every geodesic from $\hh_{R+2}^-$ to $\hh_{R+3}^+$ must meet $B$, where $s,hs\in\hh_{R+2}^-$ and $g^ms,hg^ms\in\hh_{R+3}^+$.

Let $\gamma\subset S$ be a geodesic from $s$ to $g^ms$. We have shown that $h\gamma$ comes $2R$--close to $\gamma$. Moreover, because every $\hh\in\cur S$ has $\dist(\hh^-,\hh^+)\ge1$, we must have $\dist(s,g^ms)\ge2R+4$. Hence 
\[
\dist(s,hs) \,\le\, |\gamma|+2R+|g\gamma| \,\le\, 3\dist(s,g^ms).
\]
But the action of $G$ on $S$ is proper, so there are only finitely many such elements $h$.
\end{proof}

Next we extend Proposition~\ref{prop:wpd} by showing that, in the terminology of \cite{balasubramanyachesserkerrmangahastrin:non}, for many $S$ the hyperbolic core is a \emph{universal recognising space} for stable subgroups of groups acting geometrically on $S$.

\begin{definition}[Morse, stable]
Given a function $M:\R^2\to\R$, we say that a geodesic $\alpha$ in a metric space is \emph{$M$--Morse} if every $(\lambda,\mu)$--quasigeodesic with endpoints on $\alpha$ lies in the $M(\lambda,\mu)$--neighbourhood of $\alpha$.

A finitely generated subgroup $H$ of a finitely generated group $G$ is \emph{stable} if its inclusion map is a quasiisometric embedding and there exists $M:\R^2\to\R$ such that every geodesic in $G$ between points of $H$ is $M$--Morse.
\end{definition}

For example, strongly contracting geodesics are always Morse, but the reverse can fail \cite{rafiverberne:geodesics}. In many spaces, though, the two are equivalent \cite{charneysultan:contracting,calvezwiest:morse,sistozalloum:morse}. Stability was introduced by Durham--Taylor \cite{durhamtaylor:convex}, who showed that it reformulates \emph{convex cocompactness} for mapping class groups \cite{farbmosher:convex}. The curve graph is a universal recognising space for stable subgroups of mapping class groups \cite{kentleininger:shadows}.

\begin{theorem} \label{thm:universal_recognition}
Let $S$ be a geodesic space with the property that for each $M$ there exists $D$ such that every $M$--Morse geodesic in $S$ is $D$--strongly contracting. Let $X$ be its hyperbolic core. For any group $G$ acting properly coboundedly on $S$, a finitely generated subgroup $H<G$ is stable if and only if its orbit maps on $X$ are quasiisometric embeddings.
\end{theorem}

\begin{proof}
Suppose that $H$ is stable in $G$. Given $s\in S$ and $g,h\in H$, each geodesic from $gs$ to $hs$ is uniformly Morse. By assumption, such a geodesic is therefore uniformly strongly contracting, so by the construction of $\cur S$ and Lemma~\ref{lem:contracting_iff_pass_close} we obtain a sequence of uniformly ball-separated curtains separating $gs$ from $hs$ at a uniform rate. This shows that there is some $R$ such that orbit maps of $H$ on $X_R$ are uniform quasiisometric embeddings, and it follows that orbit maps of $H$ on $X$ are uniform quasiisometric embeddings.

Conversely, suppose that orbit maps of $H$ on $X$ are quasiisometric embeddings. Because $H$ is finitely generated, its orbit maps on $S$ are coarsely Lipschitz. Moreover, the fact that $\dist(\hh^-,\hh^+)\ge1$ for every $\hh\in\cur S$ implies that the identity map $(S,\dist)\to(X,\Dist)$ is coarsely Lipschitz, so orbit maps of $H$ on $S$ are in fact quasiisometric embeddings. In particular, the inclusion of $H$ in $G$ is a quasiisometric embedding. Furthermore, if $s\in S$ then for any $g,h\in H$, each geodesic $\alpha$ from $gs$ to $hs$ uniformly quasiisometrically embeds in $X$, so is uniformly strongly contracting by Theorem~\ref{thm:universal_contracting_characterisation}. This shows that $H$ is stable in $G$.
\end{proof}

\section{Strong coarse median spaces} \label{sec:scm}

Coarse median spaces were introduced by Bowditch in \cite{bowditch:coarse}, providing a general framework for studying groups that display features of median geometry up to controlled error, such as toral relatively hyperbolic groups \cite{bowditch:invariance} and mapping class groups \cite{behrstockminsky:centroids}, among many others. As discussed in the introduction, the idea is to take the tree-approximation lemma for hyperbolic spaces \cite{gromov:hyperbolic}, and use a higher-rank version of it as an axiomatisation. We do this slightly differently to \cite{bowditch:coarse}.


Given a metric space $S$ with a ternary operator $\mu$, a subset $Y\subset S$ is said to be \emph{$k$--coarsely convex} if $\mu(y_1,y_2,s)$ lies in the $k$--neighbourhood of $Y$ for all $y_1,y_2\in Y$, $s\in S$. A map $f:(S,\mu)\to(T,\nu)$ of spaces with ternary operators is called \emph{$k$--quasimedian} if $f\mu(s_1,s_2,s_3)$ lies at distance at most $k$ from $\nu(fs_1,fs_2,fs_3)$ for all $s_1,s_2,s_3\in S$.

Recall from Section~\ref{sec:background} that the (combinatorial) convex hull of a finite subset $A$ of a CAT(0) cube complex $Q$ is the smallest convex subcomplex of $Q$ that contains $A$.

\begin{definition}[Strong coarse median] \label{def:scm}
Let $S$ be a geodesic space with a ternary operator $\mu$. We say that $(S,\mu)$ is a \emph{strong coarse median space} if there exist an $n$ and a nondecreasing function $\kappa$ such that the following hold.
\begin{itemize}
\item   $\mu$ is $\kappa(1)$--coarsely Lipschitz in each parameter.
\item   For each finite subset $A\subset S$ there is a CAT(0) cube complex $Q$ of dimension at most $n$, and maps $f:A\to Q$, $g:Q\to S$ such that:
    \begin{itemize}
    \item 	$g$ is a $\kappa(|A|)$--quasimedian $\kappa(|A|)$--quasiisometric embedding.
    \item   $Q$ is the (combinatorial) convex hull of $f(A)$, and $g(Q)$ is $\kappa(|A|)$--coarsely convex.
	\item 	$\dist(a,gf(a))\le\kappa(|A|)$ for all $a\in A$.
    \end{itemize}
\end{itemize}
We say that $S$ has \emph{rank} at most $n$.
\end{definition}

The difference between this and the definition of a coarse median space of finite rank is twofold. Firstly, in coarse median spaces, the map $g$ is only required to be quasimedian, and need not include any metric information; in the terminology of \cite{bowditch:quasiflats}, we are replacing ``quasimorphisms'' by ``strong quasimorphisms''. Secondly, the complex $Q$ here is approximating the entire \emph{coarse median hull} of $A$, whereas in coarse median spaces it only has to approximate the \emph{coarse subalgebra} generated by $A$. For example, let $S$ be $\R^2$ with the $\ell^1$ median structure. If $|A|=2$, then in a coarse median structure we can always take $Q$ to be a single (unit) 1--cell, but in a strong coarse median structure, $Q$ will be the (metric) rectangle spanned by $A$.

In the rank-one case the two notions define the same objects, namely hyperbolic spaces, by \cite[Thm~2.1]{bowditch:coarse}. Both \emph{hierarchically hyperbolic} spaces \cite{behrstockhagensisto:quasiflats} and the spaces considered in \cite{bowditch:convex} are strong coarse median spaces.

The following construction appears in \cite[\S6]{bowditch:convex}, where it is considered for coarse median spaces more generally.

\begin{definition}[Median hull] \label{def:hull}
Let $S$ be a strong coarse median space of rank $n$, and let $A\subset S$. The \emph{median hull} of $A$, denoted $\hull A$, is the subset $J^n(A)\subset S$, where $J^0(A)=A$ and $J^{k+1}(A)=\{\mu(a,b,s)\,:\,a,b\in J^k(A),s\in S\}$. There is a constant $k_0=k_0(\kappa,n)$ such that $\hull A$ is $k_0$--coarsely convex for every $A\subset S$.
\end{definition}

The existence of $k_0$ as in Definition~\ref{def:hull} is given by \cite[Lem.~6.1]{bowditch:convex}.

It follows from the definition of a strong coarse median space that if $A\subset S$ is finite, then $g(Q)$ lies at a uniform Hausdorff distance from $\hull A$, in terms of $k_0$ and $|A|$. Hence, after increasing $\kappa$ by a controlled amount, we can extend $f$ to a $\kappa(|A|)$--quasimedian $\kappa(|A|)$--quasiisometry $\hat f:\hull A\to Q$ such that $\hat f$ and $g$ are $\kappa(|A|)$--quasiinverse.

We shall sometimes write $\mu_{abc}=\mu(a,b,c)$ in order to simplify some expressions. For instance, the equality $\mu(a,b,\mu_{cde})=\mu(\mu_{abc},\mu_{abd},e)$ holds in all median algebras, and therefore the same holds up to a uniform error in terms of $\kappa(5)$ in strong coarse median spaces; see \cite[\S6]{bowditch:large:mapping}, \cite[Lem~2.18]{niblowrightzhang:four}. After increasing $\kappa$ by a controlled amount, we shall therefore assume that the expressions $\mu(a,b,\mu_{cde})$ and $\mu(\mu_{abc},\mu_{abd},e)$ differ by at most $\kappa(5)$ itself inside $S$, rather than just a uniform function of it.

\subsection{Curtains} \label{subsec:scm_curtains}

Just as in Section~\ref{sec:contracting}, we shall go via a set of geometrically defined \emph{curtains} in order to define sets $P$ and $\C$ of partitions and chains on a strong coarse median space. Throughout this section, $(S,\dist,\mu)$ will be a strong coarse median space of rank $n$ with associated function $\kappa$. For each finite subset $A\subset S$, fix a choice of $g,Q$ satisfying the assumption of Definition~\ref{def:scm}, and, as discussed after Definition~\ref{def:hull}, fix a choice of $\hat f:\hull A\to Q$. Let $k_1=\max\{k_0,\kappa(2^{n+1}),\kappa(5)\}$.

\begin{definition}[Curtains] \label{def:scm_curtains}
Let $a,b\in S$, and let $A=\{a,b\}$. Let $c$ be a chain of hyperplanes in $Q$ of length $20nk_1^5$ such that $\dist^\infty_Q(c^-,c^+)=|c|$. Let 
\[
\hh^- \,=\, \{s\in S\,:\,\hat f\mu(a,b,s)\in c^-\} \quad \text{and} \quad 
\hh^+ \,=\, \{s\in S\,:\,\hat f\mu(a,b,s)\in c^+\}.
\]
The \emph{curtain} defined by $a,b,c$ is the set $\hh=S\ssm(\hh^-\cup\hh^+)$, and $\hh^-,\hh^+$ are the \emph{halfspaces} of $\hh$.
\end{definition}



A \emph{chain of curtains} is a sequence $(\hh_i)$ of curtains such that $\hh_i$ separates $\hh_{i-1}$ from $\hh_{i+1}$ for all $i$, in the sense that (up to orientation) $\hh_{i-1}\subset\hh_i^-$ and $\hh_{i+1}\subset\hh_i^+$.

\begin{remark} \label{rem:choice_of_cube_complexes}
The set of curtains we have constructed depends on the choices of cube complex approximations of pairs of points in $S$. In particular, it need not be preserved by $\isom S$. We chose this set of curtains because it is in a sense the most concrete option, and the underlying arguments of this section do not hinge on the specifics. Here are three natural alternative choices that would yield $\isom S$--equivariance with essentially no modification to any of our proofs. 
\begin{enumerate}
\item   Declare the set of curtains to be the set of all translates of those defined above.
\item   Fix a constant $\delta$, and consider all cubical approximations (of pairs of points) that have constant at most $\delta$. Define curtains as above but using all such approximations.
\item   Fix a constant $\delta$, and let $K$ be sufficiently large in terms of $\delta$ and the parameters of $X$. Define a curtain to be any $\delta$--coarsely convex subset $\hh\subset S$ with the property that $S\smallsetminus\hh$ can be written as a disjoint union of two nonempty subsets $\hh^-$ and $\hh^+$ such that $\dist(\hull\hh^-,\hull\hh^+)>K$.
\end{enumerate}
To keep matters simple, we shall proceed with the curtains of Definition~\ref{def:scm_curtains}, but in applications to groups one should use one of these above options.
%
\end{remark}

The next lemma shows that, although the two halfspaces of a curtain could fail to be coarsely convex, their hulls are well controlled by the cube complex used to define them.

\begin{lemma} \label{lem:joins_near_ends}
Let $\hh$ be the curtain defined by points $a,b$ and a chain $c$ of hyperplanes. If $x\in J^m(\hh^-)$, then at most $3m\kappa(5)^2$ elements of $c$ can separate $\hat f\mu(a,b,x)$ from $c^-$. A similar statement holds for $J^m(\hh^+)$. In particular, $\hull\hh^-\subset S\ssm\hh^+$.
\end{lemma}

\begin{proof}
We proceed by induction. If $x\in J^0(\hh^-)$, then by definition no element of $c$ separates $\hat f\mu(a,b,x)$ from $c^-$. Suppose that we have established the lemma for some value of $m$, and let $x\in J^{m+1}(\hh^-)$. We can write $x=\mu(y,z,s)$ for some $y,z\in J^m(\hh^-)$ and some $s\in S$. We then have $\dist(\mu(a,b,x),\,\mu(\mu_{aby},\mu_{abz},s))\le\kappa(5)$. Since $\hat f$ is a $\kappa(2)$--quasiisometry, we get 
\[
\dist\big(\hat f\mu(a,b,x),\,\hat f\mu(\mu_{aby},\mu_{abz},s)\big) \,\le\, \kappa(2)\kappa(5)+\kappa(2),
\]
and as $\hat f$ is $\kappa(2)$--quasimedian, the latter point lies at distance at most $\kappa(2)$ from $\mu_Q(\hat f\mu_{aby},\hat f\mu_{abz},\hat fs)$. By the inductive hypothesis, at most $3m\kappa(5)^2$ elements of $c$ separate this point from $c^-$, and so at most $3(m+1)\kappa(5)^2$ elements of $c$ separate $\hat f\mu(a,b,x)$ from $c^-$. 
\end{proof}

The following technical lemma can be viewed as a kind of weak Helly property.

\begin{lemma} \label{lem:weak_Helly_for_joins}
Let $\hh_1,\dots,\hh_m$ be curtains. If the halfspaces $\hh_i^-$ intersect pairwise, then there is a point $z\in\bigcap_{i=1}^mJ^{m-1}(\hh_i^-)$.
\end{lemma}

\begin{proof}
For each $j\ne i$, let $x_{ij}\in\hh_i^-\cap\hh_j^-$. We have $\mu(x_{ij},x_{ik},x_{jk})\in J^1(\hh^-_i)\cap J^1(\hh^-_j)\cap J^1(\hh^-_k)$. By repeating this type of argument (see \cite[Lem.~2.18]{haettelhodapetyt:coarse}, for instance), one can find a point $z\in\bigcap_{i=1}^mJ^{m-1}(\hh_i^-)$ as desired.
\end{proof}

\begin{definition}[Strong crossing]
We say that two curtains $\hh_1$ and $\hh_2$ \emph{strongly cross} if all four quarterspaces $\hh_1^\pm\cap\hh_2^\pm$ are nonempty.
\end{definition}

\begin{proposition} \label{prop:strong_cross_rank}
Let $S$ be a strong coarse median space of rank $n$. If $\hh_1,\dots,\hh_k$ are pairwise strongly crossing curtains, then $k\le n$.
\end{proposition} 

\begin{proof}
Suppose that $k=n+1$. Let $p\subset\{1,\dots,n+1\}$, and write $p(i)\in\{+,-\}$ according to whether or not $i\in p$. By Lemma~\ref{lem:weak_Helly_for_joins}, for each such $p$ there is a point $z_p\in\bigcap_{i=1}^{n+1}J^n(\hh_i^{p(i)})$. Let $A=\{z_p\}$, which has cardinality $2^{n+1}$. Let $Q$ be the CAT(0) cube complex approximating $A$, with corresponding map $\hat f:\hull A\to Q$.

The sets $J^n(\hh_i^\pm)=\hull\hh_i^\pm$ are $k_0$--coarsely convex, as is $\hull A$. Hence for each $i$ and each choice of sign, say $+$, the set $B_i^+=\hull A\cap\hull\hh_i^+$ is $k_0$--coarsely convex. From this, one can compute that $\hat f(B_i^+)\subset Q$ is $(k_0\kappa(2^{n+1})+3\kappa(2^{n+1}))$--coarsely convex. Because $S$ has coarse median rank $n$, we have $\dim Q\le n$, so $\hull_Q\hat f(B_i^+)$ lies in the $4nk_1^2$--neighbourhood of $\hat f(B_i^+)$. Similarly, $\hull_Q\hat f(B_i^-)$ lies in the $4nk_1^2$--neighbourhood of $\hat f(B_i^-)$. The map $\hat f$ is a $\kappa(2^{n+1})$--quasiisometry, so 
\[
\dist_Q(\hat f(B_i^-),\hat f(B_i^+)) \,\ge\, \frac1{k_1}\dist(\hull\hh_i^-,\hull\hh_i^+)-k_1.
\]

Let $a_i,b_i$ be the points, and $c_i\subset Q_i$ the chain used to define $\hh_i$, with corresponding map $\hat f_i$. According to Lemma~\ref{lem:joins_near_ends}, all but at most $6n\kappa(5)^2$ elements of $c_i$ can fail to separate $\hat f_i\mu(a_i,b_i,\hull\hh_i^-)$ from $\hat f_i\mu(a_i,b_i,\hull\hh_i^+)$. Combining this with the fact that $\mu$ is $\kappa(1)$--coarsely Lipschitz in each parameter and $\hat f_i$ is $\kappa(2)$--coarsely Lipschitz, it follows that 
\[
\kappa(1)(\kappa(2)\dist(\hull\hh_i^-,\hull\hh_i^+)+\kappa(2))+\kappa(1) \,\ge\, |c_i|-6n\kappa(5)^2,
\]
and hence $\dist(\hull\hh_i^-,\hull\hh_i^+)\ge\frac1{k_1^2}|c_i|-6n-2\ge12nk_1^3$. Together with the above lower bounds, we find that 
\[
\dist_Q(\hull_Q\hat f(B_i^-),\hull_Q\hat f(B_i^+)) \,\ge\, \frac1{k_1}(12nk_1^3)- k_1-2(4nk_1^2)
\]
is positive. As the sets $\hull_Q\hat f(B_i^\pm)$ are convex in the CAT(0) cube complex $Q$, there is a hyperplane $w_i$ of $Q$ separating the two. In particular, $w_i$ separates $\hat f(z_p)$ from $\hat f(z_q)$ whenever $p(i)\ne q(i)$. But then the hyperplanes $w_1,\dots,w_{n+1}$ must pairwise cross, which is impossible because $\dim Q\le n$.
\end{proof}

\subsection{The injective dual} \label{subsec:scm_injective}

Let $(S,\dist,\mu)$ be a strong coarse median space of rank $n$ with associated function $\kappa$. We go from curtains to a set of partitions of $S$ as in Section~\ref{sec:contracting}. More precisely, each curtain $\hh$ induces two natural bipartitions of $S$, namely $(\hh^-\cup\hh,\hh^+)$ and $(\hh^-,\hh\cup\hh^+)$. Let $P$ be the \emph{set} of all bipartitions induced in this way.

There are multiple dualisable systems that one can define using $P$. Here we shall consider what is perhaps the largest reasonable choice; we shall see more in Section~\ref{subsec:scm_hyp}. That is, we let $\C$ be the set of all chains $\{h_i\}\subset P$ such that there is a chain of curtains $(\hh_i)$ with $\hh_i$ inducing $h_i$. Throughout this section, $(X,\dist_\C)$ will denote the $\C$--dual of $S$. Because we are allowing all chains of curtains, it is easy to see that $\C$ is gluable.

\begin{lemma} \label{lem:scm_gluable}
$\C$ is a 2--gluable system of chains.
\end{lemma}

\begin{proof}
Let $c_1=\{\dots,h_{-2},h_{-1}\}$ and $c_2=\{k_1,k_2,\dots\}$ be elements of $\C$ such that $\{\dots,h_{-1},k_1,\dots\}$ is a chain, $c_2\subset h_{-1}^+$, and $c_1\subset k_1^-$. Let $(\hh_i)$ induce $c_1$ and let $(\kk_i)$ induce $c_2$. We have $\kk_2\subset\kk_1^+\subset k_1^+\subset h_{-1}^+\subset\hh_{-2}^+$, and similarly $\hh_{-2}\subset\kk_2^-$. Hence $c_1\cup c_2\ssm\{h_{-1},k_1\}\in\C$.
\end{proof}

The results of Section~\ref{sec:properties} therefore show that $X$ is a coarsely injective space with a good bicombing by normal wall paths. Normal wall paths are median paths, so it is desirable to relate the median on $X$ given by Lemma~\ref{lem:dual_median} to the coarse median on $S$.

\begin{lemma} \label{lem:scm_qmqie}
The map $S\to X$ is a quasiisometric embedding that is 3--quasimedian.
\end{lemma}

\begin{proof}
Let $s,t\in S$. From the CAT(0) cube complex $Q$ approximating $\hull\{s,t\}$, we obtain a chain of curtains separating $s$ from $t$ whose length is linearly lower bounded by $\dist(s,t)$. Conversely, because $\dist(\hh^-,\hh^+)\ge1$ for every curtain $\hh$, any element of $\C$ separating $s$ from $t$ can have length at most $\dist(s,t)+2$. Thus $S\to X$ is a quasiisometric embedding.

To show that $S\to X$ is quasimedian, let $s_1,s_2,s_3\in S$, and suppose that $\hh$ is a curtain with $s_1,s_2\in\hh^-$. Because $\mu_S(s_1,s_2,s_3)\in J^1(\hh^-)$, Lemma~\ref{lem:joins_near_ends} tells us that $\mu_S(s_1,s_2,s_3)\not\in\hh^+$. This shows that any chain of curtains separating $\mu_S(s_1,s_2,s_3)$ from the majority of $\{s_1,s_2,s_3\}$ has length at most one. Hence $S\to X$ is 3--quasimedian.
\end{proof}

\begin{proposition} \label{prop:scm_dense}
Let $S$ be a strong coarse median space of rank $n$, and let $X$ be the dual space described above. The image of $S$ in $X$ is coarsely dense.
\end{proposition}

The proof of this proposition is similar in spirit to that of Proposition~\ref{prop:sufficient_dense_L}, but is more complicated for two reasons. Firstly, we only have a coarse median, rather than points that roughly realise distances. Secondly, the separation property provided by Proposition~\ref{prop:strong_cross_rank} is weaker than $L$--separation.

Here is an outline of the argument. Given $x\in X$, we find a sequence of points $s_i\in S$ with the property that for any $c\in\C$ separating $s_i$ from $x$, only a uniform number of elements of $c$ can fail to separate $x$ from all $s_j$ with $j\le i$. From this it will follow that if $c_i\in\C$ realises $\dist_\C(s_i,x)$, then the tail of $c_i$ either is empty or crosses the tails of all $c_j$ with $j<i$. Proposition~\ref{prop:strong_cross_rank} then shows that the tail of $c_{n+1}$ must be empty, which bounds $\dist_\C(s_{n+1},x)$.

The most difficult part is the construction of the $s_i$. The reason for requiring the above property involving all $j\le i$ is that we wish to avoid a situation where the chains $c_i$ ``face'' each other, as in that case there would be no way to make the process terminate. Informally, we need to ensure that $s_{i+1}$ is not ``on the opposite side'' of $x$ to $s_i$. To do this, we use an auxiliary point $t_i\in S$ and ``project'' $x$ to the hull of $\{s_i,t_i\}$ to obtain $s_{i+1}$. This is reminiscent of the fact that in a CAT(0) cube complex, every wall separating a point $z$ from its gate to a convex set $C$ actually separates $z$ from all of $C$. Again, we have to be slightly careful with this step, because $x$ is not an element of $S$, so its gate in the sense of Section~\ref{subsec:gate} might not be an element of $S$ either. 

\begin{proof}[Proof of Proposition~\ref{prop:scm_dense}]
Fix a sufficiently large constant $C$, which could be explicitly determined from the below arguments in terms of $n$ and the parameters of the strong coarse median space $S$. 

Let $x\in X$, and choose an arbitrary point $s_1\in S$. Let $c_1\in\C$ realise $\dist_\C(s_1,x)$. If $|c_1|\le C$ then we are done. Otherwise there exists some point $t_1\in S$ such that $t_1\in c_1^+$, along with $x$. Write $A_1=\{s_1,t_1\}$ and consider the corresponding finite CAT(0) cube complex $Q_1=Q(A_1)$. 

\medskip

Given a curtain $\hh_i$ separating $s_1$ from $t_1$ (not necessarily arising from $Q_1$), there are two convex subcomplexes $H_i^-=\hull_{Q_1}(\hat f_1\mu(s_1,\hh_i^-,t_1))$ and $H_i^+=\hull_{Q_1}(\hat f_1\mu(s_1,\hh_i^+,t_1))$, which may overlap. Let $h_1,h_2\in P$ be walls coming from curtains $\hh_1$ and $\hh_2$ that separate $s_1$ from $t_1$, and suppose that the halfspace $x(h_i)$ corresponds to either $\hh_i^+$ or $\hh_i^-$ for both $i$ (as opposed to, say, $\hh_i^+\cup\hh_i$). After relabelling the orientations of the $h_i$, let us assume for concreteness that $x(h_i)=h_i^-$. Since $x$ is a filter, there must be some $p\in S$ such that $p\in\hh_1^-\cap\hh_2^-$. We then have $\hat f_1\mu(p,s_1,t_1)\in H_1^-\cap H_2^-$. That is, the convex subsets of $Q_1$ corresponding to $\hh_1^-$ and $\hh_2^-$ intersect.

Now consider the set $P_1\subset P$ consisting of all walls $h$ coming from curtains $\hh$ that separate $s_1$ from $t_1$ and such that $x(h)$ corresponds to either $\hh^-$ or $\hh^+$. By the previous paragraph, the pair $(x,P_1)$ determines a set of pairwise intersecting convex subcomplexes of the finite CAT(0) cube complex $Q_1$. By Helly's theorem, there is a point $\sigma_2$ in the total intersection of those subcomplexes. Let $s_2=g_1(\sigma_2)\in\hull(A_1)$, where $g_1:Q_1\to\hull A_1$. 

\medskip

Consider a chain $c\in\C$ separating $s_2$ from $x$. We claim that only a uniformly bounded number of elements of $c$ can separate $s_2$ from either $s_1$ or $t_1$. Firstly, boundedly many can separate $s_2$ from both $s_1$ and $t_1$. Indeed, $\sigma_2=\mu_{Q_1}(\hat f_1(s_1),\hat f_1(t_1),\sigma_2)$ and $g_1$ is quasimedian, so $s_2$ is uniformly close to $\mu(s_1,s_2,t_1)$, and Lemma~\ref{lem:scm_qmqie} shows that $S\to X$ is 3--quasimedian. Secondly, any elements of $c$ that separate $s_2$ from, say, $s_1$ but not $t_1$ must separate $\{t_1,s_2\}$ from $\{x,s_1\}$. Thus if there were many such elements then we would have a large number of curtains separating $\{t_1,s_2\}$ from $\{x,s_1\}$. But this would lead to a long chain of elements of $P_1$ separating $x$ from $s_2$, contradicting the construction of $s_2$.

Thus, if we let $c_2\in\C$ realise $\dist_\C(s_2,x)$, then all but uniformly many elements of $c_2$ separate $\{s_1,t_1,s_2\}$ from $x$. Moreover, any elements of $c_2$ not separating $\{s_1,t_1,s_2\}$ from $x$ have to occur at the end of $c_2$ that is closest to $s_2$: they all separate $s_2$ from all elements of $c_2$ that do separate $\{s_1,t_1,s_2\}$ from $x$. Let $c_2'$ denote the subset of $c_2$ consisting of all elements that separate $\{s_1,t_1,s_2\}$ from $x$. We have that $|c_2\ssm c_2'|$ is uniformly bounded.

Recall that every element of $c_1$ separates $s_1$ from $\{x,t_1\}$. From this it follows that if $h_1\in c_1$ and $h_2\in c_2'$ do not cross then $h_2^+\subset h_1^+$. For $r\ge4$, if the final $r$ elements of $c_2'$ were all contained in the positive halfspace of the $(r-3)^\mathrm{th}$-last element of $c_1$, then Lemma~\ref{lem:scm_gluable} would imply the existence of an element of $\C$ longer than $c_1$ that separates $s_1$ from $x$, contradicting the choice of $c_1$. Hence the $r^\mathrm{th}$-last element of $c_2'$ crosses the $(r-3)^\mathrm{th}$-last element of $c_1$ for each $r\ge4$. Note that if $h_2\in c_2'$ crosses $h_1\in c_1$, then $h_2$ must cross every wall after $h_1$ in $c_1$. Thus, either $|c_2|\le C$, in which case we are done, or $|c_2'|\ge12$, and its eighth-last to twelfth-last elements all cross the final five elements of $c_1$.

\medskip

From here we proceed inductively. Let $k\le n$ and suppose that we have constructed $s_1,\dots,s_k$, $t_1,\dots,t_{k-1}$, such that:
\begin{itemize}
\item   if $c_i$ and $c_j$ realise $\dist_\C(s_i,x)$ and $\dist_\C(s_j,x)$, respectively, with $i<j$, then the $(7j-6)^\mathrm{th}$-last to $(7j-2)^\mathrm{th}$-last elements of $c_j$ all cross the final $7j-2$ elements of $c_i$;
\item   all but a uniformly bounded number of elements of $c_i$ separate $\{s_1,\dots,s_{i-1},t_1,\dots,t_{i-1}\}$ from $x$.
\end{itemize}

Let $c_k\in\C$ realise $\dist_\C(s_k,x)$. Since $c_k$ is finite, there exists $t_k\in S$ such that $t_k\in c_k^+$. Write $A_k=\{s_1,\dots,s_k,t_1,\dots,t_k\}$ and let $Q_k=Q(A_k)$ be the corresponding finite CAT(0) cube complex, with $g_k:Q_k\to\hull A_k$. Referring to Lemma~\ref{lem:gate_via_medians}, write $\g_k$ for the map 
\[
z \,\mapsto\, \mu(s_k,z,\mu(t_k,z,\mu(s_{k-1},z,\mu(t_{k-1},z,\mu(s_{k-2},\dots,\mu(s_1,z,t_1)))\dots).
\] 
If $z\in S$, then up to a uniform error we have $\g_k(z)\in\hull(A_k)$ for all $z\in X$. Let $Q'=Q(A_k\cup\{z\})$, with $\hat f':\hull(A_k\cup\{z\})\to Q'$. In view of Lemma~\ref{lem:gate_via_medians}, we see that $\g_k(z)$ is uniformly close to the image in $\hull(A_k\cup\{z\})$ of the gate of $\hat f'(z)$ to $\hull_{Q'}(\hat f'(A_k))$ in $Q'$. The constants here depend on $k$, but since $k$ is bounded by $n$ we view them as being uniform in $n$.

Given a curtain $\hh$ that separates some pair of elements of $A_k$, there are two convex subcomplexes $H^-=\hull_{Q_k}(\hat f_k\g_k(\hh^-))$ and $H^+=\hull_{Q_k}(\hat f_k\g_k(\hh^+))$. Let $P_k\subset P$ be the set of walls $h$ coming from curtains $\hh$ that separate a pair of elements of $A_k$ and such that $x(h)$ corresponds to either $\hh^+$ or $\hh^-$. For each $h\in P_k$, let $\sign_x(h)\in\{-,+\}$ be such that $x(h)=h^{\sign_x(h)}$. Given $h_1,h_2\in P_k$, there is a point $p\in S$ lying in $\hh_1^{\sign_x(h_1)}\cap\hh_2^{\sign_x(h_2)}$, and hence the convex subcomplexes $H_1^{\sign_x(h_1)}$ and $H_2^{\sign_x(h_2)}$ of the finite CAT(0) cube complex $Q_k$ intersect. From this pairwise intersection, Helly's theorem provides a point $\sigma_{k+1}\in Q_k$ that lies in $H^{\sign_x(h)}$ for all $h\in P_k$. Let $s_{k+1}=g_k(\sigma_{k+1})$.

\medskip

Consider an arbitrary chain $c\in\C$ separating $s_{k+1}$ from $x$. As $\sigma_{k+1}\in Q_k$, Lemma~\ref{lem:gate_via_medians} tells us that $s_{k+1}$ lies at a uniform distance from $\g_k(s_{k+1})$, and hence only uniformly many elements of $c$ can separate $s_{k+1}$ from $A_k$. Furthermore, any subchain of $c$ separating $s_{k+1}$ from a subset $B\subset A_k$ yields a slightly shorter chain of elements of $P_k$ separating $s_{k+1}$ from $B$. But $g_k$ is quasimedian, and Lemma~\ref{lem:scm_qmqie} states that $S\to X$ is 3--quasimedian, so chains in $P_k$ that separate $s_{k+1}$ from $x$ have uniformly bounded length. This shows that, for any chain in $\C$ separating $s_{k+1}$ from $x$, all but a uniformly finite number (depending on $k\le n$) of elements must separate $x$ from $s_{k+1}\cup A_k$.

Let $c_{k+1}\in\C$ realise $\dist_\C(s_{k+1},x)$. If $|c_{k+1}|\le C$ then we are done, so suppose otherwise. For $r>4$, if the final $r$ elements of $c_{k+1}$ were contained in the positive halfspace of the $(r-3)^\mathrm{th}$-last element of $c_i$ for some $i\le k$, then Lemma~\ref{lem:scm_gluable} would tell us that we could have made a longer choice of $c_i$, which is impossible. Hence the $r^\mathrm{th}$-last element of $c_{k+1}$ crosses the $(r-3)^\mathrm{th}$-last element of $c_i$ for every $i\le k$. It follows that the $(7k+1)^\mathrm{th}$-last to $(7k+5)^\mathrm{th}$-last elements of $c_{k+1}$ all cross the final $7k-2$ elements of $c_i$ for all $i\le k$.

\medskip

Suppose that the above process has not terminated before the final step $k=n$, because $\dist_\C(s_i,x)>C$ for all $i\le n$. If it is also the case that $\dist_\C(s_{n+1},x)>C$, then for each $i\le n+1$ let $h^i_1,\dots,h^i_5$ be the $(7i-6)^\mathrm{th}$-last to $(7i-2)^\mathrm{th}$-last elements of $c_i$. If $i\ne j$, then every $h^i_l$ crosses every $h^j_l$. This implies that the curtains $\hh^1_3,\dots,\hh^{n+1}_3$ pairwise strongly cross, which contradicts Proposition~\ref{prop:strong_cross_rank}. So we must have $\dist_\C(s_i,x)\le C$ for some $i\le n+1$.
\end{proof}

\begin{remark}
One can tweak the definition of $\g_k$ in the above proof to provide a notion of gate map to hulls of finite subsets of general coarse median spaces of finite rank, with constants independent of the cardinality of the finite subset.
\end{remark}

The following is the main result of this section. It includes the notion of \emph{$(n,\delta)$--hyperbolicity}, which was introduced in \cite{jorgensenlang:combinatorial} as a higher-rank form of negative curvature. We shall not discuss this notion in detail here, but it replaces the four-point condition for hyperbolicity by a $(2n+2)$--point condition.

\begin{theorem} \label{thm:scm_injective}
Let $S$ be a strong coarse median space of rank $n$, and let $X$ be as constructed above. The map $S\to X$ is a quasimedian quasiisometry to a coarsely injective space. Moreover, $X$ is $(n,\delta)$--hyperbolic. 
\end{theorem}

\begin{proof}
$X$ is coarsely injective by Theorem~\ref{thm:coarsely_injective}, and Lemma~\ref{lem:scm_qmqie} together with Proposition~\ref{prop:scm_dense} shows that $S\to X$ is a quasimedian quasiisometry. The fact that $X$ is $(n,\delta)$--hyperbolic is a combination of \cite[Thms~2.2,~2.3]{bowditch:coarse}, which control the \emph{asymptotic rank} of $X$, and then \cite[Thm~1.4]{jorgensenlang:combinatorial}.
\end{proof}

By a \emph{strong-coarse-median group}, we mean a finitely generated group $G$ with a $G$--equivariant ternary operator making it a strong coarse median space. We refer to \cite{jorgensenlang:combinatorial} for the notion of having \emph{slim simplices}, which generalises that of having slim triangles as in a hyperbolic space.

\begin{corollary} \label{cor:scm_group_injective}
If $G$ is a strong-coarse-median group of rank $n$, then $G$ acts properly coboundedly on an injective metric space with slim $n$--simplices.
\end{corollary}

\begin{proof}
Referring to Remark~\ref{rem:choice_of_cube_complexes}, the construction of $X$ can be done in a $G$--equivariant way. In view of Theorem~\ref{thm:scm_injective}, it then follows from \cite[Prop.~1.1]{haettelhodapetyt:coarse} and \cite[Prop.~3.7]{lang:injective} that $G$ acts properly coboundedly on the injective hull of $X$, which has the slim simplex property by \cite[Thm~1.3]{jorgensenlang:combinatorial}.
\end{proof}

Theorem~\ref{thm:scm_injective} should be compared with \cite[Cor.~3.6]{haettelhodapetyt:coarse}. The construction of the metric in that paper works for any strong coarse median space, but the argument that it is coarsely injective requires the space to be hierarchically hyperbolic, and relies on the hierarchy structure. Additionally, there is no equivalent of normal wall paths in that setting. In any case, since hierarchically hyperbolic groups are strong-coarse-median groups \cite{behrstockhagensisto:quasiflats}, Corollary~\ref{cor:scm_group_injective} recovers the main result of \cite{haettelhodapetyt:coarse}, that hierarchically hyperbolic groups act properly coboundedly on injective spaces. Corollary~\ref{cor:scm_group_injective} generalises \cite[Cor.~1.6]{jorgensenlang:combinatorial}.

\subsection{Hyperbolic models} \label{subsec:scm_hyp}

One can also construct hyperbolic spaces from the set $P$ of partitions constructed at the beginning of Section~\ref{subsec:scm_injective}, in a similar manner to Section~\ref{sec:contracting}. Let us say that two disjoint curtains $\hh_1$ and $\hh_2$ are \emph{$R$--separated} if there is no chain $c$ of curtains of cardinality greater than $R$ whose elements all strongly cross both $\hh_1$ and $\hh_2$. By an \emph{$R$--chain of curtains}, we mean a chain of curtains such that each pair is $R$--separated. 

Let $\C$ be as in Section~\ref{subsec:scm_injective}. For each natural number $R$, we define a dualisable system $\C_R$ on $S$ by letting $\C_R\subset\C$ consist of all elements that are induced by $R$--chains of curtains. It is routine to check the following (\emph{cf.} Lemma~\ref{lem:ball_graded}).

\begin{lemma} \label{lem:scm:graded}
The sequence $(\C_R)$ is a graded system on $(S,P)$, and each $\C_R$ is 3--gluable.
\end{lemma}

Let $Y$ be the graded dual of $S$ with respect to $(\C_R)$. By Propositions~\ref{prop:graded_hyperbolic} and~\ref{prop:graded_wrg}, $Y$ is a roughly geodesic hyperbolic space. We already described a way to construct a hyperbolic space associated to $S$ in Section~\ref{sec:contracting}, but $Y$ has the advantage that it can more easily be compared with existing hyperbolic spaces in known examples, as we now briefly discuss.

In \cite{abbottbehrstockdurham:largest}, Abbott--Behrstock--Durham constructed, for each hierarchically hyperbolic group $G$, a hyperbolic space $Z$ witnessing a largest acylindrical action of $G$. A detailed discussion is beyond our scope, so we refer the reader to \cite{behrstockhagensisto:hierarchically:2} for background on hierarchical hyperbolicity, and to \cite[\S3]{abbottbehrstockdurham:largest} for details about the construction of $Z$.

\begin{proposition} \label{prop:abd}
Let $G$ be a hierarchically hyperbolic group. Let $Y$ be the hyperbolic space constructed above, and let $Z$ be the hyperbolic space constructed in \cite{abbottbehrstockdurham:largest}. The two are $G$--equivariantly quasiisometric.
\end{proposition}

\begin{proof}[Outline of proof.]
$Z$ is a cone-off of $G$, and $Y$ is the metric quotient of an equivariant pseudometric on $G$. We show that the natural maps $G\to Z$ and $G\to Y$ yield a quasiisometry between $Y$ and $Z$, and this map is automatically $G$--equivariant. Let $s,t\in G$ and let $c$ be a chain of curtains obtained from the cube complex approximating $\hull\{s,t\}$. If two elements of $c$ are such that their images under the projection map $\pi_Z:G\to Z$ are far apart, then those curtains are $R$--separated for sufficiently large $R$. This shows that $Y\to Z$ is coarsely Lipschitz.

For the other direction, consider a subpath $\sigma'$ of the normal wall path $\sigma$ from $s$ to $t$ such that $\pi_Z\sigma'$ has small diameter. There are two cases. The first is that there is no \emph{relevant domain} $U$ for $\{s,t\}$ with $\rho^U_Z$ near $\pi_Z\sigma'$ (see \cite{behrstockhagensisto:hierarchically:2}). In this case, $\sigma'$ itself has small diameter in $G$. Otherwise there is such a $U$, and then \cite[Thm~3.7]{abbottbehrstockdurham:largest} shows that $U$ is one factor of a nontrivial product. This product structure prevents $U$ from contributing any separated curtains. Thus separated curtains can only arise when $\sigma$ makes progress in $Z$, so $Z\to Y$ is coarsely Lipschitz.
\end{proof}

In particular, Proposition~\ref{prop:abd} shows that if $S$ is the mapping class group of a surface, then $Y$ is quasiisometric to the curve graph of that surface.




\section{Other directions} \label{sec:future}

We believe that there should be many situations where the construction of Section~\ref{sec:construction} can be applied. Here we suggest a few possible avenues. 

\ubsh{CAT(0) spaces and cube complexes}
Much of this article stems from constructions in CAT(0) spaces \cite{petytsprianozalloum:hyperbolic}, though the combinatorial perspective in terms of ultrafilters used here is rather different. That said, if one uses the curtains constructed in \cite{petytsprianozalloum:hyperbolic} to induce a set of partitions as in Sections~\ref{sec:contracting} and~\ref{sec:scm}, then the graded dual with respect to the systems of $L$--separated curtains should be essentially equivalent to the \emph{curtain model}.

If one starts with $S$ a CAT(0) cube complex and $P$ the set of hyperplanes, then, as discussed in Example~\ref{eg:sageevable}, taking $\C$ to be the set of all chains simply makes $X$ the Helly thickening of $S$. Letting $\C_L$ be the set of chains of pairwise $L$--separated hyperplanes, one obtains a space very similar to that of \cite[\S6.6]{genevois:hyperbolicities} (the difference being in the precise definition of separation). The sequence $(\C_L)$ is a graded system on $S$, so one could investigate the graded dual, which is a more natural hyperbolic space than the curtain model in this setting.
\uesh

\ubsh{The Bestvina--Bromberg--Fujiwara construction}
In influential work, Bestvina--Bromberg--Fujiwara gave a general method for assembling a collection of metric spaces in a quasitree-like fashion \cite{bestvinabrombergfujiwara:constructing}, by gluing them at bounded sets. This is often used when the spaces in question are hyperbolic, quasitrees, or even quasilines \cite{sisto:contracting,bestvinabrombergfujiwara:proper,behrstockhagensisto:asymptotic}. 

Since geodesic hyperbolic spaces are quasiisometric to the dual space of some dualisable system, in many cases one can interpret the assumptions in the BBF construction in terms of extensions of the partitions of the component spaces. Whilst this alternative axiomatisation may not be directly useful, it seems plausible that it could open the door to generalisations that allow for gluings along larger subspaces: a coarse point could perhaps be replaced by a subset that is approximately gated in the sense of Section~\ref{subsec:gate}.
\uesh

\ubsh{Curve graph of a surface}
Let $\MCG(\Sigma)$ be the mapping class group of a finite-type surface $\Sigma$. One corollary of Proposition~\ref{prop:abd} is that the hyperbolic space constructed for $\MCG(\Sigma)$ in Section~\ref{subsec:scm_hyp} is equivariantly quasiisometric to the curve graph of $\Sigma$. The mapping class group also admits a proper cobounded action on an injective space $S$, either via \cite{haettelhodapetyt:coarse} or Theorem~\ref{thm:scm_injective}. According to \cite[Thm~A]{sistozalloum:morse}, Morse subsets of injective metric spaces are strongly contracting, so orbits of pseudo-Anosovs on $S$ are strongly contracting. Is the hyperbolic core also quasiisometric to the curve graph of $\Sigma$? This seems especially likely in view of Theorem~\ref{thm:universal_recognition} and \cite{durhamtaylor:convex}.
\uesh

\ubsh{Hyperbolic models for other groups}
There are various ``nonpositively curved'' groups for which hyperbolic models have been constructed, such as $\Out F_n$ \cite{hatchervogtmann:complex,kapovichlustig:geometric,handelmosher:free:1} and various Artin--Tits groups \cite{kimkoberda:embedability,calvezwiest:hyperbolic,morriswright:parabolic,martinprzytycki:acylindrical}. It would be interesting to know whether some of these models can be (coarsely) reconstructed from a suitable set of walls. For instance, can one find a natural set of curtains in Culler--Vogtmann outer space, and does this reproduce, say, the free factor complex? How about for the Deligne complex?

Moreover, in view of Theorems~\ref{introthm:stable} and~\ref{introthm:quasimorphisms}, it would also be desirable to have more examples of metric spaces where all Morse geodesics are strongly contracting.
\uesh

\ubsh{Higher-rank hyperbolicity}
In Section~\ref{subsec:relative_separation}, we saw that if a dualisable system is separated and gluable then fairly straightforward combinatorial arguments show four-point hyperbolicity of the dual. In \cite{jorgensenlang:combinatorial}, Jørgensen--Lang introduced a family of higher-rank generalisations of the four-point inequality, and a space satisfying their $(2n+2)$--point inequality is said to be $(n,\delta)$--hyperbolic.

For a higher-rank version of the $L$--separation condition, let us say that a dualisable system of chains is \emph{$(n,L)$--separated} if, whenever $c_1,\dots,c_{n+1}$ are elements of $\C$ such that every element of $c_i$ crosses every element of $c_j$ for all $i,j$, we necessarily have $|c_k|\le L$ for some $k$. Is it true that if $\C$ is gluable and $(n,L)$--separated then $X$ is $(n,\delta)$--hyperbolic? Does an analogue of Proposition~\ref{prop:sufficient_dense_L} hold? Note that this $(n,L)$--separation assumption holds in the setting of strong coarse median spaces, by Proposition~\ref{prop:strong_cross_rank}. In that setting, though, $(n,\delta)$--hyperbolicity follows \emph{a posteriori} from coarse injectivity of $X$ (Theorem~\ref{thm:coarsely_injective}) and coarse density of $S$ (Proposition~\ref{prop:scm_dense}). 
\uesh

\ubsh{Metric quotients}
One very useful consequence of the duality between CAT(0) cube complexes and discrete wallspaces \cite{chatterjiniblo:from,nica:cubulating} is a simple trick for producing quotients of a given CAT(0) cube complex $Q$. Namely, one takes a subset $P'$ of the walls of $Q$, and lets $Q'$ be the cube complex dual to $P'$. This procedure, known as a \emph{restriction quotient}, was introduced to CAT(0) cube complexes in \cite{capracesageev:rank}, though it appeared earlier for median graphs \cite{mulder:structure}.

It is easy to see that restriction quotients can be taken in the generality of Section~\ref{sec:construction}. More precisely, let $\C$ be a dualisable system for a set with walls $(S,P)$, and let $X$ be the $\C$--dual. Given a subset $P'\subset P$, the set $\C'$ consisting of all elements of $\C$ supported on $P'$ (i.e., elements $\{h_1 \cdots , h_n \} \in \C$ with each $h_i \in P'$) is a dualisable system for $(S,P')$, and there is a natural quotient map from $X$ to the $\C'$--dual of $S$.

Many desirable properties of a dualisable system, such as gluability and the property of being a system of chains, are preserved by this restriction, so often the quotient $X'$ will have similar properties to $X$. For instance, if $\C$ is the set of all chains, so that $X$ is a Helly graph, then $X'$ will also be a Helly graph.
\uesh

\ubsh{Random walks}
A common application for producing actions of a finitely generated group $G$ on hyperbolic spaces is that it can yield information about random walks on $G$ \cite{kaimanovich:poissonformula,mahertiozzo:random,qingrafitiozzo:sublinearly:2}. In the very general setting where $G$ acts properly with a strongly contracting element on some geodesic space, a combination of Theorem~\ref{thm:universal_contracting_characterisation} and Proposition~\ref{prop:wpd} with \cite[Thm~1.2]{chawlaforghanifrischtiozzo:poisson} shows that if $\mu$ is a generating probability measure on $G$ with finite entropy, then the Poisson boundary of $G$ is modelled by $(\partial X,\nu)$, where $X$ is the hyperbolic core and $\nu$ is the hitting measure on the Gromov boundary $\partial X$ of the random walk driven by $\mu$.

For more precise information about the limiting behaviour of a random walk, one can ask whether a central limit theorem holds \cite{bjorklund:central,horbez:central,fernoslecureuxmatheus:contact}. A new approach to this type of problem was introduced by Benoist--Quint \cite{benoistquint:central:linear,benoistquint:central:hyperbolic}, and this was used together with \cite{petytsprianozalloum:hyperbolic} by Le~Bars to establish a central limit theorem for random walks on groups acting on CAT(0) spaces \cite{lebars:central,lebars:marches}. Since the construction of the hyperbolic core bears many geometric similarities to that of \cite{petytsprianozalloum:hyperbolic}, it is natural to ask whether one can work along similar lines to prove a central limit theorem for random walks on groups with a contracting element.
\uesh

\ubsh{A continuous variant}
Spaces with walls can be generalised to spaces with \emph{measured walls} \cite{cherixmartinvalette:spaces}, and these still exhibit a duality with median metric spaces \cite{chatterjidrutuhaglund:kazhdan,fioravanti:roller}. One could consider a continuous generalisation of the constructions of this article. For this, measures cease to be appropriate, because, for instance, the union of two chains need not be a chain, and so the set $\C$ will generally not be a $\sigma$--algebra. However, the property of being a measure is not really an essential feature for the construction, and one can just request a function $\nu:\C\to[0,\infty]$ satisfying certain compatibility criteria. 

We believe that many of the statements in the present article would then admit continuous formulations. One concrete question would be: is there an alternative construction of the injective hull of a metric space $S$ that can be achieved by letting $P$ be the set of all balls in $S$ and taking an appropriate function $\nu$?
\uesh
\appendix \addtocounter{section}{1}
\part*{Appendix. {} Quasimorphisms \textmd{(with Davide Spriano)}}

The purpose of this appendix is to study the vector space $\widetilde{\QM}(\Gamma)$ of (nontrivial, homogeneous) quasimorphisms of groups $\Gamma$ acting coboundedly on spaces with strongly contracting geodesics. We use the hyperbolic core together with the Bestvina--Fujiwara criterion \cite{bestvinafujiwara:bounded}, which we now recall.

Given $q,\delta\ge1$, let $B=B(q,\delta)$ be a sufficiently large constant, as in \cite[p.72]{bestvinafujiwara:bounded}. Let $X$ be a $\delta$--hyperbolic space on which a group $\Gamma$ acts. Suppose that $g_1$ and $g_2$ are loxodromic isometries with $q$--quasiaxes $A_1$ and $A_2$, respectively. Write $g_1\sim g_2$ if there are arbitrarily long segments $J\subset A_1$ and elements $g_J\in\Gamma$ such that $g_J(J)$ is contained within the $B$--neighbourhood of $A_2$. For instance, if $A_1$ and $A_2$ lie at finite Hausdorff-distance, then $g_1\sim g_2$.

\begin{theorem}[{\cite{bestvinafujiwara:bounded}}] \label{thm:bestfuj}
If a group $\Gamma$ acts on a hyperbolic space $X$ and has loxodromic elements $g_1$ and $g_2$ such that $g_1\not\sim g_2$, then $\widetilde{\QM}(\Gamma)$ is infinite-dimensional.
\end{theorem}

The first step in our argument is the following, which may be of independent interest. The main novelty of the statement lies in the fact that it is strongly contracting \emph{axes} that are produced, rather than just strongly contracting geodesics, but it is also noteworthy that there is no properness assumption involved. Previously it was known that if a proper non-hyperbolic space contains some Morse ray then it contains Morse rays with arbitrarily bad Morse gauge \cite[Cor.~1.17]{cordesdurham:boundary}.

\begin{proposition} \label{prop:bad_contracting_element}
Let $S$ be a non-hyperbolic geodesic space with the property that for each Morse gauge $M$ there exists $D$ such that every $M$--Morse geodesic in $S$ is $D$--strongly contracting. Suppose that $S$ contains some biinfinite strongly contracting geodesic. For every group $\Gamma<\isom S$ acting coboundedly on $S$ there exists $q$ such that for each $D$, there is an element $g\in\Gamma$ with a $q$--quasiaxis that is strongly contracting but not $D$--strongly contracting.
\end{proposition}

\begin{proof}
Let $X$ be the hyperbolic space constructed in Section~\ref{sec:contracting}. The change-of-metric map $(S,\dist)\to(X,\Dist)$ is coarsely Lipschitz and $\isom S$--equivariant, and Theorem~\ref{thm:universal_contracting_characterisation} shows that every contracting geodesic in $S$ is quantitatively quasiisometrically embedded in $X$. Let $D_0$ be such that there is a $D_0$--strongly contracting biinfinite geodesic $A\subset S$. Let $s_0\in A$, and fix any number $D>D_0$. 

According to \cite[Cor.~3.6]{goldsboroughhagenpetytsisto:induced}, for each Morse gauge $M$ there is a Morse gauge $M'$, a constant $r$, and an $M'$--Morse ray $\beta\subset S$ emanating from $s_0$, such that for any $M$--Morse ray $\alpha$ emanating from $s_0$, we have 
\[
\diam(\pi_\alpha\beta)\le r \quad \text{and} \quad \diam(\pi_\beta\alpha)\le r,
\]
where $\pi_\alpha:S\to\alpha$ is a map such that $\Dist(s,\pi_\alpha(s))=\Dist(s,\alpha)$, and $\pi_\beta:S\to\beta$ satisfies $\Dist(s,\pi_\beta(s))=\Dist(s,\beta)$. Note that although the statement in \cite{goldsboroughhagenpetytsisto:induced} assumes that $S$ is a group, that property is not used in the proof: all that is required is a coarsely Lipschitz map $S\to X$.

Let $M$ be a Morse gauge that is sufficiently large compared to the constant $D$. In particular, choose $M$ such that every $D$--contracting geodesic is $M$--Morse. Let $M'$, $r$, and $\beta$ be obtained from $M$ as above. Let $K$ be such that the orbits of $\Gamma$ are $K$--coarsely dense in $S$. By \cite[Prop.~4.7]{sistozalloum:morse}, the space $S$ is \emph{Morse local-to-global} in the sense of \cite{russellsprianotran:local}. Let $R$ be sufficiently large in terms of $M'$, $r$, $K$, and the parameters of the Morse-local-to-global property for $S$.

Let $s_1$ be a point in $\beta$ with $\dist(s_1,s_0)>R$, and let $\beta'$ be the subsegment of $\beta$ from $s_0$ to $s_1$. By coboundedness of $\Gamma$, there is a translate $A'$ of $A$ passing within distance $K$ of $s_1$. Because $A$ is a geodesic, only one direction of $A'$ can $K$--fellow-travel $\beta'$ for a distance of more than $2K$. 
Let $s_2$ be a point of $A'$ in the other direction that has $\dist(s_2,s_1)>R$. By coboundedness of $\Gamma$, there exists $h\in\Gamma$ such that $\dist(hs_0,s_2)\le K$. We claim that if $M$ and $R$ were chosen large enough, then $h$ is strongly contracting, but not $D$--strongly contracting.

Let $\gamma_0$ be the following path, which connects $s_0$ to $hs_0$. It begins by following $\beta'$ from $s_0$ to $s_1$. Then it has a segment of length at most $K$ from $s_1$ to $A'$. It then follows $A'$ until it comes to $s_2$. Finally, it has a segment of length at most $K$ connecting $s_2$ to $hs_0$. Let $\gamma$ be the $\langle h\rangle$--invariant path $\gamma=\bigcup_{n\in\Z}h^n\gamma_0$. 

By construction, $\beta'$ and $A$ fellow-travel for a distance of at most $r$, the pair $\beta'$ and $A'$ fellow-travel for a distance of at most $K$, and $A'$ and $h\beta'$ fellow-travel for a distance of at most $r$. The same goes for every $h^n$--translate of these pairs. Since $R$ was chosen to be large compared to $K$ and $r$, it now follows, for instance by \cite[Lem.~4.3]{hagenwise:cubulating:irreducible}, that $\gamma$ is a $q$--quasigeodesic, where $q$ depends only on $K$, the bound on the lengths of the unnamed segments making up $\gamma_0$. In particular, $\gamma$ is an axis for $h$.

Furthermore, all of $h^nA$, $h^n\beta'$, and $h^nA'$ are $M'$--Morse. By the choice of $R$, it now follows from the Morse-local-to-global property that $\gamma$ is an $M$--Morse quasigeodesic. However, $\gamma$ contains an initial subsegment of $\beta$ of length $R$, so $\gamma$ cannot be $M$--Morse. By the choice of $M$, we see that $\gamma$ cannot be $D$--strongly contracting.
\end{proof}

\begin{theorem} \label{thm:quasimorphisms}
Let $S$ be a non-hyperbolic geodesic space with the property that for each $M$ there exists $D$ such that every $M$--Morse geodesic in $S$ is $D$--strongly contracting. Suppose that $S$ contains a biinfinite Morse geodesic. For every group $\Gamma<\isom S$ acting coboundedly on $S$, the space $\widetilde{\QM}(\Gamma)$ is infinite-dimensional.
\end{theorem}

\begin{proof} 
Let $X$ be the hyperbolic core of $S$, with natural map $\pi:S\to X$ and hyperbolicity constant $\delta$. According to Proposition~\ref{prop:subspace_wrg}, there is a constant $q_1$ such that every geodesic in $S$ defines an unparametrised $q$--quasigeodesic in $X$. Let $q_2$ denote the quasigeodesic constant given by Proposition~\ref{prop:bad_contracting_element}, and let $q=\max\{q_1,q_2\}$. Let $B=B(q,\delta)$ be the constant of \cite[p.72]{bestvinafujiwara:bounded}, and let $n$ be a sufficiently large constant, defined in terms of $B$.

By Proposition~\ref{prop:bad_contracting_element}, there is some element $g\in\Gamma$ with a $q$--quasiaxis $\gamma_g$ that is strongly contracting. Let $D$ be such that $\gamma_g$ is $D$--strongly contracting. By that same proposition, there is another element $h\in\Gamma$ with a $q$--quasiaxis $\gamma_h$ that is strongly contracting but not $nD$--strongly contracting. 

By Theorem~\ref{thm:universal_contracting_characterisation}, both $g$ and $h$ act loxodromically on $X$. Let $A_g=\pi\gamma_g$ and $A_h=\pi\gamma_h$ be the projections to $X$, which are uniform-quality quasigeodesics. Given a constant $R$, suppose that there is some $k\in\Gamma$ such that some subsegment $I\subset A_g$ of length at least $R$ lies in the $B$--neighbourhood of $kA_h$ in $X$.

Let $s,t\in S$ be the endpoints of $I$, and let $s',t'\in kA_h$ have $\Dist(s,s'),\Dist(t,t')\le B$. Let $\alpha$ be a geodesic in $S$ from $s$ to $t$, which is uniformly strongly contracting in terms of $D$. Moreover, since $S\to X$ is coarsely Lipschitz, the length of $\alpha$ is lower-bounded in terms of $R$. It follows that there is a chain of curtains dual to $\alpha$ of length lower-bounded in terms of $\frac RD$ that separate $s'$ from $t'$. If $R$ is chosen to be sufficiently large in terms of the $S$--translation-length of $h$, then we can apply Lemma~\ref{lem:ball_graded} to the $khk^{-1}$--translates of that chain to obtain a lower bound on the $X$--translation-length of $h$ that is a uniform multiple of $D$. According to Theorem~\ref{thm:universal_contracting_characterisation}, that gives an upper-bound on the strong-contracting constant of $\gamma_h$ that is a uniform multiple of $D$. The choice of $n$ ensures that this is a contradiction of the fact that $h$ is not $nD$--strongly contracting.

Thus there is an upper bound on the length of subsegments of $A_g$ that can lie in the $B$--neighbourhoods of $\Gamma$--translates of $A_h$. We have shown that $g$ and $h$ are loxodromic elements of $\Gamma$ with $g\not\sim h$. According to Theorem~\ref{thm:bestfuj}, this implies that $\widetilde{\QM}(\Gamma)$ is infinite-dimensional.
\end{proof}
%



\bibliographystyle{alpha}
\bibliography{bibtex}

\newcommand{\etalchar}[1]{$^{#1}$}
\begin{thebibliography}{GHP{\etalchar{+}}23}

\bibitem[AB90]{alexanderbishop:hadamard}
Stephanie~B. Alexander and Richard~L. Bishop.
\newblock The {H}adamard-{C}artan theorem in locally convex metric spaces.
\newblock {\em Enseign. Math. (2)}, 36(3-4):309--320, 1990.

\bibitem[AB95]{alonsobridson:semihyperbolic}
Juan~M. Alonso and Martin~R. Bridson.
\newblock Semihyperbolic groups.
\newblock {\em Proc. London Math. Soc. (3)}, 70(1):56--114, 1995.

\bibitem[ABD21]{abbottbehrstockdurham:largest}
Carolyn Abbott, Jason Behrstock, and Matthew~G. Durham.
\newblock Largest acylindrical actions and stability in hierarchically
  hyperbolic groups.
\newblock {\em Trans. Amer. Math. Soc. Ser. B}, 8:66--104, 2021.
\newblock With an appendix by Daniel Berlyne and Jacob Russell.

\bibitem[ACT15]{arzhantsevacashentao:growth}
Goulnara~N. Arzhantseva, Christopher~H. Cashen, and Jing Tao.
\newblock Growth tight actions.
\newblock {\em Pacific J. Math.}, 278(1):1--49, 2015.

\bibitem[Ago13]{agol:virtual}
Ian Agol.
\newblock The virtual {H}aken conjecture.
\newblock {\em Doc. Math.}, 18:1045--1087, 2013.
\newblock With an appendix by Agol, Daniel Groves, and Jason Manning.

\bibitem[{Alg}11]{algomkfir:strongly}
Yael {Algom-Kfir}.
\newblock Strongly contracting geodesics in outer space.
\newblock {\em Geom. Topol.}, 15(4):2181--2233, 2011.

\bibitem[BBF15]{bestvinabrombergfujiwara:constructing}
Mladen Bestvina, Ken Bromberg, and Koji Fujiwara.
\newblock Constructing group actions on quasi-trees and applications to mapping
  class groups.
\newblock {\em Publ. Math. Inst. Hautes \'{E}tudes Sci.}, 122:1--64, 2015.

\bibitem[BBF21]{bestvinabrombergfujiwara:proper}
Mladen Bestvina, Ken Bromberg, and Koji Fujiwara.
\newblock Proper actions on finite products of quasi-trees.
\newblock {\em Ann. H. Lebesgue}, 4:685--709, 2021.

\bibitem[BBFS20]{bestvinabrombergfujiwarasisto:acylindrical}
Mladen Bestvina, Ken Bromberg, Koji Fujiwara, and Alessandro Sisto.
\newblock Acylindrical actions on projection complexes.
\newblock {\em Enseign. Math.}, 65(1-2):1--32, 2020.

\bibitem[BCK{\etalchar{+}}23]{balasubramanyachesserkerrmangahastrin:non}
Sahana Balasubramanya, Marissa Chesser, Alice Kerr, Johanna Mangahas, and Marie
  Trin.
\newblock (non-)recognizing spaces for stable subgroups.
\newblock {\em arXiv:2311.15187}, 2023.

\bibitem[BF92]{bestvinafeighn:combination}
M.~Bestvina and M.~Feighn.
\newblock A combination theorem for negatively curved groups.
\newblock {\em J. Differential Geom.}, 35(1):85--101, 1992.

\bibitem[BF02]{bestvinafujiwara:bounded}
Mladen Bestvina and Koji Fujiwara.
\newblock Bounded cohomology of subgroups of mapping class groups.
\newblock {\em Geom. Topol.}, 6:69--89, 2002.

\bibitem[BF14]{bestvinafeighn:hyperbolicity}
Mladen Bestvina and Mark Feighn.
\newblock Hyperbolicity of the complex of free factors.
\newblock {\em Adv. Math.}, 256:104--155, 2014.

\bibitem[BH83]{bandelthedlikova:median}
Hans-J. Bandelt and Jarmila Hedl\'{\i}kov\'{a}.
\newblock Median algebras.
\newblock {\em Discrete Math.}, 45(1):1--30, 1983.

\bibitem[BHS17]{behrstockhagensisto:asymptotic}
Jason Behrstock, Mark Hagen, and Alessandro Sisto.
\newblock Asymptotic dimension and small-cancellation for hierarchically
  hyperbolic spaces and groups.
\newblock {\em Proc. Lond. Math. Soc. (3)}, 114(5):890--926, 2017.

\bibitem[BHS19]{behrstockhagensisto:hierarchically:2}
Jason Behrstock, Mark Hagen, and Alessandro Sisto.
\newblock Hierarchically hyperbolic spaces~{II}: {C}ombination theorems and the
  distance formula.
\newblock {\em Pacific J. Math.}, 299(2):257--338, 2019.

\bibitem[BHS21]{behrstockhagensisto:quasiflats}
Jason Behrstock, Mark Hagen, and Alessandro Sisto.
\newblock Quasiflats in hierarchically hyperbolic spaces.
\newblock {\em Duke Math. J.}, 170(5):909--996, 2021.

\bibitem[Bj{\"{o}}10]{bjorklund:central}
Michael Bj{\"{o}}rklund.
\newblock Central limit theorems for {G}romov hyperbolic groups.
\newblock {\em J. Theoret. Probab.}, 23(3):871--887, 2010.

\bibitem[BM11]{behrstockminsky:centroids}
Jason Behrstock and Yair~N. Minsky.
\newblock Centroids and the rapid decay property in mapping class groups.
\newblock {\em J. Lond. Math. Soc. (2)}, 84(3):765--784, 2011.

\bibitem[Bow98]{bowditch:topological}
Brian~H. Bowditch.
\newblock A topological characterisation of hyperbolic groups.
\newblock {\em J. Amer. Math. Soc.}, 11(3):643--667, 1998.

\bibitem[Bow08]{bowditch:tight}
Brian~H. Bowditch.
\newblock Tight geodesics in the curve complex.
\newblock {\em Invent. Math.}, 171(2):281--300, 2008.

\bibitem[Bow13a]{bowditch:coarse}
Brian~H. Bowditch.
\newblock Coarse median spaces and groups.
\newblock {\em Pacific J. Math.}, 261(1):53--93, 2013.

\bibitem[Bow13b]{bowditch:invariance}
Brian~H. Bowditch.
\newblock Invariance of coarse median spaces under relative hyperbolicity.
\newblock {\em Math. Proc. Cambridge Philos. Soc.}, 154(1):85--95, 2013.

\bibitem[Bow18a]{bowditch:convex}
Brian~H. Bowditch.
\newblock Convex hulls in coarse median spaces.
\newblock {\em Preprint available at
  \mbox{homepages.warwick.ac.uk/\texttildelow masgak/papers/hulls-cms.pdf}},
  2018.

\bibitem[Bow18b]{bowditch:large:mapping}
Brian~H. Bowditch.
\newblock Large-scale rigidity properties of the mapping class groups.
\newblock {\em Pacific J. Math.}, 293(1):1--73, 2018.

\bibitem[Bow19]{bowditch:quasiflats}
Brian~H. Bowditch.
\newblock Quasiflats in coarse median spaces.
\newblock {\em Preprint available at
  \mbox{homepages.warwick.ac.uk/\texttildelow masgak/papers/quasiflats.pdf}},
  2019.

\bibitem[Bow21]{bowditch:cartan--hadamard}
Brian~H. Bowditch.
\newblock A {C}artan--{H}adamard theorem for median metric spaces.
\newblock {\em Preprint available at
  \mbox{homepages.warwick.ac.uk/\texttildelow masgak/papers/ch-median.pdf}},
  2021.

\bibitem[Bow22]{bowditch:median:book}
Brian~H. Bowditch.
\newblock Median algebras.
\newblock {\em Preprint available at
  \mbox{homepages.warwick.ac.uk/\texttildelow
  masgak/papers/median-algebras.pdf}}, 2022.

\bibitem[BQ16a]{benoistquint:central:linear}
Yves Benoist and Jean-Fran\c{c}ois Quint.
\newblock Central limit theorem for linear groups.
\newblock {\em Ann. Probab.}, 44(2):1308--1340, 2016.

\bibitem[BQ16b]{benoistquint:central:hyperbolic}
Yves Benoist and Jean-Fran\c{c}ois Quint.
\newblock Central limit theorem on hyperbolic groups.
\newblock {\em Izv. Ross. Akad. Nauk Ser. Mat.}, 80(1):5--26, 2016.

\bibitem[BV91]{bandeltvandevel:superextensions}
H.-J. Bandelt and M.~van~de Vel.
\newblock Superextensions and the depth of median graphs.
\newblock {\em J. Combin. Theory Ser. A}, 57(2):187--202, 1991.

\bibitem[BW12]{bergeronwise:boundary}
Nicolas Bergeron and Daniel~T. Wise.
\newblock A boundary criterion for cubulation.
\newblock {\em Amer. J. Math.}, 134(3):843--859, 2012.

\bibitem[Can87]{cannon:almost}
James~W. Cannon.
\newblock Almost convex groups.
\newblock {\em Geom. Dedicata}, 22(2):197--210, 1987.

\bibitem[CD95]{charneydavis:strict}
Ruth~M. Charney and Michael~W. Davis.
\newblock Strict hyperbolization.
\newblock {\em Topology}, 34(2):329--350, 1995.

\bibitem[CD19]{cordesdurham:boundary}
Matthew Cordes and Matthew~G. Durham.
\newblock Boundary convex cocompactness and stability of subgroups of finitely
  generated groups.
\newblock {\em Int. Math. Res. Not. IMRN}, 6:1699--1724, 2019.

\bibitem[CDH10]{chatterjidrutuhaglund:kazhdan}
Indira Chatterji, Cornelia Dru\c{t}u, and Fr\'{e}d\'{e}ric Haglund.
\newblock Kazhdan and {H}aagerup properties from the median viewpoint.
\newblock {\em Adv. Math.}, 225(2):882--921, 2010.

\bibitem[CF10]{capracefujiwara:rank}
Pierre-Emmanuel Caprace and Koji Fujiwara.
\newblock Rank-one isometries of buildings and quasi-morphisms of {K}ac-{M}oody
  groups.
\newblock {\em Geom. Funct. Anal.}, 19(5):1296--1319, 2010.

\bibitem[CFFT22]{chawlaforghanifrischtiozzo:poisson}
Kunal Chawla, Behrang Forghani, Joshua Frisch, and Giulio Tiozzo.
\newblock The {P}oisson boundary of hyperbolic groups without moment
  conditions.
\newblock {\em arXiv:2209.02114}, 2022.

\bibitem[Che00]{chepoi:graphs}
Victor Chepoi.
\newblock Graphs of some {${\rm CAT}(0)$} complexes.
\newblock {\em Adv. in Appl. Math.}, 24(2):125--179, 2000.

\bibitem[CMV04]{cherixmartinvalette:spaces}
Pierre-Alain Cherix, Florian Martin, and Alain Valette.
\newblock Spaces with measured walls, the {H}aagerup property and property
  ({T}).
\newblock {\em Ergodic Theory Dynam. Systems}, 24(6):1895--1908, 2004.

\bibitem[CN05]{chatterjiniblo:from}
Indira Chatterji and Graham Niblo.
\newblock From wall spaces to {$\rm CAT(0)$} cube complexes.
\newblock {\em Internat. J. Algebra Comput.}, 15(5-6):875--885, 2005.

\bibitem[Cou23]{coulon:ergodicity}
R{\'e}mi Coulon.
\newblock Ergodicity of the geodesic flow for groups with a contracting
  element.
\newblock {\em arXiv:2303.01390}, 2023.

\bibitem[CS11]{capracesageev:rank}
Pierre-Emmanuel Caprace and Michah Sageev.
\newblock Rank rigidity for {${\rm CAT}(0)$} cube complexes.
\newblock {\em Geom. Funct. Anal.}, 21(4):851--891, 2011.

\bibitem[CS15]{charneysultan:contracting}
Ruth Charney and Harold Sultan.
\newblock Contracting boundaries of {$\rm CAT(0)$} spaces.
\newblock {\em J. Topol.}, 8(1):93--117, 2015.

\bibitem[CW21a]{calvezwiest:hyperbolic}
Matthieu Calvez and Bert Wiest.
\newblock Hyperbolic structures for {A}rtin-{T}its groups of spherical type.
\newblock In {\em Geometry at the frontier---symmetries and moduli spaces of
  algebraic varieties}, volume 766 of {\em Contemp. Math.}, pages 83--98. Amer.
  Math. Soc., Providence, RI, 2021.

\bibitem[CW21b]{calvezwiest:morse}
Matthieu Calvez and Bert Wiest.
\newblock Morse elements in {G}arside groups are strongly contracting.
\newblock {\em arXiv:2106.14826}, 2021.

\bibitem[Dah03]{dahmani:combination}
Fran\c{c}ois Dahmani.
\newblock Combination of convergence groups.
\newblock {\em Geom. Topol.}, 7:933--963, 2003.

\bibitem[DGO17]{dahmaniguirardelosin:hyperbolically}
F.~Dahmani, V.~Guirardel, and D.~Osin.
\newblock Hyperbolically embedded subgroups and rotating families in groups
  acting on hyperbolic spaces.
\newblock {\em Mem. Amer. Math. Soc.}, 245(1156):v+152, 2017.

\bibitem[DHS17]{durhamhagensisto:boundaries}
Matthew~G. Durham, Mark Hagen, and Alessandro Sisto.
\newblock Boundaries and automorphisms of hierarchically hyperbolic spaces.
\newblock {\em Geom. Topol.}, 21(6):3659--3758, 2017.

\bibitem[DJ91]{davisjanuszkiewicz:hyperbolization}
Michael~W. Davis and Tadeusz Januszkiewicz.
\newblock Hyperbolization of polyhedra.
\newblock {\em J. Differential Geom.}, 34(2):347--388, 1991.

\bibitem[DMS23]{durhamminskysisto:stable}
Matthew~G. Durham, Yair~N. Minsky, and Alessandro Sisto.
\newblock Stable cubulations, bicombings, and barycenters.
\newblock {\em Geom. Topol.}, 27(6):2383--2478, 2023.

\bibitem[DT15]{durhamtaylor:convex}
Matthew~G. Durham and Samuel~J. Taylor.
\newblock Convex cocompactness and stability in mapping class groups.
\newblock {\em Algebr. Geom. Topol.}, 15(5):2839--2859, 2015.

\bibitem[DT19]{dowdalltaylor:contracting}
Spencer Dowdall and Samuel~J. Taylor.
\newblock Contracting orbits in outer space.
\newblock {\em Math. Z.}, 293(1-2):767--787, 2019.

\bibitem[Dur16]{durham:augmented}
Matthew~Gentry Durham.
\newblock The augmented marking complex of a surface.
\newblock {\em J. Lond. Math. Soc. (2)}, 94(3):933--969, 2016.

\bibitem[Dur23]{durham:cubulating}
Matthew~Gentry Durham.
\newblock Cubulating infinity in hierarchically hyperbolic spaces.
\newblock {\em arXiv:2308.13689}, 2023.

\bibitem[DZ22]{durhamzalloum:geometry}
Matthew~G. Durham and Abdul Zalloum.
\newblock The geometry of genericity in mapping class groups and
  {T}eichm\"uller spaces via {${\rm CAT}(0)$} cube complexes.
\newblock {\em arXiv:2207.06516}, 2022.

\bibitem[EMR17]{eskinmasurrafi:large}
Alex Eskin, Howard Masur, and Kasra Rafi.
\newblock Large-scale rank of {T}eichm\"{u}ller space.
\newblock {\em Duke Math. J.}, 166(8):1517--1572, 2017.

\bibitem[Far06]{farb:some}
Benson Farb.
\newblock Some problems on mapping class groups and moduli space.
\newblock In {\em Problems on mapping class groups and related topics},
  volume~74 of {\em Proc. Sympos. Pure Math.}, pages 11--55. Amer. Math. Soc.,
  Providence, RI, 2006.

\bibitem[Fio20]{fioravanti:roller}
Elia Fioravanti.
\newblock Roller boundaries for median spaces and algebras.
\newblock {\em Algebr. Geom. Topol.}, 20(3):1325--1370, 2020.

\bibitem[FLM01]{farblubotzkyminsky:rank}
Benson Farb, Alexander Lubotzky, and Yair Minsky.
\newblock Rank-1 phenomena for mapping class groups.
\newblock {\em Duke Math. J.}, 106(3):581--597, 2001.

\bibitem[FLM24]{fernoslecureuxmatheus:contact}
Talia Fern\'{o}s, Jean L\'{e}cureux, and Fr\'{e}d\'{e}ric Math\'{e}us.
\newblock Contact graphs, boundaries, and a central limit theorem for
  {$\rm{CAT}(0)$} cubical complexes.
\newblock {\em Groups Geom. Dyn.}, 18(2):677--704, 2024.

\bibitem[FM02]{farbmosher:convex}
Benson Farb and Lee Mosher.
\newblock Convex cocompact subgroups of mapping class groups.
\newblock {\em Geom. Topol.}, 6:91--152, 2002.

\bibitem[Gen20]{genevois:hyperbolicities}
Anthony Genevois.
\newblock Hyperbolicities in {$\rm CAT(0)$} cube complexes.
\newblock {\em Enseign. Math.}, 65(1-2):33--100, 2020.

\bibitem[Gen23]{genevois:algebraic}
Anthony Genevois.
\newblock Algebraic properties of groups acting on median graphs.
\newblock {\em Preprint available at
  \mbox{sites.google.com/view/agenevois/books}}, 2023.

\bibitem[GHP{\etalchar{+}}23]{goldsboroughhagenpetytsisto:induced}
Antoine Goldsborough, Mark Hagen, Harry Petyt, Jacob Russell, and Alessandro
  Sisto.
\newblock Induced quasi-isometries of hyperbolic spaces, {M}arkov chains, and
  acylindrical hyperbolicity.
\newblock {\em arXiv:2309.07013}, 2023.

\bibitem[Gro87]{gromov:hyperbolic}
M.~Gromov.
\newblock Hyperbolic groups.
\newblock In {\em Essays in group theory}, volume~8 of {\em Math. Sci. Res.
  Inst. Publ.}, pages 75--263. Springer, New York, 1987.

\bibitem[GY22]{gekhtmanyang:counting}
Ilya Gekhtman and Wen-yuan Yang.
\newblock Counting conjugacy classes in groups with contracting elements.
\newblock {\em J. Topol.}, 15(2):620--665, 2022.

\bibitem[Hae21]{haettel:lattices}
Thomas Haettel.
\newblock Lattices, injective metrics and the $k(\pi,1)$ conjecture.
\newblock {\em arXiv:2109.07891}, 2021.

\bibitem[Hae22]{haettel:link}
Thomas Haettel.
\newblock A link condition for simplicial complexes, and cub spaces.
\newblock {\em arXiv:2211.07857}, 2022.

\bibitem[Hae23]{haettel:group}
Thomas Haettel.
\newblock Group actions on injective spaces and {H}elly graphs.
\newblock {\em arXiv:2307.00414}, 2023.

\bibitem[HHP23]{haettelhodapetyt:coarse}
Thomas Haettel, Nima Hoda, and Harry Petyt.
\newblock Coarse injectivity, hierarchical hyperbolicity, and
  semihyperbolicity.
\newblock {\em Geom. Topol.}, 27(4):1587--1633, 2023.

\bibitem[HM13]{handelmosher:free:1}
Michael Handel and Lee Mosher.
\newblock The free splitting complex of a free group, {I}: hyperbolicity.
\newblock {\em Geom. Topol.}, 17(3):1581--1672, 2013.

\bibitem[HMS24]{hagenmartinsisto:extra}
Mark Hagen, Alexandre Martin, and Alessandro Sisto.
\newblock Extra-large type {A}rtin groups are hierarchically hyperbolic.
\newblock {\em Math. Ann.}, 388(1):867--938, 2024.

\bibitem[HO21]{huangosajda:helly}
Jingyin Huang and Damian Osajda.
\newblock Helly meets {G}arside and {A}rtin.
\newblock {\em Invent. Math.}, 225(2):395--426, 2021.

\bibitem[Hor18]{horbez:central}
Camille Horbez.
\newblock Central limit theorems for mapping class groups and {${\rm
  Out}(F_N)$}.
\newblock {\em Geom. Topol.}, 22(1):105--156, 2018.

\bibitem[HV98]{hatchervogtmann:complex}
Allen Hatcher and Karen Vogtmann.
\newblock The complex of free factors of a free group.
\newblock {\em Quart. J. Math. Oxford Ser. (2)}, 49(196):459--468, 1998.

\bibitem[HW15a]{hagenwise:cubulating:general}
Mark Hagen and Daniel~T. Wise.
\newblock Cubulating hyperbolic free-by-cyclic groups: the general case.
\newblock {\em Geom. Funct. Anal.}, 25(1):134--179, 2015.

\bibitem[HW15b]{hsuwise:cubulating}
Tim Hsu and Daniel~T. Wise.
\newblock Cubulating malnormal amalgams.
\newblock {\em Invent. Math.}, 199(2):293--331, 2015.

\bibitem[HW16]{hagenwise:cubulating:irreducible}
Mark Hagen and Daniel~T. Wise.
\newblock Cubulating hyperbolic free-by-cyclic groups: the irreducible case.
\newblock {\em Duke Math. J.}, 165(9):1753--1813, 2016.

\bibitem[IPS21]{iozzipagliantinisisto:characterising}
Alessandra Iozzi, Cristina Pagliantini, and Alessandro Sisto.
\newblock Characterising actions on trees yielding non-trivial quasimorphisms.
\newblock {\em Ann. Math. Qu\'{e}.}, 45(1):185--202, 2021.

\bibitem[JL22]{jorgensenlang:combinatorial}
Martina J{\o}rgensen and Urs Lang.
\newblock A combinatorial higher-rank hyperbolicity condition.
\newblock {\em arXiv:2206.08153}, 2022.

\bibitem[Kai00]{kaimanovich:poissonformula}
Vadim~A. Kaimanovich.
\newblock The {P}oisson formula for groups with hyperbolic properties.
\newblock {\em Ann. of Math. (2)}, 152(3):659--692, 2000.

\bibitem[KK13a]{kimkoberda:embedability}
Sang-Hyun Kim and Thomas Koberda.
\newblock Embedability between right-angled {A}rtin groups.
\newblock {\em Geom. Topol.}, 17(1):493--530, 2013.

\bibitem[KK13b]{kotowskikotowski:random}
Marcin Kotowski and Micha{\l} Kotowski.
\newblock Random groups and property {$(T)$}: \.{Z}uk's theorem revisited.
\newblock {\em J. Lond. Math. Soc. (2)}, 88(2):396--416, 2013.

\bibitem[KL08]{kentleininger:shadows}
Autumn~E. Kent and Christopher~J. Leininger.
\newblock Shadows of mapping class groups: capturing convex cocompactness.
\newblock {\em Geom. Funct. Anal.}, 18(4):1270--1325, 2008.

\bibitem[KL09]{kapovichlustig:geometric}
Ilya Kapovich and Martin Lustig.
\newblock Geometric intersection number and analogues of the curve complex for
  free groups.
\newblock {\em Geom. Topol.}, 13(3):1805--1833, 2009.

\bibitem[KMV22]{keppelermollervarghese:automatic}
Daniel Keppeler, Philip M\"{o}ller, and Olga Varghese.
\newblock Automatic continuity for groups whose torsion subgroups are small.
\newblock {\em J. Group Theory}, 25(6):1017--1043, 2022.

\bibitem[Lan13]{lang:injective}
Urs Lang.
\newblock Injective hulls of certain discrete metric spaces and groups.
\newblock {\em J. Topol. Anal.}, 5(3):297--331, 2013.

\bibitem[{Le }22]{lebars:central}
Corentin {Le Bars}.
\newblock Central limit theorem on {${\rm CAT}(0)$} spaceswith contracting
  isometries.
\newblock {\em arXiv:2209.11648}, 2022.

\bibitem[{Le }23]{lebars:marches}
Corentin {Le Bars}.
\newblock Marches aléatoires et éléments contractants sur des espaces {${\rm
  CAT}(0)$}.
\newblock {\em arXiv:2310.18710}, 2023.

\bibitem[Lea13]{leary:metric}
Ian~J. Leary.
\newblock A metric {K}an-{T}hurston theorem.
\newblock {\em J. Topol.}, 6(1):251--284, 2013.

\bibitem[Mie14]{miesch:injective}
Benjamin Miesch.
\newblock Injective metrics on cube complexes.
\newblock {\em arXiv:1411.7234}, 2014.

\bibitem[MM99]{masurminsky:geometry:1}
Howard~A. Masur and Yair~N. Minsky.
\newblock Geometry of the complex of curves.~{I}. {H}yperbolicity.
\newblock {\em Invent. Math.}, 138(1):103--149, 1999.

\bibitem[MM00]{masurminsky:geometry:2}
Howard~A. Masur and Yair~N. Minsky.
\newblock Geometry of the complex of curves.~{II}. {H}ierarchical structure.
\newblock {\em Geom. Funct. Anal.}, 10(4):902--974, 2000.

\bibitem[{Mor}21]{morriswright:parabolic}
Rose {Morris-Wright}.
\newblock Parabolic subgroups in {FC}-type {A}rtin groups.
\newblock {\em J. Pure Appl. Algebra}, 225(1):1--13, 2021.

\bibitem[MP22]{martinprzytycki:acylindrical}
Alexandre Martin and Piotr Przytycki.
\newblock Acylindrical actions for two-dimensional {A}rtin groups of hyperbolic
  type.
\newblock {\em Int. Math. Res. Not. IMRN}, 17:13099--13127, 2022.

\bibitem[MR08]{mjreeves:combination}
Mahan Mj and Lawrence Reeves.
\newblock A combination theorem for strong relative hyperbolicity.
\newblock {\em Geom. Topol.}, 12(3):1777--1798, 2008.

\bibitem[MT91]{maitang:oninjective}
Jie~Hua Mai and Yun Tang.
\newblock On the injective metrization for infinite collapsible polyhedra.
\newblock {\em Acta Math. Sinica (N.S.)}, 7(1):13--18, 1991.

\bibitem[MT18]{mahertiozzo:random}
Joseph Maher and Giulio Tiozzo.
\newblock Random walks on weakly hyperbolic groups.
\newblock {\em J. Reine Angew. Math.}, 742:187--239, 2018.

\bibitem[Mul78]{mulder:structure}
Martyn Mulder.
\newblock The structure of median graphs.
\newblock {\em Discrete Math.}, 24(2):197--204, 1978.

\bibitem[Nic04]{nica:cubulating}
Bogdan Nica.
\newblock Cubulating spaces with walls.
\newblock {\em Algebr. Geom. Topol.}, 4:297--309, 2004.

\bibitem[NR97]{nibloreeves:groups}
Graham~A. Niblo and Lawrence Reeves.
\newblock Groups acting on {${\rm CAT}(0)$} cube complexes.
\newblock {\em Geom. Topol.}, 1:1--7, 1997.

\bibitem[NR98]{nibloreeves:geometry}
G.~A. Niblo and L.~D. Reeves.
\newblock The geometry of cube complexes and the complexity of their
  fundamental groups.
\newblock {\em Topology}, 37(3):621--633, 1998.

\bibitem[NWZ19]{niblowrightzhang:four}
Graham~A. Niblo, Nick Wright, and Jiawen Zhang.
\newblock A four point characterisation for coarse median spaces.
\newblock {\em Groups Geom. Dyn.}, 13(3):939--980, 2019.

\bibitem[Ont20]{ontaneda:riemannian}
Pedro Ontaneda.
\newblock Riemannian hyperbolization.
\newblock {\em Publ. Math. Inst. Hautes \'{E}tudes Sci.}, 131:1--72, 2020.

\bibitem[Osi16]{osin:acylindrically}
D.~Osin.
\newblock Acylindrically hyperbolic groups.
\newblock {\em Trans. Amer. Math. Soc.}, 368(2):851--888, 2016.

\bibitem[Pap22]{papin:strongly}
Chlo{\'e} Papin.
\newblock Strongly contracting axes for fully irreducible automorphisms of
  generalized {B}aumslag--{S}olitar groups.
\newblock {\em arXiv:2206.04010}, 2022.

\bibitem[Pet21]{petyt:mapping}
Harry Petyt.
\newblock Mapping class groups are quasicubical.
\newblock {\em arXiv:2112.10681, to appear in Amer. J. Math.}, 2021.

\bibitem[PS23]{petytspriano:unbounded}
Harry Petyt and Davide Spriano.
\newblock Unbounded domains in hierarchically hyperbolic groups.
\newblock {\em Groups Geom. Dyn.}, 17(2):479--500, 2023.

\bibitem[PSZ24]{petytsprianozalloum:hyperbolic}
Harry Petyt, Davide Spriano, and Abdul Zalloum.
\newblock Hyperbolic models for {${\rm CAT}(0)$} spaces.
\newblock {\em Adv. Math.}, 450(109742):1--66, 2024.

\bibitem[QRT20]{qingrafitiozzo:sublinearly:2}
Yulan Qing, Kasra Rafi, and Giulio Tiozzo.
\newblock Sublinearly morse boundary {II}: Proper geodesic spaces.
\newblock {\em arXiv:2011.03481}, 2020.

\bibitem[Rol98]{roller:poc}
Martin Roller.
\newblock Poc sets, median algebras and group actions.
\newblock {\em Habilitationsschrift, Universit\"at Regensburg}, 1998.
\newblock Available on the arXiv at \mbox{arXiv:1607.07747}.

\bibitem[RST22]{russellsprianotran:local}
Jacob Russell, Davide Spriano, and Hung~Cong Tran.
\newblock The local-to-global property for {M}orse quasi-geodesics.
\newblock {\em Math. Z.}, 300(2):1557--1602, 2022.

\bibitem[RV21]{rafiverberne:geodesics}
Kasra Rafi and Yvon Verberne.
\newblock Geodesics in the mapping class group.
\newblock {\em Algebr. Geom. Topol.}, 21(6):2995--3017, 2021.

\bibitem[Sag95]{sageev:ends}
Michah Sageev.
\newblock Ends of group pairs and non-positively curved cube complexes.
\newblock {\em Proc. London Math. Soc. (3)}, 71(3):585--617, 1995.

\bibitem[Sel97]{sela:acylindrical}
Z.~Sela.
\newblock Acylindrical accessibility for groups.
\newblock {\em Invent. Math.}, 129(3):527--565, 1997.

\bibitem[She22]{shepherd:cubulation}
Sam Shepherd.
\newblock A cubulation with no factor system.
\newblock {\em arXiv:2208.10421}, 2022.

\bibitem[Sis18]{sisto:contracting}
Alessandro Sisto.
\newblock Contracting elements and random walks.
\newblock {\em J. Reine Angew. Math.}, 742:79--114, 2018.

\bibitem[SZ]{soergelzalloum:morse}
Mireille Soergel and Abdul Zalloum.
\newblock Morse geodesics in weakly modular graphs.
\newblock {\em In preparation}.

\bibitem[SZ22]{sistozalloum:morse}
Alessandro Sisto and Abdul Zalloum.
\newblock Morse subsets of injective spaces are strongly contracting.
\newblock {\em arXiv:2208.13859}, 2022.

\bibitem[Wis21]{wise:structure}
Daniel~T. Wise.
\newblock {\em The structure of groups with a quasiconvex hierarchy}, volume
  209 of {\em Annals of Mathematics Studies}.
\newblock Princeton University Press, Princeton, NJ, 2021.

\bibitem[\.{Z}03]{zuk:property}
A.~\.{Z}uk.
\newblock Property ({T}) and {K}azhdan constants for discrete groups.
\newblock {\em Geom. Funct. Anal.}, 13(3):643--670, 2003.

\bibitem[Zal23]{zalloum:injectivity}
Abdul Zalloum.
\newblock Injectivity, cubical approximations and equivariant wall structures
  beyond {$\rm CAT(0)$} cube complexes.
\newblock {\em arXiv:2305.02951}, 2023.

\bibitem[Zbi23]{zbinden:small}
Stefanie Zbinden.
\newblock Small cancellation groups with and without sigma-compact {M}orse
  boundary.
\newblock {\em arXiv:2307.13325}, 2023.

\bibitem[Zbi24]{zbinden:hyperbolic}
Stefanie Zbinden.
\newblock Hyperbolic spaces that detect all strongly-contracting directions.
\newblock {\em arXiv:2404.12162}, 2024.

\end{thebibliography}

\end{document}